\documentclass{amsart}
\usepackage{amsmath,amssymb}
\usepackage{graphicx}
\usepackage{comment}
\usepackage{enumerate}
\usepackage{mathrsfs}
\usepackage{afterpage}
\usepackage[T1,T5]{fontenc}
\usepackage[colorlinks,citecolor=red,urlcolor=blue,bookmarks=false,hypertexnames=true]{hyperref} 
\usepackage{xcolor}
\usepackage{multicol}
\usepackage[shortalphabetic]{amsrefs}

\usepackage{etoolbox}

\usepackage{tikz-cd}
\tikzcdset{scale cd/.style={every label/.append style={scale=#1},
		cells={nodes={scale=#1}}}}

\usepackage{tikz}
\definecolor{light-light-gray}{rgb}{0.9, 0.9, 0.88}

 \definecolor{light-gray}{rgb}{0.52, 0.52, 0.51}
\definecolor{sand}{RGB}{221, 204, 119} 
\definecolor{cyan}{RGB}{136, 204, 238}
\definecolor{rose}{RGB}{204, 102, 119} 
\definecolor{green}{RGB}{34, 156, 31}

\definecolor{orange} {rgb}{1.0, 0.6, 0.4}

\definecolor{violet}{rgb}{0.75, 0.58, 0.89} 

\usepackage{tkz-graph}
\usepackage{multicol}
\usepackage{comment}
\DeclareMathOperator{\supp}{supp}

\tikzset{smalltext/.style={"\textup{\small #1}" description}}
\usetikzlibrary{positioning}
\usepackage{tkz-euclide,amsmath} 

\usetikzlibrary{arrows.meta, bending}

\tikzset{hidden/.style = {thick, dashed}}

\definecolor{dark-gray}{rgb}{0.6, 0.6, 0.58}

\definecolor{light-cyan}      {RGB}{  100,   150,   180}

\definecolor{orange} {rgb}{1.0, 0.6, 0.4}

\definecolor{violet}{rgb}{0.75, 0.58, 0.89}

\usepackage[all,cmtip]{xy}

\def\R{{\mathbb R}}
\def\Z{{\mathbb Z}}

\DeclareMathOperator{\Borel}{Borel}
\DeclareMathOperator{\Tor}{Tor}

\newcommand{\GL}{{\rm GL}}

\newcommand{\lcm}{{\rm lcm}}

\theoremstyle{plain}

\newtheorem{theorem}{Theorem}[section]
\newtheorem{lemma}[theorem]{Lemma}
\newtheorem{corollary}[theorem]{Corollary}
\newtheorem{proposition}[theorem]{Proposition}

\theoremstyle{definition}
\newtheorem{definition}[theorem]{Definition}

\newtheorem{example}[theorem]{Example}
\newtheorem{remark}[theorem]{Remark}

\newtheorem{construction}[theorem]{Construction}
\newtheorem{question}[theorem]{Question}
\newtheorem{notation}[theorem]{Notation}

\title[]{Resolutions of pinched power ideals}

\author{{\DJ}\`ao, Ho\`ai} 
\address[Department of Mathematics]{Virginia Tech}
\email{hoaidao@vt.edu}

\author{ Mermin, Jeffrey}
\address[Department of Mathematics]{Oklahoma State University}
\email{mermin@math.okstate.edu}

\date{}
\thanks{We thank Hailong Dao, Chris Francisco, Craig Huneke, Susan Morey, Uwe Nagel, Irena Peeva, Jay Schweig, and Irena Swanson for helpful conversations.  Many of these conversations took place during workshops at the Simons Laufer Institute, which also provided us with travel support.}

\makeatletter
\def\blankfootnote{\xdef\@thefnmark{}\@footnotetext}
\makeatother

\begin{document}
	
\begin{abstract} 
    In this paper, we construct resolutions of ideals obtained by removing a small number of generators from the generators of $(x_1,\dots,x_n)^d$. 
\end{abstract}
\blankfootnote{\emph{2020 Math Subject Classification:}13D02, 13F55, 05E40}
\blankfootnote{\emph{Keywords and Phrases:} Minimal resolutions, polytopal resolutions, Borel ideals}

	 \maketitle

\section{Introduction}
\phantomsection\label{sec:intro}

Let $S=k[x_1,\dots,x_n]$, $I=\left(x_1,\dots,x_n\right)^d$, and $\widehat{I}=I \smallsetminus \{x_1^{d_1} \cdots x_n^{d_n}\}$ the ideal obtained by removing $m=x_1^{d_1} \cdots x_n^{d_n}$ from the minimal generators of $I$.  In this paper, we describe a minimal polytopal resolution of $I$, based on the complex-of-boxes resolution (see \cite{NagelReiner}), which has cyclic symmetry about $m$, and then use this complex to construct a minimal polytopal resolution for $\widehat{I}$.

Understanding the minimal resolutions of monomial ideals has been a major driver of research in commutative algebra for more than fifty years.  The free resolution contains all homological information about the ideal, including its regularity and Hilbert function.  Unfortunately, the problem of describing a minimal free resolution, even for monomial ideals, without requiring prohibitive amounts of computation, seems to be intractable in general.  

The most substantial case with a computationally easy solution is that of Borel-fixed ideals (which include powers of the maximal ideal), where both the Eliahou-Kervaire resolution (\cite{EK,PS}) and the complex-of-boxes construction (\cite{NagelReiner}) allow one to read off, from the list of monomial generators, both a basis for the minimal resolution and 
the formulas for the differential maps.  The two constructions are different, meaning that they produce different choices of basis for the same resolution.  These bases are fundamentally non-canonical; in the case of a power of the maximal ideal they also rely on an ordering of the variables and so fail to maintain the symmetries of the ideal.

It is, in principle, possible to construct canonical explicit symmetry-respecting resolutions by refusing to choose a basis.  In the case of a power of the maximal ideal, this goes back at least as far as \cite{BuchsbaumEisenbud}, where $I$ is realized as a special case of a generic determinantal ideal; the resulting construction relies on considerable functorial wizardry and has full $\GL_{n}$ symmetry.  More recently, sylvan resolutions (\cite{EagonMillerOrdog}) and dynamical systems resolutions (\cite{Tchernev}) provide explicit descriptions of the minimal resolutions of arbitrary monomial ideals; Tchernev notes that dynamical systems resolutions respect all symmetries of the underlying ideal, and we suspect that the same is true of sylvan resolutions.  Both these constructions require a homology computation at every node of the LCM lattice, which is prohibitive in general.  

In recent decades, an important trend in the study of resolutions has been to identify the resolution with some topological or combinatorial object which supports its structure, such as a simplicial complex or a poset.  (See \cites{BS,BPS,GasharovPeevaWelker,Novik,mermin2011simplicial,Clark,Mapes,ClarkMapes,ClarkTchernev,TchernevVarisco}.)  While Velasco shows that even the class of CW complexes is not large enough to describe all monomial resolutions, these techniques have been very useful where they work, allowing us to read off the Betti numbers, as well as bases and the corresponding differential matrices, from the combinatorics of the underlying object.  Our results are very much in this tradition; the resolutions of $I$ and $\widehat{I}$ we construct are supported on a polytopal complex which has cyclic symmetry about the monomial $m$.

The paper is organized as follows. In section \ref{sec:background}, we introduce some notation and describe the complex-of boxes resolution.
In section \ref{sec3}, we establish some technical results that are necessary for the rest of the paper.   
In section \ref{sec4}, we construct symmetric minimal resolutions of $I=(x_{1},\dots, x_{n})^{n}$ and  $\widehat{I}=(x_{1},\dots, x_{n})\smallsetminus \{x_{1}\dots x_{n}\}$.  These constructions are special cases of those found in the next two sections, so we focus on the geometric intuition, suppressing the proofs completely.    
In section \ref{sec5}, we construct a polytopal minimal resolution of a power of the maximal ideal which is cyclically symmetric around a monomial $m$.
In section \ref{sec6},   we construct a  minimal resolution  of $\widehat{I}  = I \smallsetminus \{x_{i_1}^{d_{i_1}} \cdots x_{i_s}^{d_{i_s}}\}$.
In section \ref{sec7}, we use the mapping cone exact sequence to derive the graded Betti numbers of $\widehat{I}$ from those of $I$, and provide explicit formulas.  We observe that this is a highly unusual use of the mapping cone, as the deduction almost always goes in the other direction.
Finally, in section \ref{sec8}, we discuss the question of when our techniques can work with more than one deleted generator, and provide some partial answers in terms of staircase diagrams.

\section{Background and Notation}
\phantomsection\label{sec:background}

Throughout the paper, $S=k[x_1,\dots,x_n]$  is a polynomial ring over an arbitrary field $k$.   
 
    \subsection{Free resolutions}
	\begin{definition}
	\phantomsection \label{2.1}
	Let $I$ be an ideal of $S$. A {\it free resolution} of $\dfrac{S}{I}$ is an exact sequence
		\[
		\xymatrixcolsep{2pc}
		\xymatrix{
			\mathbf{F} \colon \cdots \ar[r] & F_i \ar[r]^{d_i} & F_{i-1} \ar[r]^{d_{i-1}} & \cdots \ar[r] & F_0 \ar[r]^{d_0} & \dfrac{S}{I} \ar[r] & 0,
		}
		\]
        where each $F_i$ is a free $S-$module.
		
	If $I$ is a homogeneous ideal of $S$, then the free resolution is said to be {\it minimal} if $d_{i+1}\left(F_{i+1}\right) \subset (x_{1},\dots, x_{n})F_i$ for all $i$.
		
        Up to isomorphism, there exists a unique minimal  free resolution of $\dfrac{S}{I}$. If        $\mathbf{F}$ is the minimal free resolution of $\dfrac{S}{I}$, the $i^{\text{th}}$ {\it Betti number} (or {\it total Betti number}) of $\dfrac{S}{I}$ over $S$, denoted by $\beta_i^S$, is 
		\begin{align*}
			\beta_i^S\left(\dfrac{S}{I}\right)={\rm{rank}}\left(F_i\right),
		\end{align*}
        and if  $F_i = \bigoplus \limits_{p\in\Z} S\left(-p\right)^{\beta_{i,p}}$ for each $i$, the exponents $\beta_{i,p}$ are called {\it graded Betti numbers} of $\dfrac{S}{I}$.
	\end{definition}

	\begin{notation}
            Following \cite{mermin2011simplicial}, we write multigradings multiplicatively.   
            That is, for each monomial $m$ of $S$, set $S_{m}$ equal to the $k$-vector space spanned by $m$. Then $S = \bigoplus S_{m}$, and $S_{m} \cdot S_{n} = S_{mn}$, so this decomposition is a grading. We say that the monomial $m$ has multidegree $m$. We allow multidegrees to have negative exponents, so, for example, the twisted module $S(m^{-1})$ is a free module with generator in multidegree 	$m$, and $S(m^{-1})_{n} \cong S_{m^{-1}n}$ as a vector space; this is one-dimensional if no exponent
		of $m^{-1}n$ is negative, and trivial otherwise. Note that $S = S(1)$.	
	\end{notation}

    \begin{notation}
        All ideals considered in this paper are monomial ideals.  Recall that every monomial ideal $J$ has a unique minimal monomial generating set, which we refer to as $G(J)$.
    \end{notation}

    \begin{definition}
        We say that a monomial ideal is \emph{equigenerated (in degree $d$)} if all its minimal generators have the same degree $d$.  All ideals considered in this paper will be equigenerated.
    \end{definition}
	
    \subsection{Cellular resolutions and polyhedral complexes}

        The idea of parametrizing a resolution by a simplicial complex was first introduced in \cite{BPS} by Bayer, Peeva, and Sturmfels, and then extended to more general cellular complexes in \cite{BS} by Bayer and Sturmfels.  For details, see the original papers or the textbooks \cite[Chapter 4]{miller2004combinatorial} or \cite[Chapters 57--60]{Peeva}.  In this paper, we are specifically interested in resolutions supported on a polytopal complex; the definitions below loosely follow the treatment in \cite{miller2004combinatorial}
    
	\begin{definition}
		\phantomsection \label{2.5}
            A {\it polyhedral cell complex} $X$ is a finite collection of convex polytopes (in $\R^n$), called {\it faces} of $X$, satisfying two properties:
		\begin{itemize}
			\item If $\mathcal{P}\in X$ and $\mathcal{F}$ is a face of $\mathcal{P}$, then $\mathcal{F}\in X$.
			
			\item If $\mathcal{P}, \mathcal{Q}\in X$, then $\mathcal{P} \cap \mathcal{Q}$ (computed as a subset of $\mathbb{R}^{n}$) is a face of both $\mathcal{P}$ and $\mathcal{Q}$.
	       \end{itemize}
            The embedding of $X$ in $\mathbb{R}^{n}$ equips it with an \emph{orientation function}, \[\text{sign}:X\times X\to \{\pm 1, 0\}\] with the property that $\text{sign}(G,F)=\pm 1$ if and only if $G$ is a facet of $F$.  For all $F,G\in X$ we have $\displaystyle\sum_{H\in X}\text{sign}(G,H)\text{sign}(H,F) = 0$.
		
	\end{definition}

	\begin{definition}
        \phantomsection \label{2.6}
            The cell complex $X$ is called a \emph{labelled} complex if it is equipped with a labeling function
            \[
            \ell:X\to\{\text{monomials of $S$}\}
            \]
            that assigns a monomial $\ell(v)$ to each vertex of $X$ and satisfies $\ell(F)=\lcm\{\ell(v):v\in F\}$ for every face $F$ of $X$.  We refer to $\ell(F)$ as the \emph{label} of $F$. 
	\end{definition} 
	
	\begin{definition}
	\phantomsection \label{2.7}
            Let $X$ be a labelled polytopal complex.  
            We construct a complex of free modules $\mathcal{F}_{X}$ with basis given by the faces of $X$ and differential arising from the labels and the topological boundary maps.  More precisely, for an $i$-dimensional face $F\in X$, we create a basis element (also called $F$) with homological degree $i$ and internal degree $\ell(F)$.  Then the differential $\partial$ of $\mathcal{F}_X$ is 
		\begin{equation*}  
			\partial (F) = \sum \limits_{\text{facets } G \text{ of } F} {\rm{sign}} (G,F)  
			\dfrac{\ell(F)}{\ell(G)}G
            \end{equation*}

            If $\mathcal{F}_{X}$ is  acyclic, we say that it is a \emph{polytopal resolution}. 
    
	\end{definition}

        For any multidegree $m$, denote
        \begin{gather*}
            X_{\leq m}=\{F\in X : \ell(F) \text{ divides } m\}. 
        \end{gather*}

	\begin{theorem}
            \phantomsection\label{2.8}
            Let $X$ be a labelled polytopal complex.  The free complex $\mathcal{F}_X$ supported on $X$ is a (polytopal) resolution if and only if $X_{\leq m}$ is acyclic for all $m$.  
            In this case, it is a free resolution of the quotient module $\dfrac{S}{I}$, where $I=(\ell(v) : v\in X \text{ is a vertex})$ is generated by the monomial labels of vertices.
	\end{theorem}

    \subsection{Borel ideals}
	
	\begin{definition}
            We say that a monomial ideal $B$ is \emph{Borel} if it satisfies the condition
		\begin{quotation}
			If $m x_{j}\in B$ and $i<j$, then $m x_{i}\in B$.
		\end{quotation}	
	\end{definition}
	
	\begin{definition}  
            We refer to the operation sending $m x_{j}$ to $m x_{i}$ as a \emph{Borel move}, and the operation sending $m x_{i+1}$ to $m x_{i}$ as an \emph{elementary Borel move}.  Observe that a monomial ideal is Borel if and only if it is closed under (elementary) Borel moves.
	\end{definition}

	\begin{definition}
		\phantomsection\label{2.15}
            Let $m$ be a monomial. Define ${\rm{Borel}}(m)$ to be the smallest Borel ideal containing $m$. We say that ${\rm{Borel}}(m)$
		is  the {\it principal Borel ideal} generated by $m$ and $m$ is the Borel generator of ${\rm{Borel}}(m)$.
	\end{definition}

	\begin{remark}
            Let $\mathbf{m}=(x_{1},\dots, x_{n})$ be the homogeneous ideal.  Then $\mathbf{m}^{d}=\text{Borel}(x_{n}^{d})$ is a (principal) Borel ideal.		
	\end{remark}

        The notion of $Q$-Borel ideals was introduced in \cite{FMS2} to study ideals satisfying a weaker exchange property than Borel ideals.  We need it here to simultaneously study ideals that are Borel with different orders on the variables.  
        
	\begin{definition}
		\phantomsection\label{2.27}
            Let $Q$ be a poset on $\{x_1,\cdots,x_n\}$ and $m$ be a monomial in $S$.  A {\it $Q$-Borel move} on $m$ is an operation that
		sends $m$ to a monomial $m \cdot \dfrac{x_{i_1}}{x_{j_1}} \cdots \dfrac{x_{i_s}}              {x_{j_s}}$, where $i_t <_Q j_t$ for all $t$, and all $x_{j_t}$ divide $m$. 
	\end{definition}
	
	\begin{definition}
		\phantomsection\label{2.28}
		A monomial ideal $B$ is a {\it $Q$-Borel ideal} if $B$ is closed under $Q$-Borel moves.
            That is, if $m \in  B$, then any monomial obtained from a $Q$-Borel move on $m$ is also in $B$.

        For a monomial $m$, define the \emph{principal $Q$-Borel ideal generated by $m$}, written $\Borel_{Q}(m)$, to be the smallest $Q$-Borel ideal containing $m$.  
	\end{definition}
	
        There are several explicit ways to resolve a Borel ideal. Two of the most important are the Eliahou-Kervaire resolution \cites{EK,PS} and the complex-of-boxes resolution \cite{NagelReiner}. 

        The Eliahou-Kervaire resolution allows us to write out a basis for each module explicitly, and thus compute the (graded) Betti numbers of a Borel ideal by counting the basis elements associated to each monomial generator.  (See \cite{PS}*{Corollary 3.1} or \cite{FranciscoMerminSchweigBorelgenerators}*{Section 6} for formulas.) 

        The Eliahou-Kervaire resolution is supported on a regular cellular complex \cite{mermin2010cellular}, but not a polytopal complex.

	\subsection{The complex-of-boxes resolution} 
	
            In this section, we recall  the complex-of-boxes resolutions for equigenerated Borel ideals.  The complex of boxes, introduced by Nagel and Reiner \cite{NagelReiner}, is always   polytopal.
	
	\begin{definition}
            An \emph{elementary box} or \emph{interval} is a simplex whose vertices are a subset of variables.  Anticipating Notation \ref{n:cyclicinterval}, we set $\Delta_{[i,j]}$ equal to the simplex on vertices $\{x_{i},\dots, x_{j}\}$.
	\end{definition}
	
	\begin{definition}
            A \emph{box} is a finite product of intervals.  The simplicial structure of the intervals induces a polytopal structure on the box.
	\end{definition}

        \begin{example}
            In Figure \ref{f-2}, $\Delta_{[a,c]}\times \Delta_{[c,d]}$ is a box, as is $\Delta_{[a,b]}\times \Delta_{[b,c]}\times \Delta_{[c,d]}$.
                
                \begin{figure}[hbt]

                \begin{multicols}{2}
 
                    \begin{tikzpicture}[scale=0.6][line join = round, line cap = round]
                    \draw[-, thick, red]  (-3,1.5) -- (0,1.5) -- (-1.5,2.5) --  (-3,1.5); 
                    \draw[-, thick, blue]  (-3,1.5) -- (-3,-1.5);   
                    \draw[-, thick]  (-3,-1.5) -- (0,-1.5)  -- (0,1.5);   
                    \draw[-, dashed] (-3,-1.5) --  (-1.5,-0.5) --(0,-1.5)  ; 
                    \draw[-, dashed] (-1.5,2.5) -- (-1.5,-0.5)  ; 
  
                    \fill[] (-3,1.5) circle (3pt) node[left] {$ac$}; 
                    \fill[] (0,1.5) circle (3pt) node[right] {$bc$}; 
                    \fill[] (-1.5,2.5) circle (3pt) node[above] {$c^2$};   
                    
	                    \fill[] (-3,-1.5) circle (3pt) node[below] {$ad$}; 
                    \fill[] (0,-1.5) circle (3pt) node[below] {$bd$}; 
                    \fill[] (-1.5,-0.5) circle (3pt) node[right] {$cd$};   

                    \end{tikzpicture}       
                    
                \columnbreak

                    \begin{tikzpicture}[scale=0.6][line join = round, line cap = round]

                    \draw[-, thick]  (-3,1.5) -- (0,1.5) --(1.5,2.5)-- (-1.5,2.5) --  (-3,1.5);
                    \draw[-, thick]  (-3,1.5) -- (-3,-1.5) -- (0,-1.5) --  (0,1.5);

                    \draw[-, thick] (0,-1.5) -- (1.5,-0.5) --  (1.5,2.5);
  
                    \draw[-, dashed]  (-1.5,2.5) -- (-1.5,-0.5);
                    \draw[-, dashed]  (-3,-1.5) -- (-1.5,-0.5);
                    \draw[-, dashed]  (-1.5,-0.5) -- (1.5,-0.5);
                    \fill[] (-3,1.5) circle (3pt) node[left] {$abc$};
                    \fill[] (0,1.5) circle (3pt) node[right] {$ac^2$};
                    \fill[] (-1.5,2.5) circle (3pt) node[above] {$b^2c$};      
                    \fill[] (1.5,2.5) circle (3pt) node[above] {$bc^2$};    

                    \fill[] (-3,-1.5) circle (3pt) node[below] {$abd$};
                    \fill[] (0,-1.5) circle (3pt) node[below] {$acd$};
                    \fill[] (-1.5,-0.5) circle (3pt) node[below] {$b^2d$};
                    \fill[] (1.5,-0.5) circle (3pt) node[right] {$bcd$};
                    \end{tikzpicture} 
  
                \end{multicols}
                \caption{The polytopal structures of ${\color{red}\Delta_{[a,c]}}\times {\color{blue}\Delta_{[c,d]}}$ and $\Delta_{[a,b]}\times \Delta_{[b,c]}\times \Delta_{[c,d]}$.  The vertices are labelled with monomials rather than ordered tuples, using the labelling of Definition \ref{d:boxlabels}.}
                \label{f-2}
                \end{figure}
        \end{example}

	\begin{definition}\label{d:boxlabels}
            Let $\Gamma = \displaystyle\prod_{i=1}^{n}\Delta_{i}$ be a box.  Every vertex of $\Gamma$ is an ordered $n$-tuple of variables, $v=(\dots, x_{j_{i}}\in \supp \Delta_{i}, \dots)$.  We label the vertices with monomials by forgetting the ordering, $\ell(v)=\displaystyle\prod_{i=1}^{n}x_{j_{i}}$.
		
            We say that $\Gamma$ is an \emph{admissible box} if all vertices have distinct monomial labels.
	\end{definition}

         \begin{example}
            As shown in Figure \ref{f-3}, $\Delta_{[a,c]}\times \Delta_{[c]}\times \Delta_{[c,d]}$ is an admissible box, whereas $\Delta_{[a,b]}\times \Delta_{[a,b]}$ is not admissible.
                
                \begin{figure}[hbt]

                \begin{multicols}{2}
 
                    \begin{tikzpicture}[scale=0.6][line join = round, line cap = round]
                    
                    \draw[-, thick]  (-3,1.5) -- (0,1.5) -- (-1.5,2.5) --  (-3,1.5);
                    \draw[-, thick]  (-3,1.5) -- (-3,-1.5) -- (0,-1.5) --  (0,1.5);
                    \draw[-, dashed]  (-1.5,2.5) -- (-1.5,-0.5);
                    \draw[-, dashed]  (-3,-1.5) -- (-1.5,-0.5);
                    \draw[-, dashed]  (0,-1.5) -- (-1.5,-0.5);
  
                    \fill[] (-3,1.5) circle (3pt) node[left] {$ac^2$};
                    \fill[] (0,1.5) circle (3pt) node[right] {$bc^2$};
                    \fill[] (-1.5,2.5) circle (3pt) node[above] {$c^3$};    

                    \fill[] (-3,-1.5) circle (3pt) node[below] {$acd$};
                    \fill[] (0,-1.5) circle (3pt) node[below] {$bcd$};
                    \fill[] (-1.5,-0.5) circle (3pt) node[right] {$c^2d$};   
                    
                    \end{tikzpicture}

                \columnbreak

                    \begin{tikzpicture}[scale=0.6][line join = round, line cap = round]
                    \draw[-, thick]  (0,0)--(3,0)--(3,3)--(0,3)--(0,0);
  
                    \fill[red] (0,3) circle (3pt) node[above] {$ab$};
                    \fill[] (3,3) circle (3pt) node[above] {$b^2$};

                    \fill[] (0,0) circle (3pt) node[below] {$a^2$};
                    \fill[red] (3,0) circle (3pt) node[below] {$ab$};
                 
                    \end{tikzpicture}       
  
                \end{multicols}
                \caption{The polytopal structures of $\Delta_{[a,c]}\times \Delta_{[c]}\times \Delta_{[c,d]}$ and $\Delta_{[a,b]}\times \Delta_{[a,b]}$.}
                \label{f-3}
                \end{figure}
        \end{example}

        \begin{remark}
        \label{r:monomialsAreBoxes} 
            Every monomial is itself an admissible box.  More precisely, if $m=\prod_{i=1}^{d}x_{i}$ (with indices possibly repeated), then the box $\Gamma=\prod_{i=1}^{d}\Delta_{\{i\}}$ is a zero-dimensional box consisting of the single vertex $m$.
        \end{remark}

	\begin{construction}\label{c:boxcells}
            For a monomial $m=x_{i_{1}}\dots x_{i_{d}}$, set $\Gamma(m)$ equal to the $(i_{d}-1)$-dimensional box
		\[
            \Gamma(m)=\Delta_{[1,i_{1}]} \times \Delta_{[i_{1},i_{2}]} \times \Delta_{[i_{2},i_{3}]}\times \dots \times \Delta_{[i_{d-1},i_{d}]}.
            \]
		Observe that $\Gamma(m)$ is admissible. 
		
            If $I$ is an equigenerated Borel ideal with minimal generating set $I=(m_{1},\dots, m_{s})$, then the \emph{complex of boxes} associated to $I$ is the polytopal complex generated by the faces $\{\Gamma(m_{1}), \dots, \Gamma(m_{s})\}$.  (This generating set is not minimal, since some faces, e.g. $\Gamma(x_{1}^{d})$, are not facets.)
	\end{construction}

        \begin{remark}\label{r:squarefreeSameBox}
            If $m$ is a monomial and $\sqrt{m}$ is its squarefree part, then $\Gamma(m)\cong \Gamma(\sqrt{m})$, since, after rearranging the factors, we have $\Gamma(m) = \Gamma(\sqrt{m})\times (\frac{m}{\sqrt{m}})$.
        \end{remark}

	\begin{example}
            Like the Eliahou-Kervaire resolution, the complex of boxes associated to $(x_{1},\dots,x_{n})^{d}$ is a subdivision of the topological $n$-simplex with vertices labeled by powers of the variables.
            In three variables, the two complexes coincide (modulo a permutation) as shown in Figure \ref{f-4}.  But with more variables, differences begin to emerge.  Most notably, the complex of boxes is polytopal, whereas the Eliahou-Kervaire resolution is not, as seen in Figure \ref{f-5}. For example, the complex-of-boxes resolution of $(a,b,c,d)^2$ can be decomposed into four polytopes as in Figure \ref{f-6}. 
                
            \begin{figure}[hbt]

            \begin{multicols}{2}
 
                \begin{tikzpicture}[scale=0.6][line join = round, line cap = round]
                    \draw[-, thick]  (-3,1.5) -- (3,1.5) -- (0,6) -- (-3,1.5);
                    \draw[-, thick]  (-1,4.5) -- (1,1.5);
                    \draw[-, thick]  (1,4.5) --  (-1,1.5);
                    \draw[-, thick]  (-2,3) -- (-1,1.5);
                    \draw[-, thick]  (1,4.5) -- (0,3) -- (1,1.5) -- (2,3);
                        
                    \fill[] (-3,1.5) circle (3pt) node[below] {$a^3$};
                    \fill[] (-1,1.5) circle (3pt) node[left,below] {$a^2b$};
                        
                    \fill[] (1,1.5) circle (3pt) node[below] {$ab^2$};
                    \fill[] (3,1.5) circle (2pt) node[below] {$b^3$};
                        
                    \fill[] (-2,3) circle (3pt) node[left] {$a^2c$};
                    \fill[] (0,3) circle (3pt) node[right] {$abc$};
                    \fill[] (2,3) circle (3pt) node[right] {$b^2c$};
                        
                    \fill[] (-1,4.5) circle (3pt) node[left] {$ac^2$};
                    \fill[] (1,4.5) circle (3pt) node[right] {$bc^2$};
                        
                    \fill[] (0,6) circle (3pt) node[above] {$c^3$};
                        
                \end{tikzpicture}  

            \columnbreak

                \begin{tikzpicture}[scale=0.6][line join = round, line cap = round]
                 
                    \draw[-, thick]  (-3,1.5) -- (3,1.5) -- (0,6) -- (-3,1.5);
                    \draw[-, thick]  (-1,4.5) -- (1,4.5);
                    \draw[-, thick]   (-2,3)  -- (2,3) ;
                    \draw[-, thick]   (-2,3)   --  (-1,1.5) ;
                    \draw[-, thick]   (-1,4.5)   --  (1,1.5) ;

                    \fill[] (-3,1.5) circle (3pt) node[below] {$a^3$};
                    \fill[] (-1,1.5) circle (2pt) node[left,below] {$a^2b$};
                        
                    \fill[] (1,1.5) circle (3pt) node[below] {$ab^2$};
                    \fill[] (3,1.5) circle (3pt) node[below] {$b^3$};
                        
                    \fill[] (-2,3) circle (3pt) node[left] {$a^2c$};
                    \fill[] (0,3) circle (3pt) node[below] {};
                    \fill[] (-0.2,3) circle (0pt) node[below] {$abc$};
                    \fill[] (2,3) circle (3pt) node[right] {$b^2c$};
                        
                    \fill[] (-1,4.5) circle (3pt) node[left] {$ac^2$};
                    \fill[] (1,4.5) circle (3pt) node[right] {$bc^2$};
                        
                    \fill[] (0,6) circle (2pt) node[above] {$c^3$};
                        
                \end{tikzpicture}

            \end{multicols}
            \caption{The Eliahou-Kervaire resolution (left) and the complex-of-boxes resolution (right) of $(a,b,c)^{3}$.}
            \label{f-4}
            \end{figure}

            \begin{figure}[hbt]

            \begin{multicols}{2}
 
                \begin{tikzpicture}[scale=0.6][line join = round, line cap = round]
                    \draw[-, thick]  (-3,1.5) -- (3,1.5) -- (0,6) -- (-3,1.5);
                    \draw[-, dashed]  (0,3.5) -- (-3,1.5);
                    \draw[-, dashed]  (0,3.5) -- (3,1.5);
                    \draw[-, dashed]  (0,3.5) -- (0,6);
                    \draw[-, dashed]  (0,4.75) -- (-1.5,3.75);
                    \draw[-, dashed]  (0,4.75) -- (1.5,3.75);  

                    \draw[-, dashed]  (-1.5,2.5) -- (-1.5,3.75);  
                    \draw[-, dashed]  (-1.5,2.5) -- (0,1.5);
                    \draw[-, dashed]  (1.5,2.5) -- (0,1.5);
                    \draw[-, dashed]  (1.5,2.5) -- (1.5,3.75) ;

                    \draw[-, thick]  (-1.5,3.75) -- (0,1.5) ;
                    \draw[-, thick]  (1.5,3.75) -- (0,1.5) ;
                        
                    \fill[] (-3,1.5) circle (3pt) node[below] {$a^2$};
                    \fill[] (3,1.5) circle (3pt) node[below] {$b^2$};
                    \fill[] (0,6) circle (3pt) node[above] {$c^2$};
                    \fill[] (0,3.5) circle (3pt) node[below] {$d^2$};
   
                    \fill[] (0,1.5) circle (3pt) node[below] {$ab$};

                    \fill[] (-1.5,3.75) circle (3pt) node[left] {$ac$};
                    \fill[] (1.5,3.75) circle (3pt) node[right] {$bc$};
                    \fill[] (-1.5,2.5) circle (3pt) node[below] {$ad$};
                    \fill[] (1.5,2.5) circle (3pt) node[below] {$bd$};
                    \fill[] (0,4.75) circle (3pt) node[right] {$cd$}; 
                \end{tikzpicture}

            \columnbreak

                \begin{tikzpicture}[scale=0.6][line join = round, line cap = round]
                    \draw[-, thick]  (-3,1.5) -- (3,1.5) -- (0,6) -- (-3,1.5);   
                    \draw[-, dashed]  (0,3.5) -- (-3,1.5);
                    \draw[-, dashed]  (0,3.5) -- (3,1.5);
                    \draw[-, dashed]  (0,3.5) -- (0,6);
                    \draw[-, dashed]  (-1.5,3.75) --  (-1.5,2.5);
                    \draw[-]  (-1.5,3.75) --  (1.5,3.75);
                    \draw[-]  (-1.5,3.75) --  (0,1.5) ;

                    \draw[-, dashed]  (-1.5,2.5) --  (0,1.5) ;
                    \draw[-, dashed]  (-1.5,2.5) --  (1.5,2.5) ;

                    \draw[-, dashed]  (1.5,3.75) --  (1.5,2.5) ; 

                    \draw[-, dashed]  (0,4.75) --  (-1.5,2.5) ; 

                    \draw[-, dashed]  (0,4.75) --   (1.5,2.5) ; 
                        
                    \fill[] (-3,1.5) circle (3pt) node[below] {$a^2$};
                    \fill[] (3,1.5) circle (3pt) node[below] {$b^2$};
                    \fill[] (0,6) circle (3pt) node[above] {$c^2$};
                    \fill[] (0,3.5) circle (3pt) node[below] {$d^2$};
   
                    \fill[] (0,1.5) circle (3pt) node[below] {$ab$};
                        
                    \fill[] (-1.5,3.75) circle (3pt) node[left] {$ac$};
                    \fill[] (1.5,3.75) circle (3pt) node[right] {$bc$};
                    \fill[] (-1.5,2.5) circle (3pt) node[below] {$ad$};
                    \fill[] (1.5,2.5) circle (3pt) node[below] {$bd$};
                    \fill[] (0,4.75) circle (3pt) node[right] {$cd$}; 
                \end{tikzpicture}

            \end{multicols}
            \caption{The complexes supporting the Eliahou-Kervaire resolution and the complex-of-boxes resolution of  $(a,b,c,d)^{2}$.  The Eliahou-Kervaire complex is not polytopal, since the cell containing $d^2$ contains $ac$ and $bc$ but none of the segment connecting them.  The box complex consists of two tetrahedra and two triangular prisms.}
            \label{f-5}
            \end{figure}

            \begin{figure}[hbt]

            \begin{multicols}{2}

                \begin{tikzpicture}[scale=0.8][line join = round, line cap = round]
                         
                \draw[-, thick]  (-1.5,2.5) -- (1.5,2.5) -- (0,4.75) -- (-1.5,2.5);             
                \draw[-, dashed]  (0,3.5) -- (-1.5,2.5);
                \draw[-, dashed]  (0,3.5) -- (1.5,2.5);
                \draw[-, dashed]  (0,3.5) -- (0,4.75);
                 
                \fill[] (0,3.5) circle (2pt) node[below] {$d^2$};
                \fill[] (-1.5,2.5) circle (2pt) node[below] {$ad$};
                \fill[] (1.5,2.5) circle (2pt) node[below] {$bd$};
                \fill[] (0,4.75) circle (2pt) node[above] {$cd$};
                         
                \end{tikzpicture} 

            \columnbreak

                \begin{tikzpicture}[scale=0.8][line join = round, line cap = round]                  \draw[-, thick]  (-3,1.5)  --  (0,1.5) -- (-1.5,3.75) --  (-3,1.5);             
                \draw[-, dashed]   (-1.5,2.5) -- (-3,1.5);
                \draw[-, dashed]   (-1.5,2.5) -- (0,1.5);
                \draw[-, dashed]   (-1.5,2.5) -- (-1.5,3.75); 
                        
                \fill[] (-3,1.5) circle (2pt) node[below] {$a^2$};
                \fill[] (0,1.5) circle (2pt) node[below] {$ab$};
                \fill[] (-1.5,3.75) circle (2pt) node[above] {$ac$};
                \fill[] (-1.5,2.5) circle (2pt) node[below] {$ad$};
                \end{tikzpicture}
 
            \end{multicols}

            \begin{multicols}{2} 

                \begin{tikzpicture}[scale=0.8][line join = round, line cap = round]
                
                \draw[-, thick]  (0,1.5) -- (-1.5,3.75) -- (-1.5,2.5) -- (0,1.5);   
                \draw[- ]  (-1.5,3.75) -- (1.5,3.75) --  (3,1.5)--  (0,1.5) --  (-1.5,3.75);
                \draw[-, dashed]   (-1.5,2.5) -- (1.5,2.5) -- (3,1.5);
                \draw[-, dashed]     (1.5,2.5) -- (1.5,3.75);
                
                \fill[] (3,1.5) circle (2pt) node[below] {$b^2$};
                \fill[] (0,1.5) circle (2pt) node[below] {$ab$};
                \fill[] (-1.5,3.75) circle (2pt) node[above] {$ac$};
                \fill[] (1.5,3.75) circle (2pt) node[above] {$bc$};
                \fill[] (-1.5,2.5) circle (2pt) node[left] {$ad$};
                \fill[] (1.5,2.5) circle (2pt) node[below] {$bd$};
                
                \end{tikzpicture}

            \columnbreak 
                \begin{tikzpicture}[scale=0.8][line join = round, line cap = round]
                \draw[-, thick]  (0,6) -- (-1.5,3.75) --  (1.5,3.75)  -- (0,6);    
                \draw[-, thick]   (-1.5,3.75) -- (-1.5,2.5) --  (1.5,2.5) --  (1.5,3.75) --  (-1.5,3.75); 
                \draw[-, dashed] (-1.5,2.5) --  (0,4.75) -- (1.5,2.5) ;
                \draw[-, dashed] (0,6) --  (0,4.75)   ;
                
                \fill[] (0,6) circle (2pt) node[above] {$c^2$};
                \fill[] (-1.5,3.75) circle (2pt) node[left] {$ac$};
                \fill[] (1.5,3.75) circle (2pt) node[right] {$bc$};
                \fill[] (-1.5,2.5) circle (2pt) node[below] {$ad$};
                \fill[] (1.5,2.5) circle (2pt) node[below] {$bd$};
                \fill[] (0,4.75) circle (2pt) node[right] {$cd$};
                
                \end{tikzpicture}
  
            \end{multicols}

            \caption{The four boxes comprising the complex-of-boxes resolution of $(a,b,c,d)^{2}$.}
            \label{f-6}
            \end{figure}
        \end{example}

	\begin{theorem}\cite[Theorem 3.12]{NagelReiner}
            Let $I$ be an equigenerated Borel ideal.  Then the complex of boxes associated to $I$ is a polytopal complex, and supports a minimal resolution of $I$.
	\end{theorem}	

    \subsection{Notation for cyclic intervals}

        Many of the constructions and proofs throughout the paper will involve working with a cyclically symmetric collection of ideals.  
        To avoid notational nightmares, we introduce some notation for dealing with cyclic intervals.  Let $\sigma = (1\;\;2\;\;3\;\;\dots\;\;n)$ be the standard generator for the cyclic group on $\{1,2,\dots, n\}$.

    \begin{notation}\label{n:cyclicinterval}
        For any $t\in \{1,\dots, n\}$, the \emph{initial interval} $[1,t]$ refers to the set $\{1,2,\dots, t\}$.  For any $i,j\in\{1,\dots,n\}$, the cyclic interval refers to the interval obtained by applying $\sigma^{i-1}$ to the appropriate initial interval.  That is, $[i,j]=\{i,i+1,\dots, j\}$ if $i\leq j$ and $[i,j]=\{i,i+1,\dots, n,1,\dots, j\}$ if $i>j$.  

        We will define many objects in terms of cyclic intervals.  
        \begin{itemize}
            \item The \emph{simplicial complex supported on $[i,j]$} is $\Delta_{[i,j]}=\langle \{x_{s}:s\in [i,j]\}\rangle$.
            \item The \emph{monomial supported on $[i,j]$} is $m_{[i,j]}=\prod_{s\in[i,j]}x_{s}$.
            \item The \emph{prime ideal supported on $[i,j]$} is $\mathbf{m}_{[i,j]} = (x_{s}:s\in [i,j])$.
            \item The poset $Q_{i}$ is the chain on $[i,i-1]$ defined by the covering relations $x_{i}<x_{i+1}<\dots <x_{n}<x_{1}<\dots <x_{i-1}$.
        \end{itemize}
    \end{notation}

 
    \section{Gluing of complexes}\label{sec3}

        The goal of this paper is to build (new) polytopal resolutions for powers and pinched powers of the maximal ideal by gluing simpler complexes together.  The gluing process is unfortunately quite technical, and we rely on a number of lemmas due to Sinefakopoulos \cite{Sinefakopoulos2007}.  We collect the necessary results in this section.

	\begin{lemma}\cite[Lemma 3.1]{Sinefakopoulos2007}
        \phantomsection\label{2.9}
        Let $I$ and $J$ be  two monomial ideals in $S$ such that $G(I+J)=G(I) \cup G(J)$ set-theoretically. Suppose that 
        \begin{enumerate}[{\normalfont (i)}]
            \item $X$ and $Y$ are regular cell complexes in $\R^N$ that support a (minimal) free resolution of $I$ and $J$, respectively, and
            \item $X\cap Y$ is a regular cell complex that supports a (minimal) free resolution of $I \cap J$.
	\end{enumerate}
		
	Then $X\cup Y$ supports a (minimal) free resolution of $I + J$.
		
	\end{lemma}

	\begin{remark}\cite[Remark 3.2]{Sinefakopoulos2007}
        \phantomsection\label{rmk4.3-P}
        For any two monomial ideals $I$ and $J$, we have
        $G(I + J) \subseteq G(I) \cup  G(J)$.
        Our assumption that $G(I + J) = G(I) \cup G(J)$ guarantees the right labelling of the
        cell complex $X \cup Y$. A case where equality holds is when all elements of
        $G(I) \cup G(J)$ are of the same degree.
        \end{remark}

        \begin{remark}
        From assumption (ii) of Lemma \ref{2.9}, $I\cap J$ is generated by the labels of the vertices of $X\cap Y$.  We conclude that 
        $G(I \cap J) = G(I) \cap G(J)$.  This is an extremely strong assumption, as intersecting ideals almost always produces generators in higher degree.  But it does hold in the cases of interest, where, among other things, the ideals in question all have linear quotients, and the boundaries of the complexes coincide.
	\end{remark}

        \begin{lemma}
        \phantomsection \label{l:unionmanyboxes}
        Let $M_1,M_2,\dots,M_n$ be   equigenerated monomial ideals, all generated in the same degree. Assume that 
        \begin{enumerate}[{\normalfont (i)}]
            \item $W_1, \dots,W_n$ are regular cell complexes in $\R^N$ that support (minimal) free resolutions of $M_1,\dots,M_n$, respectively, and     	
            \item For all subsets $\{i_{1},\dots, i_{\ell}\}$, $W_{i_1}\cap W_{i_2} \cap \cdots \cap W_{i_{\ell}}$ is a regular cell complex that supports a (minimal) free resolution of $M_{i_1}\cap M_{i_2} \cap \cdots \cap M_{i_{\ell}}$.
        \end{enumerate}

        Then $W_{1}\cup W_{2} \cup \cdots \cup W_{n}$ supports a (minimal) free resolution of $M_1 + M_2 + \cdots + M_n$.
        \end{lemma}

            \begin{proof} 
            Since $M_1,M_2,\dots,M_n$ are all  equigenerated monomial ideals, 
                \begin{align*}
                    G\left(M_{i_1} + M_{i_2} + \cdots + M_{i_\ell}\right) = G\left(M_{i_1}\right) \cup G\left(  M_{i_2}  \right) \cup \cdots \cup G\left(  M_{i_\ell}\right).
                \end{align*}

            We will prove the statement by induction on the number of ideals, $n$.

            The case $n=1$ is trivial and the case $n=2$ is given by Lemma \ref{2.9}. Assume that 
            the statement holds for $n$. We shall show that it is also true for $n+1$. Let
                \begin{align*}
                    M_1'&=M_1 + M_{n+1}, & W_1'&=W_1 \cup W_{n+1}, \\
                    M_2'&=M_2 + M_{n+1}, & W_2'&=W_2 \cup W_{n+1}, \\
                    & \vdots &&\vdots\\
                    M_n'&=M_n + M_{n+1},  & W_n'&=W_n \cup W_{n+1}.
                \end{align*}

            It follows from Lemma \ref{2.9} that $W_i'$ is a (minimal) free resolution of $M_i'$, for $1\leq i \leq n$, and that 
                \begin{align*}
                    W'_{i_1}\cap W'_{i_2} \cap \cdots \cap W'_{i_{\ell}}=\left(W_{i_1}\cap W_{i_2} \cap \cdots \cap W_{i_{\ell}}\right) \cup W_{n+1}
                \end{align*}
            is a (minimal) free resolution of 
                \begin{align*}
                    \left(M_{i_1}\cap M_{i_2} \cap \cdots \cap M_{i_{\ell}}\right)+M_{n+1}=M_1'\cap M_2' \cap \cdots \cap M_{i_{\ell}}'.
                \end{align*}

            By the inductive hypothesis (applied to $M_{i}'$ and $W_{i}'$), we obtain that 
                \begin{align*}
                    W_1\cup W_2 \cdots \cup W_{n+1} = W_1' \cup W_2' \cup \cdots  \cup W_n'
                \end{align*}
            is a (minimal) free resolution of 
                \begin{align*}
                    M_1 + M_2 + \cdots M_{n+1} = M_1' + M_2' +  \cdots + M_n'.
                \end{align*}
            \end{proof}

        \begin{corollary}
        \phantomsection \label{4.4P}
        Let $M_1,M_2,\dots,M_n$ be   equigenerated monomial ideals, all generated in the same degree. Assume that 
        \begin{enumerate}[{\normalfont (i)}]
            \item $W_1, \dots,W_n$ are regular cell complexes in $\R^N$ that support (minimal) free resolutions of $M_1,\cdots,M_n$, respectively, and
            \item $W_{i_1}\cap W_{i_2} \cap \cdots \cap W_{i_{\ell}}$ is a regular cell complex that supports a (minimal) free resolution of $M_{i_1}\cap M_{i_2} \cap \cdots \cap M_{i_{\ell}}$.
        \end{enumerate}

        Then $\left(W_{1}\cup W_{2} \cup \cdots \cup W_{n-1}\right)\cap W_{n}$ supports a (minimal) free resolution of 
            \begin{align*}
                \left(M_1 + M_2 + \cdots + M_{n-1}\right)\cap M_n.
            \end{align*}
    
        \end{corollary}

            \begin{proof}
            Apply Lemma \ref{l:unionmanyboxes} to $M_1\cap M_n,M_2\cap M_n, \cdots,M_{n-1} \cap M_n$.  
            \end{proof}

    \begin{lemma}\cite[Lemma 3.3]{Sinefakopoulos2007}
        \phantomsection\label{l:Achilleasdisjoint}
        Let $I \subset k[x_1,\dots,x_s]$ and $J\subset k[x_{s+1},\dots,x_n]$ be two monomial ideals. Suppose that $X$ and $Y$ are regular cell complexes in $\R^N$ of dimension $s-1$ and $n-s-1$, respectively, that support (minimal) free resolutions of $I$ and $J$, respectively. Then the regular cell complex $X\times Y$ supports a (minimal) free resolution for $IJ.$
		
	\end{lemma}
	
	\begin{lemma}\cite[Lemma 3.5]{Sinefakopoulos2007}
        \phantomsection\label{2.11}
        Let $I \subset k[x_1,\dots,x_s]$ and $J\subset k[x_{s},\dots,x_n]$ be two monomial ideals such that $ \lvert G(IJ) \rvert = \lvert G(I) \rvert \cdot \lvert G(J) \rvert$. Suppose that $X$ and $Y$ are regular cell complexes in some $\R^N$ of dimension $s-1$ and $n-s$, respectively, that support (minimal) free resolutions of $I$ and $J$, respectively. Then the regular cell complex $X\times Y$ supports a (minimal) free resolution for $IJ.$
		
	\end{lemma}


\section{Symmetric polytopal resolutions of $I=(x_{1},\dots, x_{n})^{n}$ and $I\smallsetminus x_{1}\cdots x_{n}$}\label{sec4} 

    In this section we build the intuition for the rest of the paper.  Unfortunately, when we work in the full generality of $(x_{1},\dots, x_{n})^{d}$, the notation gets messy enough to obscure the ideas.  This section covers the special case where the degree is equal to the number of variables and we have the full symmetry arising from working around a ``central'' monomial (in this case, the product of the variables).  Everything here is a special case of \S\ref{sec5} and \S\ref{sec6}, so we further emphasize the intuition by suppressing the proofs.  Readers who don't need our geometric intuition can safely skip this section and jump directly to \S\ref{sec5}.

    The goal of this section is to build a resolution of   $I=\mathfrak{m}^{n}=(x_{1},\dots, x_{n})^{n}$ which is a polytopal subdivision of the $(n-1)$-dimensional simplex, and is also symmetric about its center $m=x_{1}\cdots x_{n}$.  (In particular, no variable is preferenced over the others, so all corners of the simplex will look the same.)
    Our strategy is to let the cyclic group act on the principal Borel ideal $I_{1}=\Borel(m)$ and its complex of boxes.  This produces $n$ polytopal complexes which intersect in positive codimension; gluing them together yields the desired resolution (Theorem \ref{th:centralPinchedPower}).

    Using this symmetric resolution, we then construct a resolution for the ``pinched power'' ideal $\widehat{I}$, generated by all degree $n$ monomials except the central monomial $m=x_{1}\cdots x_{n}$.  This is accomplished in Construction \ref{c:centralPinchedPower} by merging all boxes containing $m$ into a single larger polyhedron.

        \begin{notation}
            For each $i$, let $I_{i}=\Borel_{Q_{i}}(m)$ be the principal $Q_{i}$-Borel ideal generated by $m$.  Denote by $X_{1}$ the complex of boxes resolving $I_{1}$, and by $X_{i}$ the appropriate cyclic permutation of $X_{1}$.  Observe that $X_{i}$ minimally resolves $I_{i}$; we abuse notation by referring to $X_{i}$ as the \emph{complex of boxes resolving $I_{i}$}.
        \end{notation}

        \begin{example}  
            When $n=3$, the ideals $I_{i}$ are $I_{1}=(abc,a^2{c},ab^{2},a^{2}b,a^{3})$, $I_{2}=(abc,ab^{2},bc^{2},b^{2}c,b^{3})$ and $I_{3}=(abc,bc^{2},a^{2}c,ac^{2},c^{3})$.  In Figure \ref{f:3varcomplexes}, the complexes of boxes resolving these ideals are shown in red, yellow, and blue, respectively.

            \begin{figure}[hbt]
            \centering
                \begin{tikzpicture}[scale=0.8][line join = round, line cap = round]
                
                    \fill[rose] (-3,1.5) -- (1,1.5) -- (0,3) -- (-2,3) --  (-3,1.5);
				\fill[sand]  (0,3) --  (1,1.5) -- (3,1.5) --(1,4.5) --  (0,3);
				\fill[cyan] (0,6) --(1,4.5) -- (0,3) -- (-2,3) --(0,6);
				
				\draw[-, thick]  (-3,1.5) -- (3,1.5) -- (0,6) -- (-3,1.5);
				\draw[-, thick]  (-2,3) -- (-1,1.5);
				\draw[-, thick]  (-2,3) -- (0,3);
				\draw[-, thick]  (-1,4.5) -- (1,4.5);
				\draw[-, thick]  (1,4.5) -- (0,3) -- (1,1.5) -- (2,3);

				\fill[] (-3,1.5) circle (3pt) node[below] {$a^3$};
				\fill[] (-1,1.5) circle (3pt) node[left,below] {$a^2b$};
				
				\fill[] (1,1.5) circle (3pt) node[below] {$ab^2$};
				\fill[] (3,1.5) circle (3pt) node[below] {$b^3$};
				
				\fill[] (-2,3) circle (3pt) node[left] {$a^2c$};
				\fill[blue] (0,3) circle (3pt) node[right] {$abc$};
				\fill[] (2,3) circle (3pt) node[right] {$b^2c$};
				
				\fill[] (-1,4.5) circle (3pt) node[left] {$ac^2$};
				\fill[] (1,4.5) circle (3pt) node[right] {$bc^2$};
				
				\fill[] (0,6) circle (3pt) node[above] {$c^3$};

			\end{tikzpicture}
            \caption{The ideals $I_{1}$, $I_{2}$, and $I_{3}$ inside $(a,b,c)^{3}$.}
            \label{f:3varcomplexes}
            \end{figure}

        When $n=4$, we have $I_{1}=(abcd,ab^{2}d,a^{2}cd,a^{2}bd,a^{3}d, abc^{2},a^{2}c^{2},ab^{2}c,a^{2}bc,a^{3}c,ab^{3},a^{2}b^{2},a^{3}b,a^{4})$.  $I_{2}$, $I_{3}$, and $I_{4}$ are cyclic permutations of $I_{1}$.  The complexes of boxes resolving these ideals are shown in Figure \ref{f:4varcomplexes}.

        \begin{tiny}
       \begin{figure}
        \begin{center}
            \begin{tikzpicture}[scale=0.12][line join = round, line cap = round] 
			    \draw[-, thick]   (5,24)   --    (-10,6)  ; 
                \draw[-, thick]    (-10,6)   --  (2,2)  ;  
                \draw[-, thick]    (7,8) --   (5,24)    ; 
				\draw[-, thick]   (6,16)  --  (0,18) ; 
			    \draw[-, thick]   (7,8) --     (-5,12) ; 
                \draw[-, thick]  (6,16)  --  (-4,4)   ; 
                \draw[-, thick]  (6,16)  --   (10,18)  ; 
                \draw[-, thick,dashed]  (10,18)    --   (0,18) ; 
                \draw[-, thick,dashed]    (1,10)  --  (5,12) ; 
                \draw[-, thick,dashed]    (5,12) --   (-5,12)  ; 
                \draw[-, thick,dashed]   (10,18)  --  (0,6)  ; 
                \draw[-, thick]      (5,24) -- (10,18) ; 
                \draw[-, thick,dashed]   (-10,6)   -- (0,6)   ; 

                \draw[-, thick,dashed]  (-4,4)   --     (0,6)  ; 

                \draw[-, thick] (7,8)    (7,8) --   (11,10)   ; 
                \draw[-, thick] (2,2)  (11,10) --  (10,18)  ; 
                \draw[-, thick,dashed]  (5,12) --  (11,10) ; 
                \draw[-, thick]   (7,8) --  (2,2) ; 
                \draw[-, thick]   (11,10) -- (6,4) --  (2,2) ; 
                \draw[-, thick,dashed]   (6,4)   -- (0,6)   ; 

                \fill[] (-10,6) circle (8pt) node[left] {$a^3d$}; 
                \fill[] (-4,4) circle (8pt) node[below] {$a^2bd$}; 
                \fill[] (2,2) circle (8pt) node[below] {$ab^2d$}; 
                \fill[] (0,6) circle (8pt) node[right] {$a^2cd$};
                \fill[red] (6,4) circle (8pt) node[right] {$abcd$}; 

                \fill[] (5,12) circle (8pt) node[below] {$acd^2$}; 
                \fill[] (0,18) circle (8pt) node[left] {$ad^3$}; 
                \fill[] (-5,12) circle (8pt) node[left] {$a^2d^2$}; 
                \fill[] (1,10) circle (8pt) node[left] {$abd^2$}; 

                \fill[] (7,8) circle (8pt) node[right] {$b^2d^2$}; 

                \fill[] (5,24) circle (8pt) node[above] {$d^4$}; 

                \fill[] (6,16) circle (8pt) node[right] {$bd^3$}; 

                \fill[] (10,18) circle (8pt) node[right] {$cd^3$}; 

                \fill[] (11,10) circle (8pt) node[right] {$bcd^2$}; 

			\end{tikzpicture}
        \end{center}

        \begin{multicols}{2}
           \begin{tikzpicture}[scale=0.12][line join = round, line cap = round] 
		      
			    \draw[-, thick]  (-15,0)  --  (3,-6)   ; 
                \draw[-, thick]   (3,-6)  -- (11,-2) ;  
                \draw[-, thick,dashed]    (5,0)   --  (-15,0); 
				\draw[-, thick,dashed]  (-5,0)  --  (-9,-2) ; 
			    \draw[-, thick,dashed]   (5,0) --  (-3,-4) ; 
                \draw[-, thick,dashed]  (-5,0)  --  (7,-4) ; 
                \draw[-, thick,dashed] (-5,0)   -- (-10,6)    ; 
                \draw[-, thick]   (-10,6)  --   (-9,-2); 
                \draw[-, thick,dashed]  (1,-2)  --  (-4,4)  ; 
                \draw[-, thick]     (-4,4) --  (-3,-4) ; 
                \draw[-, thick]  (-10,6)  --  (2,2) ; 
                \draw[-, thick]  (-15,0)  -- (-10,6) ; 

                \draw[-, thick] (7,-4) -- (2,2)    ; 
                \draw[-, thick]  (2,2)  -- (3,-6)  ; 

                \draw[-, thick,dashed] (5,0)  -- (0,6)    ; 
                \draw[-, thick]  (0,6) -- (-10,6)  ; 
                \draw[-, thick]  (-4,4) --   (0,6)  ; 
                \draw[-, thick,dashed]  (5,0) --   (11,-2)  ; 
                \draw[-, thick]   (0,6) --   (6,4) -- (2,2)  ; 
                \draw[-, thick]      (6,4) -- (11,-2) ; 

			\fill[] (-15,0) circle (8pt) node[left] {$a^4$}; 
                \fill[] (-9,-2) circle (8pt) node[below] {$a^3b$}; 
                \fill[] (-3,-4) circle (8pt) node[below] {$a^2b^2$}; 
                \fill[] (3,-6) circle (8pt) node[below] {$ab^3$}; 
                \fill[] (7,-4) circle (8pt) node[below] {$ab^2c$}; 
                \fill[] (1,-2) circle (8pt) node[below] {$a^2bc$}; 
                \fill[] (-5,0) circle (8pt) node[below] {$a^3c$}; 
                \fill[] (-10,6) circle (8pt) node[above] {$a^3d$}; 
                \fill[] (-4,4) circle (8pt) node[right] {$a^2bd$}; 
                \fill[] (2,2) circle (8pt) node[left] {$ab^2d$}; 
                \fill[] (0,6) circle (8pt) node[above] {$a^2cd$};
                \fill[red] (6,4) circle (8pt) node[right] {$abcd$}; 
                \fill[] (11,-2) circle (8pt) node[right] {$abc^2$}; 
                \fill[] (5,0) circle (8pt) node[right] {$a^2c^2$}; 

			\end{tikzpicture}

   \columnbreak

   \begin{tikzpicture}[scale=0.12][line join = round, line cap = round] 
		      
			\draw[-, thick]   (25,0)   --   (10,18)   ; 
                \draw[-, thick]    (10,18)   --   (0,6) ;  
                \draw[-, thick,dashed]   (5,0)  --   (25,0)  ; 
			\draw[-, thick,dashed]   (15,0)  --  (20,6)  ; 
			\draw[-, thick,dashed]  (5,0)  --  (10,6)   ; 
                \draw[-, thick,dashed]  (10,6)  --  (15,12)   ; 
                \draw[-, thick,dashed]   (15,0)  --    (10,6)  ; 
                \draw[-, thick,dashed]   (10,6)  --   (5,12)  ; 
                \draw[-, thick,dashed]  (15,0)  --    (21,-2) ; 
                \draw[-, thick]   (21,-2)   --  (20,6)  ; 
                \draw[-, thick,dashed]   (10,6)   --  (16,4) ; 
                \draw[-, thick]  (16,4)   --   (15,12)  ; 
                \draw[-, thick]    (21,-2) --  (11,10)  ; 
                \draw[-, thick]   (25,0)   -- (21,-2) ; 

                \draw[-, thick]     (5,12) --   (11,10)    ; 
                \draw[-, thick]    (11,10)  --   (10,18)  ; 

                \draw[-, thick] (7,8)  (5,0) --   (11,-2)    ; 
                \draw[-, thick] (11,-2)    --  (21,-2)  ; 
                \draw[-, thick]  (16,4) --  (11,-2) ; 
                \draw[-, thick]   (5,0) --   (0,6) ; 
                \draw[-, thick]   (0,6) --  (6,4)--  (11,10) ; 
                \draw[-, thick]     (6,4)-- (11,-2)  ; 

                \fill[] (0,6) circle (8pt) node[left] {$a^2cd$};
                \fill[red] (6,4) circle (8pt) node[right] {$abcd$}; 
                \fill[] (11,-2) circle (8pt) node[below] {$abc^2$}; 
                \fill[] (5,0) circle (8pt) node[below] {$a^2c^2$}; 

                \fill[] (15,0) circle (8pt) node[below] {$ac^3$}; 
                \fill[] (10,6) circle (8pt) node[left] {$ac^2d$}; 
                \fill[] (5,12) circle (8pt) node[left] {$acd^2$}; 
                
                \fill[] (15,12) circle (8pt) node[right] {$c^2d^2$}; 

                \fill[] (25,0) circle (8pt) node[right] {$c^4$}; 
                
                \fill[] (21,-2) circle (8pt) node[below] {$bc^3$}; 

                \fill[] (10,18) circle (8pt) node[above] {$cd^3$}; 
                \fill[] (20,6) circle (8pt) node[right] {$c^3d$}; 

                \fill[] (11,10) circle (8pt) node[below] {$bcd^2$}; 

                \fill[] (16,4) circle (8pt) node[left] {$bc^2d$}; 
			\end{tikzpicture}
            
        \end{multicols}

        \begin{tikzpicture}[scale=0.12][line join = round, line cap = round] 
		      
			    \draw[-, thick]   (9,-8)  --   (21,-2)   ; 
                \draw[-, thick]   (21,-2)   --  (11,10) ;  
                \draw[-, thick]    (7,8)   --   (9,-8) ; 
				\draw[-, thick]   (8,0)  --   (13,-6); 
			    \draw[-, thick]  (7,8) --  (17,-4)  ; 
                \draw[-, thick]  (8,0)  --   (16,4) ; 
                \draw[-, thick]   (8,0)  --   (3,-6)   ; 
                \draw[-, thick,dashed]    (3,-6) --    (13,-6); 
                \draw[-, thick,dashed]   (12,2)  --   (7,-4) ; 
                \draw[-, thick,dashed]      (7,-4) --   (17,-4) ; 
                \draw[-, thick,dashed]   (3,-6)  --  (11,-2) ; 
                \draw[-, thick]   (9,-8)  --  (3,-6); 

                \draw[-, thick,dashed]   (16,4) --   (11,-2)   ; 
                \draw[-, thick,dashed]   (11,-2)  --  (21,-2) ; 

                \draw[-, thick] (7,8)   --  (2,2)   ; 
                \draw[-, thick] (2,2)   -- (3,-6)  ; 
                \draw[-, thick,dashed]  (7,-4) --  (2,2) ; 
                \draw[-, thick]  (7,8) --   (11,10) ; 
                \draw[-, thick,dashed]   (2,2) -- (6,4) --  (11,10) ; 
                \draw[-, thick,dashed]     (6,4) --  (11,-2) ; 

                \fill[] (3,-6) circle (8pt) node[below] {$ab^3$}; 
                \fill[] (7,-4) circle (8pt) node[below] {$ab^2c$}; 
                
                \fill[] (2,2) circle (8pt) node[left] {$ab^2d$}; 
                
                \fill[red] (6,4) circle (8pt) node[right] {$abcd$}; 
                \fill[] (11,-2) circle (8pt) node[below] {$abc^2$}; 

                \fill[] (9,-8) circle (8pt) node[below] {$b^4$}; 
                \fill[] (7,8) circle (8pt) node[above] {$b^2d^2$}; 

                \fill[] (17,-4) circle (8pt) node[below] {$b^2c^2$}; 

                \fill[] (8,0) circle (8pt) node[left] {$b^3d$}; 
                \fill[] (13,-6) circle (8pt) node[below] {$b^3c$}; 

                \fill[] (21,-2) circle (8pt) node[right] {$bc^3$}; 

                \fill[] (11,10) circle (8pt) node[above] {$bcd^2$}; 
                \fill[] (12,2) circle (8pt) node[right] {$b^2cd$}; 

                \fill[] (16,4) circle (8pt) node[right] {$bc^2d$}; 
			\end{tikzpicture}
       \caption{The complexes of boxes resolving $I_{1}$ (left), $I_{2}$ (bottom), $I_{3}$ (right), and $I_{4}$ (top) inside $(a,b,c,d)^{4}$.}
        \label{f:4varcomplexes}
   \end{figure}
   \end{tiny}

        \end{example}

    \begin{proposition}        
       The $I_{i}$, and the box complexes $X_{i}$ resolving them, satisfy the hypotheses of Lemma 3.4.  
    \end{proposition}
    \begin{proof}
        This is a special case of Lemma \ref{l:truedecomposition} and Proposition \ref{3.4}.
    \end{proof}

    \begin{definition}
        Set $\displaystyle X=\bigcup_{i=1}^{n}X_{i}$.
    \end{definition}
    
    \begin{theorem}
    \label{th:centralPinchedPower}
        The polytopal complex $X$ supports the minimal resolution of $I$.
    \end{theorem}

	\begin{example}
		\phantomsection\label{e:threevariablepicture-1}
		
		The complex in Figure \ref{f:3varcomplexes} supports the minimal  resolution of $(a,b,c)^3$.

        The complex obtained by gluing the four complexes in Figure \ref{f:4varcomplexes} supports the minimal resolution of $(a,b,c,d)^{4}$.

	\end{example}

We now modify the minimal resolution $X$ to describe a symmetric polytopal resolution of the ``pinched power'' ideal $\widehat{I}$ obtained by deleting $m=x_{1}\dots x_{n}$ from the generating set for $I$.  The plan is to simply remove from $X$ all faces containing $m$.  Naively, doing so will cause the $n$ facets adjoining $m$ to lose their boundaries; we fix the issue by merging these cells into a single polytope.

\begin{notation}
    Let $\widehat{I}$ be the ideal generated by all degree $n$ monomials except $m=x_{1}\dots x_{n}$.  We also use hats on our previous notation to denote the removal of $m$, to wit:
    \begin{itemize}
        \item $\widehat{I_{i}}$ is the $Q_{i}$-Borel ideal generated by all monomials of $I_{i}$ except $m$.
        \item $\widehat{X_{i}}$ is the complex of boxes resolving $\widehat{I_{i}}$.
    \end{itemize}
    We also define, for each $i$, the box $Y_{i}=[x_{i}]\times [x_{i},x_{i+1}]\times \dots \times [x_{i-2},x_{i-1}]$, the unique facet of $X_{i}$ containing $m$, and $Y=\bigcup Y_{i}$.  

    Finally, set $J_{i}$ equal to the ideal generated by the vertices of $Y_{i}$, $J=\sum J_{i}$, and $\widehat{J}$ the ideal generated by all monomials of $J$ except $m$.
\end{notation}

    \begin{example}\label{e:Y}
        In three variables, $Y_{1}$, $Y_{2}$, and $Y_{3}$ are rhombi and $Y$ is their union.  In four variables, $Y_{1}, \dots, Y_{4}$ are parallelepipeds and $Y$ is their union.

    \begin{figure}[hbt] 
    \begin{multicols}{2}

    \begin{tikzpicture}[scale=1][line join = round, line cap = round]
				
				\fill[rose]  (-1,1.5) --  (-2,3) --  (0,3) -- (1,1.5) --  (-1,1.5) ;
				\fill[sand]  (1,1.5) -- (0,3) -- (1,4.5) -- (2,3) -- (1,1.5);
				\fill[cyan] (-2,3) --(0,3) --   (1,4.5) -- (-1,4.5);
				
				\draw[-, thick]    (-1,1.5) --  (-2,3) --  (0,3) -- (1,1.5) --  (-1,1.5)  ;
				 
				\draw[-, thick]    (0,3) -- (1,4.5) --  (2,3) -- (1,1.5) ;

                \draw[-, thick]    (-2,3) --  (-1,4.5) --  (1,4.5)  ; 
				 
				\fill[] (-1,1.5) circle (2pt) node[left,below] {$a^2b$};
				
				\fill[] (1,1.5) circle (2pt) node[below] {$ab^2$}; 
				
				\fill[] (-2,3) circle (2pt) node[left] {$a^2c$};
				\fill[blue] (0,3) circle (2pt) node[right] {$abc$};
				\fill[] (2,3) circle (2pt) node[right] {$b^2c$};
				
				\fill[] (-1,4.5) circle (2pt) node[above] {$ac^2$};
				\fill[] (1,4.5) circle (2pt) node[above] {$bc^2$}; 
				
			\end{tikzpicture}

    \columnbreak

         \begin{tikzpicture}[scale=0.23][line join = round, line cap = round]

                \draw[-, dashed]   (6,4) -- (0,6)  ; 
                \draw[-, dashed] (0,6)   -- (-4,4)  ; 
                \draw[-, thick]   (-4,4) -- (2,2)  ; 
                \draw[-, dashed] (2,2)   -- (6,4)  ; 
                \draw[-, thick]   (-4,4) --   (1,-2) ; 
                \draw[-, thick]   (1,-2) -- (7,-4)   ; 
                \draw[-, thick]   (7,-4) --   (2,2) ; 
                \draw[-, thick]   (7,-4)  --  (11,-2)  ; 
                \draw[-, dashed]    (11,-2)  --   (6,4)  ; 
                \draw[-, dashed] (0,6)  --  (5,0)  ; 
                \draw[-, dashed] (1,-2)  --  (5,0)  ; 
                \draw[-, dashed] (11,-2)  --  (5,0)  ; 
                \draw[-, thick]  (1,10)  -- (5,12) --  (11,10) -- (7,8) -- (1,10)  ; 
                \draw[-, thick]   (-4,4)  --  (1,10)   ; 
                \draw[-, thick]   (2,2)   --  (7,8)   ; 
                \draw[-, thick]   (7,8)   --    (12,2) --  (7,-4)  ; 
                \draw[-, thick]    (11,10)  --   (16,4)   --   (11,-2)  ; 
                \draw[-, thick]      (12,2)--    (16,4)   ; 
                \draw[-, dashed]   (6,4) --    (11,10) ; 
                \draw[-, dashed]   (0,6) --   (5,12)   ; 
                \draw[-, dashed]    (5,12) --    (10,6)  -- (16,4); 
                \draw[-, dashed]   (10,6)   --    (5,0)    ; 

                \fill[red] (6,4) circle (8pt) node[right] {}; 
				 
                \fill[] (11,-2) circle (8pt) node[right] {$abc^2$}; 
                  
                \fill[] (7,8) circle (8pt) node[above] {$b^2d^2$}; 
                 
                \fill[] (5,12) circle (8pt) node[above] {$acd^2$}; 
                \fill[] (10,6) circle (8pt) node[right ] { }; 
                \fill[] (10,6.5) circle (0 pt) node[right ] {$ac^2d$}; 
                \fill[] (1,10) circle (8pt) node[left] {$abd^2$}; 
                 
                \fill[] (12,2) circle (8pt) node[left] { }; 
                \fill[] (11.7,2) circle (0pt) node[left] {$b^2cd$}; 
                \fill[] (16,4) circle (8pt) node[right] {$bc^2d$}; 
                \fill[] (11,10) circle (8pt) node[right] {$bcd^2$}; 
                \fill[] (2,2) circle (8pt) node[left] { }; 
                \fill[] (2,1.6) circle (0pt) node[left] {$ab^2d$}; 
                \fill[] (-4,4) circle (8pt) node[left] {$a^2bd$}; 

                \fill[] (0,6) circle (8pt) node[right] { }; 
                \fill[] (0.3,6.4) circle (0pt) node[right] {$a^2cd$ }; 
                \fill[] (5,0) circle (8pt) node[right] {}; 
                \fill[] (5.3,0.3) circle (0pt) node[right] {$a^2c^2$}; 
                 
                \fill[] (1,-2) circle (8pt) node[below] {$a^2bc$}; 
                 
                \fill[] (7,-4) circle (8pt) node[below] {$ab^2c$}; 
                 
			\end{tikzpicture}

    \end{multicols}
    
            \caption{$Y$ for the ideals $(a,b,c)^{3}$ (on the left) and $(a,b,c,d)^{4}$ (on the right).}
            \label{f:Y}
    \end{figure}
        
    \end{example}

    \begin{proposition}
        $Y$ supports the minimal resolution of $J$.
    \end{proposition}
    \begin{proof}
        This is a special case of Theorem \ref{th:3.17}.
    \end{proof}

    \begin{construction}
    \label{c:centralPinchedPower}
           Refer to the polytope obtained by removing $m$ from $Y$ as $\widehat{Y}$.  More precisely, consider the geometric realization of $Y$ attained by placing each vertex $\prod x_{i}^{e_{i}}$ at the point $(e_{1},\dots, e_{n})$ inside $\mathbb{R}^{n}$.  Let $\widehat{Y}$ be the convex hull of all these vertices.  Then $\widehat{Y}$ is an $(n-1)$ dimensional polytope (inside the hyperplane $\sum x_{i}=n$) with a single facet, which contains $m$ on its interior.  The lower-dimensional faces of $\widehat{Y}$ are precisely the faces of $Y$ which don't contain $m$.     
    \end{construction}

    \begin{example}
    In three variables, $\widehat{Y}$ is the hexagon on the left in Figure \ref{f:YHat}.  In four variables, $\widehat{Y}$ is the rhombic dodecahedron on the right.

    \begin{figure}[hbt]\label{f:YHat}
    \begin{multicols}{2}

    \begin{tikzpicture}[scale=1][line join = round, line cap = round]
				
				\fill[light-gray] (-2,3) -- (-1,4.5) -- (1,4.5) -- (2,3) -- (1,1.5) -- (-1,1.5) -- (-2,3);
				
				\draw[-, thick]    (-2,3) -- (-1,4.5) -- (1,4.5) -- (2,3) -- (1,1.5) -- (-1,1.5) -- (-2,3)  ; 
				 
				\fill[] (-1,1.5) circle (2pt) node[left,below] {$a^2b$};
				
				\fill[] (1,1.5) circle (2pt) node[below] {$ab^2$}; 
				
				\fill[] (-2,3) circle (2pt) node[left] {$a^2c$};
				
				\fill[] (2,3) circle (2pt) node[right] {$b^2c$};
				
				\fill[] (-1,4.5) circle (2pt) node[above] {$ac^2$};
				\fill[] (1,4.5) circle (2pt) node[above] {$bc^2$}; 
				
			\end{tikzpicture}

    \columnbreak

    \begin{tikzpicture}[scale=0.23][line join = round, line cap = round]

                \draw[-, dashed] (0,6)   -- (-4,4)  ; 
                \draw[-, thick]   (-4,4) -- (2,2)  ; 
                 
                \draw[-, thick]   (-4,4) --   (1,-2) ; 
                \draw[-, thick]   (1,-2) -- (7,-4)   ; 
                \draw[-, thick]   (7,-4) --   (2,2) ; 
                \draw[-, thick]   (7,-4)  --  (11,-2)  ; 
                
                \draw[-, dashed] (0,6)  --  (5,0)  ; 
                \draw[-, dashed] (1,-2)  --  (5,0)  ; 
                \draw[-, dashed] (11,-2)  --  (5,0)  ; 
                \draw[-, thick]  (1,10)  -- (5,12) --  (11,10) -- (7,8) -- (1,10)  ; 
                \draw[-, thick]   (-4,4)  --  (1,10)   ; 
                \draw[-, thick]   (2,2)   --  (7,8)   ; 
                \draw[-, thick]   (7,8)   --    (12,2) --  (7,-4)  ; 
                \draw[-, thick]    (11,10)  --   (16,4)   --   (11,-2)  ; 
                \draw[-, thick]      (12,2)--    (16,4)   ; 
                 
                \draw[-, dashed]   (0,6) --   (5,12)   ; 
                \draw[-, dashed]    (5,12) --    (10,6)  -- (16,4); 
                \draw[-, dashed]   (10,6)   --    (5,0)    ; 

                \fill[] (11,-2) circle (8pt) node[right] {$abc^2$}; 

                \fill[] (7,8) circle (8pt) node[above] {$b^2d^2$}; 
                 
                \fill[] (5,12) circle (8pt) node[above] {$acd^2$}; 
                \fill[] (10,6) circle (8pt) node[right ] { }; 
                \fill[] (10,6.5) circle (0 pt) node[right ] {$ac^2d$}; 
                \fill[] (1,10) circle (8pt) node[left] {$abd^2$}; 
                 
                \fill[] (12,2) circle (8pt) node[left] { }; 
                \fill[] (11.7,2) circle (0pt) node[left] {$b^2cd$}; 
                \fill[] (16,4) circle (8pt) node[right] {$bc^2d$}; 
                \fill[] (11,10) circle (8pt) node[right] {$bcd^2$}; 
                \fill[] (2,2) circle (8pt) node[left] { }; 
                \fill[] (2,1.6) circle (0pt) node[left] {$ab^2d$}; 
                \fill[] (-4,4) circle (8pt) node[left] {$a^2bd$}; 

                \fill[] (0,6) circle (8pt) node[right] { }; 
                \fill[] (0.3,6.4) circle (0pt) node[right] {$a^2cd$ }; 
                \fill[] (5,0) circle (8pt) node[right] {}; 
                \fill[] (5.3,0.3) circle (0pt) node[right] {$a^2c^2$}; 
                 
                \fill[] (1,-2) circle (8pt) node[below] {$a^2bc$}; 
                 
                \fill[] (7,-4) circle (8pt) node[below] {$ab^2c$}; 
                
			\end{tikzpicture}

    \end{multicols}
    \caption{The polytope $\widehat{Y}$, in three and four variables.}
    \end{figure}
     \end{example}

    \begin{proposition}
        $\widehat{Y}$ supports the minimal resolution of $\widehat{J}$.
    \end{proposition}
        \begin{proof}
            This is a special case of Lemma \ref{4.2}.
        \end{proof}

    \begin{construction}
        Let $\displaystyle \widehat{X}=\widehat{Y}\cup \bigcup_{i=1}^{n}\widehat{X_{i}}$ be the polytopal complex obtained from $X$ by replacing $Y$ with $\widehat{Y}$.
    \end{construction}

    \begin{theorem}
        $\widehat{X}$ supports the minimal resolution of $\widehat{I}$.
    \end{theorem}
    \begin{proof}
        This is a special case of Theorem \ref{t:thetheorem}.
    \end{proof}

\begin{example} 
    The complex in Figure \ref{f:3pinched} supports the minimal resolution of $\widehat{I}=(a,b,c)^{3} \setminus \{abc\}$ in $k[a,b,c]$.

    \begin{figure}[hbt]
    \centering

    \begin{tikzpicture}[scale=0.8][line join = round, line cap = round]
				
				\fill[rose]  (-3,1.5) -- (-1,1.5) -- (-2,3) -- (-3,1.5);
				\fill[sand]   (1,1.5) -- (3,1.5) -- (2,3) -- (1,1.5);
				\fill[cyan]  (-1,4.5) -- (1,4.5) -- (0,6)  --  (-1,4.5);
				\fill[light-gray] (-2,3) -- (-1,4.5) -- (1,4.5) -- (2,3) -- (1,1.5) -- (-1,1.5) -- (-2,3);
				
				\draw[-, thick]  (-3,1.5) -- (3,1.5) -- (0,6) -- (-3,1.5);
				\draw[-, thick]  (-2,3) -- (-1,1.5); 
				\draw[-, thick]  (-1,4.5) -- (1,4.5);
				\draw[-, thick]   (2,3) -- (1,1.5); 
				
				\fill[] (-3,1.5) circle (3pt) node[below] {$a^3$};
				\fill[] (-1,1.5) circle (3pt) node[left,below] {$a^2b$};
				
				\fill[] (1,1.5) circle (3pt) node[below] {$ab^2$};
				\fill[] (3,1.5) circle (3pt) node[below] {$b^3$};
				
				\fill[] (-2,3) circle (3pt) node[left] {$a^2c$}; 
				\fill[] (2,3) circle (3pt) node[right] {$b^2c$};
				
				\fill[] (-1,4.5) circle (3pt) node[left] {$ac^2$};
				\fill[] (1,4.5) circle (3pt) node[right] {$bc^2$};
				
				\fill[] (0,6) circle (3pt) node[above] {$c^3$};

			\end{tikzpicture}
            
    \caption{The minimal resolution of $\widehat{I}=(a,b,c)^{3}\smallsetminus \{abc\}$.}    
    \label{f:3pinched}
    \end{figure}

 The complex $\widehat{X}$, obtained by gluing the five complexes in Figure \ref{f:4pinched}, supports a minimal resolution of   $\widehat{I}=(a,b,c,d)^4 \setminus \{abcd\}$ in $k[a,b,c,d]$.

 \begin{tiny}
       \begin{figure}[hbt] 
       
        \begin{center}
            \begin{tikzpicture}[scale=0.15][line join = round, line cap = round] 
		      
			    \draw[-, thick]   (5,24)   --    (-10,6)  ; 
                \draw[-, thick]    (-10,6)   --  (2,2)  ;  
                \draw[-, thick]    (7,8) --   (5,24)    ; 
				\draw[-, thick]   (6,16)  --  (0,18) ; 
			    \draw[-, thick]   (7,8) --     (-5,12) ; 
                \draw[-, thick]  (6,16)  --  (-4,4)   ; 
                \draw[-, thick]  (6,16)  --   (10,18)  ; 
                \draw[-, thick,dashed]  (10,18)    --   (0,18) ; 
                \draw[-, thick,dashed]    (1,10)  --  (5,12) ; 
                \draw[-, thick,dashed]    (5,12) --   (-5,12)  ; 
                \draw[-, thick,dashed]   (10,18)  --  (0,6)  ; 
                \draw[-, thick]      (5,24) -- (10,18) ; 
                \draw[-, thick,dashed]   (-10,6)   -- (0,6)   ; 

                \draw[-, thick,dashed]  (-4,4)   --     (0,6)  ; 

                \draw[-, thick] (7,8)    (7,8) --   (11,10)   ; 
                \draw[-, thick] (2,2)  (11,10) --  (10,18)  ; 
                \draw[-, thick,dashed]  (5,12) --  (11,10) ; 
                \draw[-, thick]   (7,8) --  (2,2) ; 

                \fill[] (-10,6) circle (8pt) node[below] {$a^3d$}; 
                \fill[] (-4,4) circle (8pt) node[below] {$a^2bd$}; 
                \fill[] (2,2) circle (8pt) node[below] {$ab^2d$}; 
                \fill[] (0,6) circle (8pt) node[right] {$a^2cd$};
                 
                \fill[] (5,12) circle (8pt) node[below] {$acd^2$}; 
                \fill[] (0,18) circle (8pt) node[left] {$ad^3$}; 
                \fill[] (-5,12) circle (8pt) node[left] {$a^2d^2$}; 
                \fill[] (1,10) circle (8pt) node[left] {$abd^2$}; 

                \fill[] (7,8) circle (8pt) node[right] {$b^2d^2$}; 

                \fill[] (5,24) circle (8pt) node[above] {$d^4$}; 

                \fill[] (6,16) circle (8pt) node[right] {$bd^3$}; 

                \fill[] (10,18) circle (8pt) node[right] {$cd^3$}; 

                \fill[] (11,10) circle (8pt) node[right] {$bcd^2$}; 
                 
			\end{tikzpicture}
        \end{center}

        \begin{multicols}{3}
           \begin{tikzpicture}[scale=0.15][line join = round, line cap = round] 
		      
			    \draw[-, thick]  (-15,0)  --  (3,-6)   ; 
                \draw[-, thick]   (3,-6)  -- (11,-2) ;  
                \draw[-, thick,dashed]    (5,0)   --  (-15,0); 
				\draw[-, thick,dashed]  (-5,0)  --  (-9,-2) ; 
			    \draw[-, thick,dashed]   (5,0) --  (-3,-4) ; 
                \draw[-, thick,dashed]  (-5,0)  --  (7,-4) ; 
                \draw[-, thick,dashed] (-5,0)   -- (-10,6)    ; 
                \draw[-, thick]   (-10,6)  --   (-9,-2); 
                \draw[-, thick,dashed]  (1,-2)  --  (-4,4)  ; 
                \draw[-, thick]     (-4,4) --  (-3,-4) ; 
                \draw[-, thick]  (-10,6)  --  (2,2) ; 
                \draw[-, thick]  (-15,0)  -- (-10,6) ; 

                \draw[-, thick] (7,-4) -- (2,2)    ; 
                \draw[-, thick]  (2,2)  -- (3,-6)  ; 

                \draw[-, thick] (5,0)  -- (0,6)    ; 
                \draw[-, thick]  (0,6) -- (-10,6)  ; 
                \draw[-, thick]  (-4,4) --   (0,6)  ; 
                 \draw[-, thick]  (5,0) --   (11,-2)  ; 

				\fill[] (-15,0) circle (8pt) node[below] {$a^4$}; 
                \fill[] (-9,-2) circle (8pt) node[below] {$a^3b$}; 
                \fill[] (-3,-4) circle (8pt) node[below] {$a^2b^2$}; 
                \fill[] (3,-6) circle (8pt) node[below] {$ab^3$}; 
                \fill[] (7,-4) circle (8pt) node[below] {$ab^2c$}; 
                \fill[] (1,-2) circle (8pt) node[below] {$a^2bc$}; 
                \fill[] (-5,0) circle (8pt) node[below] {$a^3c$}; 
                \fill[] (-10,6) circle (8pt) node[above] {$a^3d$}; 
                \fill[] (-4,4) circle (8pt) node[right] {$a^2bd$}; 
                \fill[] (2,2) circle (8pt) node[left] {$ab^2d$}; 
                \fill[] (0,6) circle (8pt) node[above] {$a^2cd$};
                 
                \fill[] (11,-2) circle (8pt) node[below] {$abc^2$}; 
                \fill[] (5,0) circle (8pt) node[right] {$a^2c^2$}; 

			\end{tikzpicture}

   \columnbreak 

\begin{tikzpicture}[scale=0.15][line join = round, line cap = round]

                \draw[-, dashed] (0,6)   -- (-4,4)  ; 
                \draw[-, thick]   (-4,4) -- (2,2)  ; 
                 
                \draw[-, thick]   (-4,4) --   (1,-2) ; 
                \draw[-, thick]   (1,-2) -- (7,-4)   ; 
                \draw[-, thick]   (7,-4) --   (2,2) ; 
                \draw[-, thick]   (7,-4)  --  (11,-2)  ; 
                 
                \draw[-, dashed] (0,6)  --  (5,0)  ; 
                \draw[-, dashed] (1,-2)  --  (5,0)  ; 
                \draw[-, dashed] (11,-2)  --  (5,0)  ; 
                \draw[-, thick]  (1,10)  -- (5,12) --  (11,10) -- (7,8) -- (1,10)  ; 
                \draw[-, thick]   (-4,4)  --  (1,10)   ; 
                \draw[-, thick]   (2,2)   --  (7,8)   ; 
                \draw[-, thick]   (7,8)   --    (12,2) --  (7,-4)  ; 
                \draw[-, thick]    (11,10)  --   (16,4)   --   (11,-2)  ; 
                \draw[-, thick]      (12,2)--    (16,4)   ; 
                 
                \draw[-, dashed]   (0,6) --   (5,12)   ; 
                \draw[-, dashed]    (5,12) --    (10,6)  -- (16,4); 
                \draw[-, dashed]   (10,6)   --    (5,0)    ; 

                \fill[] (11,-2) circle (8pt) node[right] {$abc^2$}; 

                \fill[] (7,8) circle (8pt) node[above] {$b^2d^2$}; 
                 
                \fill[] (5,12) circle (8pt) node[above] {$acd^2$}; 
                \fill[] (10,6) circle (8pt) node[right ] { }; 
                \fill[] (10,6.5) circle (0 pt) node[right ] {$ac^2d$}; 
                \fill[] (1,10) circle (8pt) node[left] {$abd^2$}; 
                 
                \fill[] (12,2) circle (8pt) node[left] { }; 
                \fill[] (11.7,2) circle (0pt) node[left] {$b^2cd$}; 
                \fill[] (16,4) circle (8pt) node[right] {$bc^2d$}; 
                \fill[] (11,10) circle (8pt) node[right] {$bcd^2$}; 
                \fill[] (2,2) circle (8pt) node[left] { }; 
                \fill[] (2,1.6) circle (0pt) node[left] {$ab^2d$}; 
                \fill[] (-4,4) circle (8pt) node[left] {$a^2bd$}; 

                \fill[] (0,6) circle (8pt) node[right] { }; 
                \fill[] (0.3,6.4) circle (0pt) node[right] {$a^2cd$ }; 
                \fill[] (5,0) circle (8pt) node[right] {}; 
                \fill[] (5.3,0.3) circle (0pt) node[right] {$a^2c^2$}; 
                 
                \fill[] (1,-2) circle (8pt) node[below] {$a^2bc$}; 
                 
                \fill[] (7,-4) circle (8pt) node[below] {$ab^2c$}; 
                 
			\end{tikzpicture}

   \columnbreak

   \begin{tikzpicture}[scale=0.15][line join = round, line cap = round] 
		      
			    \draw[-, thick]   (25,0)   --   (10,18)   ; 
                \draw[-, thick]    (10,18)   --   (0,6) ;  
                \draw[-, thick,dashed]   (5,0)  --   (25,0)  ; 
				\draw[-, thick,dashed]   (15,0)  --  (20,6)  ; 
			    \draw[-, thick]  (5,0)  --  (10,6)   ; 
                \draw[-, thick,dashed]  (10,6)  --  (15,12)   ; 
                \draw[-, thick,dashed]   (15,0)  --    (10,6)  ; 
                \draw[-, thick]   (10,6)  --   (5,12)  ; 
                \draw[-, thick,dashed]  (15,0)  --    (21,-2) ; 
                \draw[-, thick]   (21,-2)   --  (20,6)  ; 
                \draw[-, thick]   (10,6)   --  (16,4) ; 
                \draw[-, thick]  (16,4)   --   (15,12)  ; 
                \draw[-, thick]    (21,-2) --  (11,10)  ; 
                \draw[-, thick]   (25,0)   -- (21,-2) ; 

                \draw[-, thick]     (5,12) --   (11,10)    ; 
                \draw[-, thick]    (11,10)  --   (10,18)  ; 

                \draw[-, thick] (7,8)  (5,0) --   (11,-2)    ; 
                \draw[-, thick] (11,-2)    --  (21,-2)  ; 
                \draw[-, thick]  (16,4) --  (11,-2) ; 
                \draw[-, thick]   (5,0) --   (0,6) ; 

                \fill[] (0,6) circle (8pt) node[left] {$a^2cd$};
                 
                \fill[] (11,-2) circle (8pt) node[below] {$abc^2$}; 
                \fill[] (5,0) circle (8pt) node[below] {$a^2c^2$}; 

                \fill[] (15,0) circle (8pt) node[below] {$ac^3$}; 
                \fill[] (10,6) circle (8pt) node[left] {$ac^2d$}; 
                \fill[] (5,12) circle (8pt) node[left] {$acd^2$}; 
                 
                \fill[] (15,12) circle (8pt) node[right] {$c^2d^2$}; 

                \fill[] (25,0) circle (8pt) node[below] {$c^4$}; 
                
                \fill[] (21,-2) circle (8pt) node[below] {$bc^3$}; 

                \fill[] (10,18) circle (8pt) node[above] {$cd^3$}; 
                \fill[] (20,6) circle (8pt) node[right] {$c^3d$}; 

                \fill[] (11,10) circle (8pt) node[below] {$bcd^2$}; 

                \fill[] (16,4) circle (8pt) node[left] {$bc^2d$}; 
			\end{tikzpicture}
        \end{multicols}
        \begin{tikzpicture}[scale=0.15][line join = round, line cap = round] 
		      
			    \draw[-, thick]   (9,-8)  --   (21,-2)   ; 
                \draw[-, thick]   (21,-2)   --  (11,10) ;  
                \draw[-, thick]    (7,8)   --   (9,-8) ; 
				\draw[-, thick]   (8,0)  --   (13,-6); 
			    \draw[-, thick]  (7,8) --  (17,-4)  ; 
                \draw[-, thick]  (8,0)  --   (16,4) ; 
                \draw[-, thick]   (8,0)  --   (3,-6)   ; 
                \draw[-, thick,dashed]    (3,-6) --    (13,-6); 
                \draw[-, thick,dashed]   (12,2)  --   (7,-4) ; 
                \draw[-, thick,dashed]      (7,-4) --   (17,-4) ; 
                \draw[-, thick,dashed]   (3,-6)  --  (11,-2) ; 
                \draw[-, thick]   (9,-8)  --  (3,-6); 

                \draw[-, thick,dashed]   (16,4) --   (11,-2)   ; 
                \draw[-, thick,dashed]   (11,-2)  --  (21,-2) ; 

                \draw[-, thick] (7,8)   --  (2,2)   ; 
                \draw[-, thick] (2,2)   -- (3,-6)  ; 
                \draw[-, thick,dashed]  (7,-4) --  (2,2) ; 
                \draw[-, thick]  (7,8) --   (11,10) ; 

                \fill[] (3,-6) circle (8pt) node[below] {$ab^3$}; 
                \fill[] (7,-4) circle (8pt) node[below] {$ab^2c$}; 
                 
                \fill[] (2,2) circle (8pt) node[left] {$ab^2d$}; 
               
                \fill[] (11,-2) circle (8pt) node[below] {$abc^2$}; 
                 
                \fill[] (9,-8) circle (8pt) node[below] {$b^4$}; 
                \fill[] (7,8) circle (8pt) node[above] {$b^2d^2$}; 

                \fill[] (17,-4) circle (8pt) node[below] {$b^2c^2$}; 

                \fill[] (8,0) circle (8pt) node[left] {$b^3d$}; 
                \fill[] (13,-6) circle (8pt) node[below] {$b^3c$}; 

                \fill[] (21,-2) circle (8pt) node[right] {$bc^3$}; 

                \fill[] (11,10) circle (8pt) node[above] {$bcd^2$}; 
                \fill[] (12,2) circle (8pt) node[right] {$b^2cd$}; 

                \fill[] (16,4) circle (8pt) node[right] {$bc^2d$}; 
			\end{tikzpicture}
       \caption{Gluing the five complexes gives us a symmetric minimal resolution of $(a,b,c,d)^4 \setminus \{abcd\}$ in $k[a,b,c,d]$.}\label{f:4pinched}
   \end{figure}
   \end{tiny}

 \end{example}


	\section{A cyclic resolution of $(x_1,\cdots,x_n)^d$} \label{sec5}
	
 Fix $m= x_1^{d_1} x_2^{d_2} \cdots x_n^{d_n}$ and put $d=d_{1}+\dots +d_{n}$.  We will build a polytopal resolution of $I = \mathfrak{m}^d = (x_1,\cdots,x_n)^d$ which is symmetric around $m$. 

    We shall see that the general shape of the resolution depends substantially on which variables divide $m$.  Intuitively, the complex of boxes arising from the ordering $Q_{i}:x_{i}<x_{i+1}<\dots<x_{i-1}$ will appear if and only if $x_{i-1}$ divides $m$.

        \begin{definition}
            Fix $m= x_1^{d_1} x_2^{d_2} \cdots x_n^{d_n}$. Denote ${\rm{supp}}(m) \colon  = \{x_i \colon d_i \not=0\}$ the support of $m$.  For convenience, we will abuse notation by setting $\displaystyle {\rm{supp}}(m) \colon = \{i \colon x_i \in {\rm{supp}}(m)\}$.
        \end{definition}

    \begin{notation}
        For each $i$, denote by $I_{i}$ the principal $Q_{i}$-Borel ideal generated by $m$, $I_{i}=\Borel_{Q_{i}}(m)$.  Let $X_{i}$ be the complex of boxes resolving $I_{i}$.
    \end{notation}

	\begin{lemma} 		\phantomsection \label{l:factors} 
                  For any $1\leq i \leq n$, $I_{i}$ is the cyclic product of primes

    \[
    I_{i}=\prod_{j=i}^{i-1} \mathbf{m}_{[i,j]}^{d_{j}}.
    \]          

        \end{lemma}

        \begin{proof}
          This is immediate from \cite[Proposition 2.7]{Sinefakopoulos2007}, which states that $\Borel(\prod x_{i}^{d_{i}})=\prod (\Borel(x_{i}))^{d_{i}}$.   
        \end{proof}

        \begin{example}
    Let $m=x_2^2 x_4^3  = x_1^0x_2^2x_3^0x_4^3$ and $I=(x_1,\cdots ,x_4)^5$ in $k[x_1,x_2,x_3,x_4]$. Then  
        \begin{align*}
        & I_1=  (x_1)^0(x_1,x_2)^2(x_1,x_2,x_3)^0(x_1,x_2,x_3,x_4)^3,\\
        & I_2 = (x_2)^2(x_2,x_3)^0(x_2,x_3,x_4)^3(x_2,x_3,x_4,x_1)^0,\\
        & I_3 = (x_3)^0(x_3,x_4)^3(x_3,x_4, x_1)^0(x_3,x_4,x_1, x_2)^2,\\
        & I_4 = (x_4)^3(x_4,x_1)^0(x_4,x_1, x_2)^2(x_4,x_1,x_2,x_3)^0.\\
    \end{align*} 
\end{example} 
                        
        \begin{lemma}\label{l:truedecomposition}
          We have  $I=I_1 + \cdots + I_{n}$.
        \end{lemma}

\begin{proof}
    It suffices to show that, for an arbitrary monomial $f$ of degree $d$, we have $f\in I_{i}$ for some $i$.
          
          Write $f=\prod x_{j}^{e_{j}}$, and set $\displaystyle \delta_{[r,s]}(f)=\sum_{j\in [r,s]}(e_{j}-d_{j})$.  Then $f\in I_{i}$ if and only if $\delta_{[i,s]}(f)\geq 0$ for all $s$ (including $s=i$, which is the statement that $x_{i}^{d_{i}}$ divides $f$).  
        
        Choose $t$ such that $\delta_{[1,t]}(f)$ is minimal.  If $\delta_{[1,t]}(f)\geq 0$, then $f\in I_{1}$.  If not, we claim that $f\in I_{t+1}$.  Indeed, setting $r=t+1$, we have
        \[
        \delta_{[r,s]}(f)=\left\{\begin{matrix}\delta_{[1,s]}(f)-\delta_{[1,t]}(f) & \text{if} & r\leq s\leq n\\ \delta_{[r,n]}(f)+\delta_{[1,s]}(f) & \text{if} & 1\leq s<r\end{matrix}\right.
        \]
        Observe that $\delta_{[1,t]}(f)+\delta_{[r,n]}(f)=\delta_{[1,n]}(f)=\deg(f)-d=0$.  Thus $\delta_{[r,n]}(f)=-\delta_{[1,t]}(f)$ and in either case we have $\delta_{[r,s]}(f)=\delta_{(1,s)}(f)-\delta_{[1,t]}(f)$, which is nonnegative by the choice of $t$.
\end{proof}

The following notation is convenient to describe the intersection of any collection of $I_i$.

\begin{notation}
    Let  $m=x_1^{d_1}\cdots x_n^{d_n}$ and $i_1 < \cdots< i_\ell$ in $\{1,\cdots,n\}$.
		For   each index $j \in \{i_1,\cdots,i_\ell\}$, denote by $\mathbf{p}_{j}$ the product of cyclic primes
  \begin{gather*}
      \mathbf{p}_{j}:= \mathbf{m}_{[i_{j},i_{j}]}^{d_{i_j}}\mathbf{m}_{[i_{j},i_{j}+1]}^{d_{i_{j}+1}}  \dots \mathbf{m}_{[i_{j},i_{j+1}-1]}^{d_{i_{j+1}-1}}.
  \end{gather*} 

Observe that  $\mathbf{p}_{j}=\Borel_{Q_{i_j}}(x_{i_j}^{d_{i_j}} \cdots {x}_{i_{j+1}-1}^{d_{i_{j+1}-1}})$ in $k[x_{i_j}, \cdots, {x}_{i_{j+1}-1}]$. Denote by $\Gamma_{[i_{j},i_{j+1}-1]}$ the box complex resolving  $\mathbf{p}_{j}$.
\end{notation}

\begin{example}
    Let $m=x_2^2x_3x_5^3x_8 = x_1^0x_2^2x_3x_4^0x_5^3x_6^0x_7^0x_8x_9^0$ and $I=(x_1,\cdots ,x_9)^7$ in $k[x_1,\cdots,x_9]$ and $(i_{1},i_{2},i_{3})=(2,4,8)$. Then  
    \begin{gather*}
        \mathbf{p}_1 = (x_2)^2(x_2,x_3),\\
        \mathbf{p}_2 = (x_4)^0(x_4,x_5)^3(x_4,x_5,x_6)^0(x_4,x_5,x_6,x_7)^0,\\
        \mathbf{p}_3 = (x_8)(x_8,x_9)^0(x_8,x_9,x_1)^0.
    \end{gather*} 

    Observe that  
    \begin{gather*}
        I_2 \cap I_4 \cap I_8 = \mathbf{p}_1 \mathbf{p}_2 \mathbf{p}_3, \\
         X_2 \cap X_4 \cap X_8 =\Gamma_{[2,3]}\times \Gamma_{[4,7]}\times \Gamma_{[8,1]}.
    \end{gather*} 

    We will prove this property generally in the following theorem.
\end{example}

	\begin{theorem}
		\phantomsection\label{3.4} 
		Let  $m=x_1^{d_1}\cdots x_n^{d_n}$ and $\{i_1,\dots,i_\ell\} \subset \{1,\dots,n\}$. Then:
                \begin{enumerate}[(1)]
                 \item We have
                \begin{gather*}
                    I_{i_1} \cap \cdots \cap I_{i_{\ell}} = \mathbf{p}_{{1}} \mathbf{p}_{{2}}\dots \mathbf{p}_{{\ell}}.
                \end{gather*}
                  
                \item The intersection of the $I_{j}$
                is minimally resolved by the  product of the box complexes resolving the $\mathbf{p}_{j}$.  That is,
                  \[
                  X_{i_{1},\dots, i_{\ell}}:=\Gamma_{[i_{1},i_{2}-1]}\times \Gamma_{[i_{2},i_{3}-1]}\times \dots \times \Gamma_{[i_{\ell},i_{1}-1]} ,
                  \] 
            gives a minimal resolution of $ I_{i_1} \cap \cdots \cap I_{i_{\ell}}$. 
            
            \item    The intersection of the $I_{j}$
            is minimally resolved by the intersection of the box complexes resolving the $\Borel_{Q_{i_{j}}}(m)$, \[X_{i_{1},\dots, i_{\ell}}=  X_{i_1} \cap \cdots \cap X_{i_\ell}.\]

	        \end{enumerate}

        \end{theorem}
	
	\begin{proof}

          To show that $I_{i_{1}}\cap \dots \cap I_{i_{\ell}}\subseteq \prod \mathbf{p}_{j}$, observe that $I_{i_{j}}\subset \mathbf{p}_{j}$ for all $j$ by Lemma \ref{l:factors}.  
          Since the $\mathbf{p}_{j}$ have disjoint support, we have $\bigcap \mathbf{p}_{j}=\prod \mathbf{p}_{j}$.  Thus 
          \[
            \bigcap_{j=1}^{\ell} I_{i_{j}} \subset \bigcap_{j=1}^{\ell} \mathbf{p}_{j}= \prod_{j=1}^{\ell} \mathbf{p}_{j}.
          \]

          To show that $\prod \mathbf{p}_{j}\subseteq I_{i_{1}}\cap \dots \cap I_{i_{\ell+1}}$, it suffices to show that, without loss of generality, $\prod \mathbf{p}_{j}\subseteq I_{i_{1}}$.  Observe that
          
          \begin{align*}
          \mathbf{p}_{j}&=\mathbf{m}_{[i_{j},i_{j}]}^{d_{i_j}}\mathbf{m}_{[i_{j},i_{j}+1]}^{d_{i_{j}+1}}  \dots \mathbf{m}_{[i_{j},i_{j+1}-1]}^{d_{i_{j+1}-1}}\\
            &\subseteq \mathbf{m}_{[i_{1},i_{j}]}^{d_{i_j}} \mathbf{m}_{[i_{1},i_{j}+1]}^{d_{i_{j}+1}}\dots\mathbf{m}_{[i_{1},i_{j+1}-1]}^{d_{i_{j+1}-1}}.\\
          \end{align*}
          Thus,
          \[\prod \mathbf{p}_{j} \subseteq \biggl[\mathbf{m}_{[i_{1},i_{1}]}^{d_{i_{1}}} \dots \mathbf{m}_{[i_{1},i_{2}-1]}^{d_{i_{2}-1}}\biggr]\biggl[\mathbf{m}_{[i_{1},i_{2}]}^{d_{i_{2}}} \dots \mathbf{m}_{[i_{1},i_{3}-1]}^{d_{i_{3}-1}}\biggr]\dots \biggl[\mathbf{m}_{[i_{1},i_{\ell}]}^{d_{i_{\ell}}} \dots \mathbf{m}_{[i_{1},i_{1}-1]}^{d_{i_{1}-1}}\biggr] = I_{i_{1}}\] as desired. This proves (1).

For (2), observe that the supports of the $\mathbf{p}_{i_{j}}$ are pairwise disjoint and induct on Lemma \ref{l:Achilleasdisjoint}; we conclude that $X_{i_{1},\dots,i_{\ell}}$ yields a minimal resolution.

For (3), set  $\mathcal{I}=\displaystyle \bigcap_{j=1}^{\ell}X_{i_j}$ and $\mathcal{P}=\Gamma_{[i_{1},i_{2}-1]} \times \dots \times \Gamma_{[i_{\ell,i_{1}-1}]}$.  We will show $\mathcal{I}=\mathcal{P}$.  

To show $\mathcal{P}\subseteq \mathcal{I}$, let $\Lambda$ be a facet of $\mathcal{P}$.  We may write $\Lambda=\prod_{j}\Gamma_{i_{j}}(f_{j})$, with each $f_{j}\in \mathbf{p}_{j}$.  Observe that the poset  on $\{x_{i_{j}},\dots,x_{i_{j+1}-1}\}$ defined by the covering relations $x_{i_{j}}<x_{i_{j}+1}<\dots < x_{i_{j+1}-1}$ is a sub-poset of $Q_{i_{j}}$ for all $j$.  Setting $f=\prod_{j} f_{j}$, we immediately have, for all $t\in \{1,\dots, \ell\}$,  $f\in \Borel_{Q_{i_{t}}}(m)$ and thus $\Gamma_{i_{j}}(f_{j})\subseteq \Gamma_{i_{t}}(f_{j})$.  We conclude $\Lambda\subseteq \Gamma_{i_{t}}(f)\in \Gamma_{[i_{t},i_{t+1}-1]}$.  Since $t$ is arbitrary, we have $\Lambda\in \mathcal{I}$ as desired.

\newcommand{\Zeta}{\mathrm{Z}}

Now we show $\mathcal{I}\subseteq \mathcal{P}$. 
It follows from (2) that $I_{i_1} \cap \cdots \cap I_{i_{\ell}} = \mathbf{p}_{{1}} \mathbf{p}_{{2}}\dots \mathbf{p}_{{\ell}} $ is an equigenerated ideal and its minimal generators are of degree $d$, so $G \left(I_{i_1} \cap \cdots \cap I_{i_{\ell}}\right) = \displaystyle \bigcap \limits_{j=1}^{\ell} G \left(I_{i_{j}}\right)$.  This implies that vertices of $\displaystyle \bigcap_{j=1}^{\ell}X_{i_j}$ are labeled with monomials in  $\displaystyle \bigcap_{j=1}^{\ell}G \left(I_{i_{j}}\right)$. Now we show that for each vertex $f$ of $\displaystyle \bigcap_{j=1}^{\ell}X_{i_j}$, $\displaystyle \bigcap_{j=1}^{\ell}\Gamma_{Q_{i_j}}(f)$ is contained in 
$\prod_{j=1}^{\ell}\Gamma_{[i_{j},i_{j}-1]}$. Since $f\in \mathbf{p}_{{1}} \mathbf{p}_{{2}}\dots \mathbf{p}_{{\ell}}$, it can be written as 
\begin{gather*}
    f = f_1 \cdots f_l,
\end{gather*}
where for each $j$, $f_j \in \Borel_{Q_{i_j}}(x_{i_j}^{d_{i_j}} \cdots {x}_{i_{j+1}-1}^{d_{i_{j+1}-1}})$. For each $j$, let $Z_j\colon = \Gamma_{Q_{i_j}}(f_j)  \subset \Gamma_{[i_{j},i_{j}-1]}$. Then 
\begin{align*}
    \bigcap_{j=1}^{\ell}\Gamma_{Q_{i_j}}(f)  = Z_1 \times \cdots \times Z_{\ell} 
\end{align*}
which is contained in $\Gamma_{[i_{1},i_{2}-1]}\times \Gamma_{[i_{2},i_{3}-1]}\times \dots \times \Gamma_{[i_{\ell},i_{1}-1]} $. This proves (3). 
        \end{proof}

        We have now shown that the box complexes resolving the $I_{i}$ satisfy the hypotheses of Lemma \ref{l:unionmanyboxes}.  We conclude that their union resolves $I$ and  obtain the following:

 \begin{theorem}\label{c:symmetricresolution}
     For any generator   $m$ of $I=\mathbf{m}^{d}$, the complex 
     \begin{align*}
         \displaystyle X = \bigcup_{i=1}^n X_{i}  
     \end{align*}
      supports a minimal resolution of $I$.
 \end{theorem}

 We now simplify the union in Theorem \ref{c:symmetricresolution}. We need a lemma.

\begin{lemma}
\label{l:truesum} 
Let $m= x_1^{d_1} x_2^{d_2} \cdots x_n^{d_n}$. 
\begin{enumerate}[(i)]
    \item If $d_i = 0$, then $I_{i} \supseteq I_{{i+1}}.$

    \item $\displaystyle I=\sum_{ u\in {\rm{supp}}(m)} I_{u+1}$.
\end{enumerate}
          
        \end{lemma}

 \begin{proof}
     For (i), assume that $d_i=0$. On one hand, 
     \begin{gather*}
         \mathbf{p}_1 = \mathbf{m}_{[i,i]}^{d_i} = \mathbf{m}_{[i,i]}^{0} = (1), \\
         \mathbf{p}_2 = \mathbf{m}_{[i+1,i+1]}^{d_{i+1}} \mathbf{m}_{[i+1,i+2]}^{d_{i+2}} \cdots  \mathbf{m}_{[i+1,i-1]}^{d_{i-1}}, 
     \end{gather*}
   so  it follows from Theorem \ref{3.4} that 
     \begin{gather*}
          I_i \cap I_{i+1} = \mathbf{p}_1 \mathbf{p}_2 = \mathbf{p}_2. 
     \end{gather*} 
    
    On the other hand, 
     \begin{align*}
             I_{i+1} = \mathbf{m}_{[i+1,i+1]}^{d_{i+1}} \mathbf{m}_{[i+1,i+2]}^{d_{i+2}} \cdots  \mathbf{m}_{[i+1,i-1]}^{d_{i-1}} \mathbf{m}_{[i+1,i]}^{d_{i}} = \mathbf{p}_2 \mathbf{m}_{[i+1,i]}^{0} = \mathbf{p}_2.
         \end{align*}
    
    Thus $I_i \cap I_{i+1} = I_{i+1}$, or equivalently, $I_{i} \supseteq I_{{i+1}}.$

         To prove (ii), assume that $ \displaystyle \supp (m) = \bigcup \limits_{t=1}^{\ell} [i_t,j_t]$ where $i_t \leq j_t$, $j_t + 1 < i_{t+1}$. It follows from (i)   that  $I_j \supset I_{j_t + 1}$, for all $j_t<j \leq i_{t+1}$. By Lemma \ref{l:truedecomposition}, 
         \begin{align*}
             I=I_1 + \cdots + I_{n} = \sum \limits_{t=1}^{\ell} \sum \limits_{u=i_{t}}^{j_{t}} I_{u+1} = \sum_{ u\in {\rm{supp}}(m)} I_{u+1}.
         \end{align*}

This completes the proof.
     
 \end{proof}

 Consequently, we immediately obtain the following corollary. 
 
 \begin{corollary}
   If $d_i=0$ and $j$ is the smallest index such that $j>i$ and $d_j\not=0$, then    $X_{i} \supseteq X_{{\ell}}$ for all $\ell \in [i,j]$. 
\end{corollary}

 \begin{theorem}\label{c:symmetricresolution1}
     For any generator   $m$ of $I=\mathbf{m}^{d}$, the complex 
     \begin{align*}
         \displaystyle X =  \bigcup_{u \in {\rm{supp}}(m)} X_{u+1}
     \end{align*}
      supports a minimal resolution of $I$.
 \end{theorem}

	\begin{example}
		\phantomsection\label{e:threevariablepicture}
		
		The complexes in Figure \ref{fig:3} are minimal resolutions of $(a,b,c)^4$, centered around $ab^{2}$ and around $b^{2}c^{2}$.  Observe that the complex on the left contains boxes of all three orientations, while the complex on the right has no yellow boxes (because $a$ does not divide $b^{2}c^{2}$).

\begin{figure}[hbt]
    \centering
    \begin{multicols}{2}

     \begin{tikzpicture}[scale=0.6][line join = round, line cap = round]

\fill[rose] (-4,0) --(-2,3) -- (0,3) -- (2, 0) --(-4,0);
\fill[sand] (4,0) --(1,4.5) -- (0,3) -- (2, 0) --(4,0);
  \fill[cyan] (0,6) --(1,4.5) -- (0,3) -- (-2,3) --(0,6);
  
\draw[-, thick]  (-4,0) -- (4, 0) -- (0,6) -- (-4,0);

\draw[-, thick]  (-3,1.5) -- (-2, 0);
\draw[-, thick]  (-2,3) -- (0, 0);
\draw[-, thick]  (-2,3) -- (0,3);
\draw[-, thick]  (2, 0) -- (0,3);
\draw[-, thick]  (-3,1.5) -- (1,1.5);

\draw[-, thick]  (2, 0) -- (3,1.5);
\draw[-, thick]  (1,1.5) -- (2,3);
\draw[-, thick]  (0,3) -- (1,4.5);

\draw[-, thick]  (-1,4.5) -- (1,4.5);

\fill[] (-4,0) circle (2pt) node[below] {$a^4$};
\fill[] (-2, 0) circle (2pt) node[below] {$a^3b$};
\fill[] (0, 0) circle (2pt) node[below] {$a^2b^2$};
\fill[] (2, 0) circle (2pt) node[below] {$ab^3$};
\fill[] (4, 0) circle (2pt) node[below] {$b^4$};

\fill[] (-3,1.5) circle (2pt) node[left] {$a^3c$};
\fill[] (-1,1.5) circle (2pt) node[left,below] {};
\fill[] (-1.3,1.5) circle (0 pt) node[below] {$a^2bc$};
\fill[] (1,1.5) circle (2pt) node[right] {$ab^2c$};
\fill[] (3,1.5) circle (2pt) node[right] {$b^3c$};

\fill[] (-2,3) circle (2pt) node[left] {$a^2c^2$};
\fill[red] (0,3) circle (2pt) node[right] {$abc^2$};
\fill[] (2,3) circle (2pt) node[right] {$b^2c^2$};

\fill[] (-1,4.5) circle (2pt) node[left] {$ac^3$};
\fill[] (1,4.5) circle (2pt) node[right] {$bc^3$};

\fill[] (0,6) circle (2pt) node[above] {$c^4$};

\end{tikzpicture}

        \columnbreak

        \begin{tikzpicture}[scale=0.6][line join = round, line cap = round]

\fill[rose] (-4,0) --(-2,3) -- (2,3) -- (4, 0) --(-4,0);
\fill[cyan] (0,6) --(2,3) -- (-2,3) --(0,6);

\draw[-, thick]  (-4,0) -- (4, 0) -- (0,6) -- (-4,0);
\draw[-, thick]  (-3,1.5) -- (3,1.5);
\draw[-, thick]  (-2,3) -- (2,3);

\draw[-, thick]  (-3,1.5) -- (-2, 0);
\draw[-, thick]  (-2,3) -- (0,0);
\draw[-, thick]  (0,3) -- (2, 0);

\draw[-, thick]  (-1,4.5) -- (1,4.5);
\draw[-, thick]  (1,4.5) -- (0,3);

\fill[] (-4,0) circle (2pt) node[below] {$a^4$};
\fill[] (-2, 0) circle (2pt) node[below] {$a^3b$};
\fill[] (0, 0) circle (2pt) node[below] {$a^2b^2$};
\fill[] (2, 0) circle (2pt) node[below] {$ab^3$};
\fill[] (4, 0) circle (2pt) node[below] {$b^4$};

\fill[] (-3,1.5) circle (2pt) node[left] {$a^3c$};
\fill[] (-1,1.5) circle (2pt) node[left,below] {};
\fill[] (-1.3,1.5) circle (0 pt) node[below] {$a^2bc$};
\fill[] (1,1.5) circle (2pt) node[left,above] {};
\fill[] (0.7,1.5) circle (0pt) node[below] {$ab^2c$};
\fill[] (3,1.5) circle (2pt) node[right] {$b^3c$};

\fill[] (-2,3) circle (2pt) node[left] {$a^2c^2$};
\fill[] (0,3) circle (2pt) node[right] {};
\fill[] (-0.3,3) circle (0 pt) node[below] {$abc^2$};
\fill[red] (2,3) circle (2pt) node[right] {$b^2c^2$};

\fill[] (-1,4.5) circle (2pt) node[left] {$ac^3$};
\fill[] (1,4.5) circle (2pt) node[right] {$bc^3$};

\fill[] (0,6) circle (2pt) node[above] {$c^4$};

\end{tikzpicture}
    \end{multicols}
    
    \caption{Minimal linear resolutions of $(a,b,c)^4$.}
    \label{fig:3}
\end{figure}

	\end{example}

\begin{example} 
The complex in Figure \ref{fig:abcd3} supports a minimal resolution of $I=(a,b,c,d)^3$ (centered around $bcd$).
    \begin{figure}[hbt]
        \centering
        \begin{tikzpicture}[scale=0.2][line join = round, line cap = round]
    
\draw[-, thick]  (-15,0) -- (-9,-2);
\draw[-, thick]  (-3,-4) -- (-9,-2);
\draw[-, thick]  (-3,-4) -- (3,-6);
\draw[-, thick]  (7,-4) -- (3,-6);
\draw[-, thick]  (7,-4) -- (11,-2);
\draw[-, thick]  (-15,0) -- (-10,6);
\draw[-, thick]  (-9,-2) -- (-10,6);
\draw[-, thick]  (-4,4) -- (-10,6);
\draw[-, thick]  (-4,4) -- (-3,-4);
\draw[-, thick]  (-4,4) -- (2,2);
\draw[-, thick]  (3,-6) -- (2,2);
\draw[-, thick]  (7,-4) -- (2,2);
\draw[-, thick]  (6,4) -- (2,2);
\draw[-, thick]  (6,4) -- (11,-2);
\draw[-, dashed]  (-15,0) -- (-5,0);
\draw[-, dashed]  (-5,0) -- (5,0);
\draw[-, dashed]  (5,0) -- (11,-2);
\draw[-, dashed]  (-5,0) -- (-10,6);
\draw[-, dashed]  (-5,0) -- (-9,-2);
\draw[-, dashed]  (-4,4) -- (1,-2);
\draw[-, dashed]  (-3,-4) -- (1,-2);
\draw[-, dashed]  (7,-4) -- (1,-2);
\draw[-, dashed]  (-5,0) -- (1,-2);
\draw[-, dashed]  (5,0) -- (1,-2);
\draw[-, dashed]  (-10,6) -- (0,6);
\draw[-, dashed]  (5,0) -- (0,6);
\draw[-, dashed]  (-4,4) -- (0,6);
\draw[-, dashed]  (0,6) -- (6,4);

\draw[-, thick]  (15,0) -- (11,-2);
\draw[-, thick]  (15,0) -- (10,6);
\draw[-, thick]  (5,12) -- (10,6);
\draw[-, thick]  (11,-2) -- (10,6);
\draw[-, thick]  (6,4) -- (5,12);
\draw[-, thick]  (1,10) -- (5,12);
\draw[-, thick]  (1,10) -- (2,2);
\draw[-, dashed]  (5,0) -- (15,0);
\draw[-, dashed]  (5,0) --  (10,6);
\draw[-, dashed]  (0,6) --  (5,12);

\draw[-, thick]  (5,12) -- (0,18);
\draw[-, thick]  (-5,12) -- (0,18);
\draw[-, thick]  (-5,12) -- (-10,6);
\draw[-, thick]  (0,18) --  (1,10);
\draw[-, thick]  (-5,12) --  (1,10);
\draw[-, thick]  (-4,4) --  (1,10);
\draw[-, dashed]  (-5,12) --  (5,12);

\fill[] (-15,0) circle (4pt) node[left] {$a^3$};
\fill[] (-9,-2) circle (4pt) node[below] {$a^2b$};
\fill[] (-3,-4) circle (4pt) node[below] {$ab^2$};
\fill[] (3,-6) circle (4pt) node[below] {$b^3$};
\fill[] (7,-4) circle (4pt) node[below] {$b^2c$};
\fill[] (1,-2) circle (4pt) node[below] {$abc$};
\fill[] (-5,0) circle (4pt) node[below] {$a^2c$};
\fill[] (-10,6) circle (4pt) node[left] {$a^2d$};

\fill[red] (6,4) circle (4pt) node[right] {$bcd$};
\fill[] (11,-2) circle (4pt) node[below] {$bc^2$};

\fill[] (15,0) circle (4pt) node[right] {$c^3$};
\fill[] (10,6) circle (4pt) node[right] {$c^2d$};
\fill[] (5,12) circle (4pt) node[right] {$cd^2$};
\fill[] (0,18) circle (4pt) node[above] {$d^3$};
\fill[] (-5,12) circle (4pt) node[left] {$ad^2$};

\end{tikzpicture}
        \caption{A minimal resolution of $(a,b,c,d)^3$.}
        \label{fig:abcd3}
    \end{figure}
\end{example}


       	\section{A minimal resolution of the ideal   obtained by removing $m= x_{1}^{d_{1}}\cdots x_{n}^{d_{n}}$ from the generators of $I= (x_1,\dots,x_n)^d$.}

    \label{sec6}

       In this section, we will use the polytopal resolution for $I=(x_{1},\dots, x_{n})^{d}$ established in Theorem \ref{c:symmetricresolution} to describe the minimal resolution of the pinched power ideal $\widehat{I}$ obtained by removing $m$.  
        The technique is to modify the resolution of $I$ (centered on $m$) obtained in the previous section by replacing all boxes containing $m$ with their (topological) union, a polytope that does not contain $m$ as a vertex.  Essentially the same operation appears in \cite{JenniferBiermann}, where Biermann uses it to construct a cellular complex supporting the resolution of the complement of an $n$-cycle.

        As in the previous section, the resolution of $\widehat{I}$ depends strongly on the support of $m$.  For example, if $m=x_{j}^{d}$ is a pure power, then $\widehat{I}$ is a Borel-with-holes ideal  \cite{CharalambousEvans} and so is resolved by a truncation of the Eliahou-Kervaire resolution (or by a truncation of the complex of boxes).

	\begin{notation}
		\phantomsection\label{4.1}
		We require a tremendous amount of notation.  The general philosophy is to adapt notation from the previous section, and add a hat when $m$ has been removed.  Fix a monomial $m$ of degree $d$, and write $m= x_{1}^{d_{1}}\cdots x_{n}^{d_{n}}$.  We denote:
            \begin{itemize}
            \item $I=(x_{1},\dots, x_{n})^{d}$, the $d^{\text{th}}$ power of the maximal ideal, and $X$, the polytopal complex supporting the resolution of $I$ centered at $m$, as constructed in the previous section. 
			\item $\widehat{I}$  the ideal obtained by removing $m$ from the generators of $I=(x_1,\dots,x_n)^d$. 
			\item $Y$ the star of $m$ in $X$, that is, the subcomplex of $X$ generated by all facets containing $m$.
            \item $J=\mathbf{m}_{1}^{d_{1}}\mathbf{m}_{2}^{d_{2}}\dots\mathbf{m}_{n}^{d_n}$, 
            the ideal supported by $Y$.  
			\item $\widehat{J}$ the ideal obtained by removing $m$ from the generators of $J$.

            \item $\widehat{X_{i}}$ the complex obtained by removing all faces containing $m$ from $X_{i}$.
            \item $Y_{i}=\Gamma_{i}(m)$ (so $Y=\bigcup Y_{i}$), the collection of boxes containing $m$ inside $X$.
            \item $\widehat{Y}$ the polytope obtained by forgetting the interior structure (i.e., the faces containing $m$) of $Y$.  That is, $\widehat{Y}$ has a single $(n-1)$-dimensional face, which is the (interior of the) union of the $Y_{i}$; for $d<n-1$ its $d$-dimensional faces are the $d$-dimensional faces of the $Y_{i}$ which don't contain $m$. 
            
            \item    $\widehat{I_i}$  the ideal obtained by removing $m$ from the generators of $I_i$. Then $\widehat{I_i}$ is a $Q_i$-Borel ideal and $\widehat{X_i}$ is the complex-of-boxes resolving $\widehat{I_{i}}.$  
            
            \item  $\widehat{J_i} $ the ideal obtained by removing $m$ from the generators of $J_i$, and $\widehat{Y_i}$ the complex obtained by deleting all faces of $Y_i$ containing $m$. 
            \item $\displaystyle \widehat{X}=\widehat{Y}\cup \bigcup_{i=1}^n\widehat{X_{i}}$, the complex obtained from $X$ by replacing $Y$ with $\widehat{Y}$.

		\end{itemize} 
	\end{notation}

 \begin{remark}
 We have $\displaystyle Y= \bigcup_{u\in {\rm{supp}}(m)}  Y_{u+1}$ and $\displaystyle \widehat{X}=  \widehat{Y}\cup \bigcup_{u\in {\rm{supp}}(m)}\widehat{X_{u+1}}$. 
 \end{remark}

 \begin{remark}  
 
 The defining half-spaces of $\widehat{Y}$ are

     \[
    \left\{\begin{matrix}
    x_{1}+\dots+x_{n}=d,\\
    d_{a}+\dots+d_{b}-1\leq x_{a}+\dots+x_{b}\leq d_{a}+\dots+d_{b}+1,\\
     0\leq x_{a}+\dots+x_{b}\leq d
     \end{matrix}
     \right\}
     \]
     (for all cyclic intervals $[a,b]$).
 
 \end{remark}

Observe that all $Y_{i}=\Gamma_{Q_i}(m)$ and $J_i$ satisfy the hypotheses of Lemma \ref{l:unionmanyboxes}, so their union resolves $J$.

	\begin{theorem}
		\phantomsection\label{th:3.17}
		The complex $Y$ supports the minimal resolution of $J$.
	\end{theorem}

	\begin{example}
		\phantomsection\label{example:abcd3-bcd}
		The complex in Figure \ref{fig:abcd3-bcd} supports a minimal resolution of\\ $J=\left({\color{red}bcd},abd,acd,abc,bd^2,cd^2,c^2d,b^2d, ac^2,bc^2,b^2c\right)$.
    \end{example}

		\begin{figure}[hbt] 
         \begin{center}
         \begin{tikzpicture}[scale=0.3][line join = round, line cap = round]

\draw[-, thick]  (-4,4) -- (1,10) --(5,12) -- (10,6) -- (11,-2) --  (7,-4) -- (1,-2) -- (-4,4); 
\draw[-, thick]  (-4,4) --  (2,2) -- (7,-4)  ; 
\draw[-, thick]  (1,10) --  (2,2) --  (6,4) --  (5,12); 
\draw[-, thick]    (6,4) --  (11,-2); 

\draw[-, dashed]  (-4,4) --  (0,6) --  (6,4); 
\draw[-, dashed]   (5,12) --  (0,6)  ; 
\draw[-, dashed]  (0,6) -- (5,0) -- (11,-2); 
\draw[-, dashed]  (1,-2) -- (5,0) -- (10,6); 

\fill[] (7,-4) circle (4pt) node[below] {$b^2c$}; 
\fill[] (1,-2) circle (4pt) node[below] {$abc$}; 

\fill[] (-4,4) circle (4pt) node[left] {$abd$}; 

\fill[] (2,2) circle (4pt) node[left] {};
\fill[] (1.5,2) circle (0pt) node[below] {$b^2d$};

\fill[] (0,6) circle (4pt) node[below] {$acd$}; 
\fill[red] (6,4) circle (4pt) node[right] {$bcd$}; 
\fill[] (11,-2) circle (4pt) node[below] {$bc^2$}; 

\fill[] (5,0) circle (4pt) node[right] {}; 
\fill[] (5.5,0.5) circle (0pt) node[right] {$ac^2$}; 

\fill[] (10,6) circle (4pt) node[right] {$c^2d$}; 
\fill[] (5,12) circle (4pt) node[above] {$cd^2$}; 

\fill[] (1,10) circle (4pt) node[above] {$bd^2$}; 

\end{tikzpicture}
			
		\end{center}
         \caption{A minimal resolution of $J=  (b,c)(c,d)(d,a,b)$, inside $I=(a,b,c,d)^3$. } 
         \label{fig:abcd3-bcd}
    \end{figure}

	\begin{lemma}
		\phantomsection\label{4.2}
		The complex $\widehat{Y}$ supports a minimal resolution of $\widehat{J}$.
	\end{lemma}

To prove Lemma \ref{4.2}, we will show that $\widehat{Y}_{\leq \alpha}$ is acyclic for all multidegrees $\alpha$. Indeed, we will make use of a lemma of Nagel and Reiner \cite{NagelReiner} to show that $\widehat{Y}_{\leq \alpha}$ is homotopic to $Y_{\alpha}$, which is acylic by Corollary \ref{th:3.17} and Theorem \ref{2.8}. 

    \begin{lemma}\cite[Lemma 6.4]{NagelReiner} 
		\phantomsection\label{2.25}
		Let $\mathcal{C}$ be a  polytopal complex and $v$ a vertex in $\mathcal{C}$ that lies in a uinique facet $P$, and assume that $P$ has strictly positive dimension.
		
		Then the vertex-induced subcomplex $\mathcal{C} \smallsetminus \{v\}$, obtained by deleting $v$ and all faces that contain it, is homotopy equivalent to  $\mathcal{C}$.
	\end{lemma}

	\begin{proof}[Proof of Lemma \ref{4.2}]

		Let $\alpha=x_1^{\alpha_1} \cdots x_n^{\alpha_n}$ be a monomial. We will show that $\widehat{Y}_{\leq \alpha}$ is   acyclic.     

            Assume that $m=x_{i_1}^{d_{i_1}}\cdots x_{i_s}^{d_{i_s}}$, where $i_1<\cdots<i_s$ and ${\rm{supp}}(m) = \{i_1,\dots,i_s\}$. 
		
		We first observe that if $m=x_{i_1}^{d_{i_1}}\cdots x_{i_s}^{d_{i_s}}$ does not divide $\alpha$, then $\widehat{Y}_{\leq \alpha}=Y_{\leq \alpha}$   which is
		acyclic because $Y$ supports a minimal resolution by Corollary \ref{th:3.17}.

        Next observe that, if $\alpha$ is a multiple of $(x_1\cdots x_n)m$, the least common multiple of the generators of $J$, then $\widehat{Y}_{\leq \alpha}=\widehat{Y}$, which is homeomorphic to $Y$.
        Thus, since $Y=Y_{\leq \alpha}$ and $Y$ supports a minimal resolution by Corollary \ref{th:3.17}, we have that $\widehat{Y}_{\leq \alpha}$ is acyclic.   

        Finally, assume that $\alpha$ is divisible by $m$ but not by $(x_1\cdots x_n)m$. Then, without loss of generality, we may assume that  
        \begin{gather*}
            \alpha = (x_{u_1} \cdots x_{u_t})(x_{v_1} \cdots x_{v_\ell}) m,
        \end{gather*}
        where $A= \{x_{u_1}, \dots, x_{u_t}\} \subset {\rm{supp}}(m)$, $B=\{x_{v_1}, \dots, x_{v_\ell}\} $ and ${\rm{supp}}(m)$ are disjoint,  and $A\cup B$ is a proper subset of $\{x_1,\dots,x_n\}$. Denote $B_{j} = B \cap \{x_{i_j}, \dots , x_{i_{j+1}}\}$. 

         If $A=\supp(m)$, let $R'=k[A,B]$ and $m'=\prod_{x_{i}\in A\cup B}x_{i}^{d_{i}}$.  Construct $Y'$ and $J'$ analogously to $Y$ and $J$, but for this smaller collection of variables.  Then $Y_{\leq \alpha}=m'\times Y'\cong Y'$ is acyclic because $Y'$ supports the minimal resolution of $J'$.

        Next, consider the case that $A$ is a proper subset of ${\rm{supp}}(m)$. We claim that $m$ is contained in a unique facet of $Y_{\leq \alpha}$.  By Remark \ref{r:squarefreeSameBox}, we may assume that $m$ is squarefree.  Let $\mathcal{P}_{m,\alpha}$ be the box
            \[\mathcal{P}_{m,\alpha} = \left(\prod_{i_j \in A} \left( \{x_{i_j}, x_{i_{j+1}}\} \cup B_j \right) \right) \times \left(\prod_{i_j \in {\rm{supp}}(m) \smallsetminus A} \left( \{ {x_{i_{j+1}}}\} \cup B_j \right) \right).
        \]

        We claim that $\mathcal{P}_{m,\alpha}$ is the unique facet of ${Y}_{\leq \alpha}$ containing $m$.  First, observe that $m\in \mathcal{P}_{m,\alpha}$ since one can choose $x_{i_j+1}$ from every interval in the product.  Next, observe that the multidegree of $\mathcal{P}_{m,\alpha}$ is $\gcd\left(\alpha,(x_1\cdots x_n)m\right)$, which is divisible by $\alpha$.  Thus $\mathcal{P}_{m,\alpha}$ is a face of ${Y}_{\leq \alpha}$.  

        It remains to show that every face of $Y_{\leq \alpha}$ containing $m$ is contained in $\mathcal{P}_{m,\alpha}$. The least common multiple of vertices of every such face is of the form 
        \begin{gather*}
            \alpha' = \left( \prod_{x_u \in A'} x_u\right)   \left( \prod_{x_v \in B'} x_v\right)  m,
        \end{gather*}
        where $A'\subset A$ and $B' \subset B$. Denote $B'_{j} = B' \cap \{x_{i_j}, \dots , x_{i_{j+1}}\}$.        
        Every such face must have the form 
        \[
        \left(\prod_{i_j \in A'} \left( \{x_{i_j}, x_{i_{j+1}}\} \cup B'_j \right) \right) \times \left(\prod_{i_j \in {\rm{supp}}(m) \smallsetminus A'} \left( \{ {x_{i_{j+1}}}\} \cup B'_j \right) \right),
        \]
        which implies that the face is contained in $\mathcal{P}_{m,\alpha}$.

    If $\alpha\neq m$, then $\mathcal{P}_{m,\alpha}$ has positive dimension and satisfies the hypotheses of Lemma \ref{2.25}, so $\widehat{Y}_{\leq \alpha}$ is acyclic.  If $\alpha=m$, then $\widehat{Y}_{\leq \alpha}$ is the empty complex, which is acyclic. 
	\end{proof}

\begin{example}
    The complex at the center in Figure \ref{f:4pinched}  supports a minimal resolution of 
    \begin{gather*}
        (a,b)(b,c)(c,d)(d,a)\smallsetminus \{abcd\}
    \end{gather*}  
    and  the complex in Figure \ref{f:bcd-bcd} supports a minimal resolution of 
    \begin{gather*}
        (b,c)(c,d)(d,a,b)\smallsetminus \{bcd\}
    \end{gather*}
    in $k[a,b,c,d]$.

     \begin{figure}[hbt] 
            \begin{center}
			\begin{tikzpicture}[scale=0.25][line join = round, line cap = round]

\draw[-, thick]  (-4,4) -- (1,10) --(5,12) -- (10,6) -- (11,-2) --  (7,-4) -- (1,-2) -- (-4,4); 
\draw[-, thick]  (-4,4) --  (2,2) -- (7,-4)  ; 
\draw[-, thick]  (1,10) --  (2,2)  ; 

\draw[-, dashed]  (-4,4) --  (0,6)  ; 
\draw[-, dashed]   (5,12) --  (0,6)  ; 
\draw[-, dashed]  (0,6) -- (5,0) -- (11,-2); 
\draw[-, dashed]  (1,-2) -- (5,0) -- (10,6); 

\fill[] (7,-4) circle (4pt) node[below] {$b^2c$}; 
\fill[] (1,-2) circle (4pt) node[below] {$abc$}; 

\fill[] (-4,4) circle (4pt) node[left] {$abd$}; 

\fill[] (2,2) circle (4pt) node[left] {};
\fill[] (1.5,2) circle (0pt) node[below] {$b^2d$};

\fill[] (0,6) circle (4pt) node[below] {$acd$}; 

\fill[] (11,-2) circle (4pt) node[below] {$bc^2$}; 

\fill[] (5,0) circle (4pt) node[right] {}; 
\fill[] (5.5,0.5) circle (0pt) node[right] {$ac^2$}; 

\fill[] (10,6) circle (4pt) node[right] {$c^2d$}; 
\fill[] (5,12) circle (4pt) node[above] {$cd^2$}; 

\fill[] (1,10) circle (4pt) node[above] {$bd^2$}; 

\end{tikzpicture}
	       \end{center}
         \caption{A minimal resolution of $(b,c)(c,d)(d,a,b)\smallsetminus \{bcd\}$, inside $I=(a,b,c,d)^3$ with $m=bcd$.}
         \label{f:bcd-bcd}
     \end{figure}
\end{example}

 We are now ready to prove that $\widehat{I}$ is minimally resolved by $\widehat{X}$.

\begin{theorem}\label{t:thetheorem}
The pinched power ideal $\widehat{I}$ has polytopal minimal resolution supported on the complex $\widehat{X}$.
\end{theorem}

\begin{proof}
Let $\alpha$ be given; we will verify that $\widehat{X}_{\leq \alpha}$ is acyclic.  As in the proof of Lemma \ref{4.2}, we show that in every case $\widehat{X}_{\leq \alpha}$ is homeomorphic to $X_{\leq \alpha}$.  

If $\alpha$ is not divisible by $m$, then $\widehat{X}_{\leq \alpha}=X_{\leq \alpha}$.  If $\alpha$ is divisible by $\deg(\widehat{Y})=(x_{1}\dots x_{n})m$, then $\widehat{X}_{\leq \alpha}$ contains $\widehat{Y}$ and $X_{\leq \alpha}$ contains $Y$, 
so the two complexes are homeomorphic. 

Otherwise, we may assume without loss of generality that $\alpha$ has full support.

In this case, following the proof of Lemma \ref{4.2}, $m$ is contained in a unique facet of $Y_{\leq \alpha}$ 
; since $m$ is not contained in any face of $X\smallsetminus Y$, it is thus contained in a unique facet of $X_{\leq \alpha}$.  By Lemma \ref{2.25}, $X_{\leq \alpha}$ is homotopic to $\widehat{X}_{\leq \alpha}$.  
\end{proof}

\begin{example} 
 The complex in Figure \ref{f:abcd3-bcd} supports a minimal resolution of  $(a,b,c,d)^3 \smallsetminus \{bcd\}$ in  $k[a,b,c,d]$.

     \begin{figure}[hbt] 
       
        \begin{center}
            \begin{tikzpicture}[scale=0.2][line join = round, line cap = round]
    
\draw[-, thick]  (-15,0) -- (-9,-2);
\draw[-, thick]  (-3,-4) -- (-9,-2);
\draw[-, thick]  (-3,-4) -- (3,-6);
\draw[-, thick]  (7,-4) -- (3,-6);
\draw[-, thick]  (7,-4) -- (11,-2);
\draw[-, thick]  (-15,0) -- (-10,6);
\draw[-, thick]  (-9,-2) -- (-10,6);
\draw[-, thick]  (-4,4) -- (-10,6);
\draw[-, thick]  (-4,4) -- (-3,-4);
\draw[-, thick]  (-4,4) -- (2,2);
\draw[-, thick]  (3,-6) -- (2,2);
\draw[-, thick]  (7,-4) -- (2,2);
 
\draw[-, dashed]  (-15,0) -- (-5,0);
\draw[-, dashed]  (-5,0) -- (5,0);
\draw[-, dashed]  (5,0) -- (11,-2);
\draw[-, dashed]  (-5,0) -- (-10,6);
\draw[-, dashed]  (-5,0) -- (-9,-2);
\draw[-, dashed]  (-4,4) -- (1,-2);
\draw[-, dashed]  (-3,-4) -- (1,-2);
\draw[-, dashed]  (7,-4) -- (1,-2);
\draw[-, dashed]  (-5,0) -- (1,-2);
\draw[-, dashed]  (5,0) -- (1,-2);
\draw[-, dashed]  (-10,6) -- (0,6);
\draw[-, dashed]  (5,0) -- (0,6);
\draw[-, dashed]  (-4,4) -- (0,6);

\draw[-, thick]  (15,0) -- (11,-2);
\draw[-, thick]  (15,0) -- (10,6);
\draw[-, thick]  (5,12) -- (10,6);
\draw[-, thick]  (11,-2) -- (10,6);
 
\draw[-, thick]  (1,10) -- (5,12);
\draw[-, thick]  (1,10) -- (2,2);
\draw[-, dashed]  (5,0) -- (15,0);
\draw[-, dashed]  (5,0) --  (10,6);
\draw[-, dashed]  (0,6) --  (5,12);

\draw[-, thick]  (5,12) -- (0,18);
\draw[-, thick]  (-5,12) -- (0,18);
\draw[-, thick]  (-5,12) -- (-10,6);
\draw[-, thick]  (0,18) --  (1,10);
\draw[-, thick]  (-5,12) --  (1,10);
\draw[-, thick]  (-4,4) --  (1,10);
\draw[-, dashed]  (-5,12) --  (5,12);

\fill[] (-15,0) circle (6pt) node[left] {$a^3$};
\fill[] (-9,-2) circle (6pt) node[below] {$a^2b$};
\fill[] (-3,-4) circle (6pt) node[below] {$ab^2$};
\fill[] (3,-6) circle (6pt) node[below] {$b^3$};
\fill[] (7,-4) circle (6pt) node[below] {$b^2c$};
\fill[] (1,-2) circle (6pt) node[below] {$abc$};
\fill[] (-5,0) circle (6pt) node[below] {$a^2c$};
\fill[] (-10,6) circle (6pt) node[left] {$a^2d$};

\fill[] (11,-2) circle (6pt) node[below] {$bc^2$};

\fill[] (15,0) circle (6pt) node[right] {$c^3$};
\fill[] (10,6) circle (6pt) node[right] {$c^2d$};
\fill[] (5,12) circle (6pt) node[right] {$cd^2$};
\fill[] (0,18) circle (6pt) node[above] {$d^3$};
\fill[] (-5,12) circle (6pt) node[left] {$ad^2$};

\end{tikzpicture}
       \caption{A minimal resolution of $\widehat{I}= (a,b,c,d)^3 \smallsetminus\{bcd\}$ in $k[a,b,c,d]$.}\label{f:abcd3-bcd}
       \end{center}
   \end{figure} 
 \end{example}

\section{Betti numbers}\label{sec7}

In this section, we will use the polytopal structure of the resolution in Theorem \ref{t:thetheorem} to compute the graded Betti numbers of $\widehat{I}$.
 
From the construction in Theorem \ref{t:thetheorem}, the difference between the minimal resolutions of $I$ and $\widehat{I}$ is the same as the difference between the resolutions of $J$ and $\widehat{J}$, namely the faces containing $m$ in $Y$.  Thus, since the Betti numbers of $I$ are known (see, for example, \cite{FranciscoMerminSchweigBorelgenerators}*{Example 6.5}), computing the Betti numbers of $\widehat{I}$ is equivalent to understanding the combinatorics of $Y$ and $\widehat{Y}$.  Doing so is a tremendously long, and not particularly enlightening, exercise.  Instead, we will take advantage of the mapping cone construction.

\begin{theorem}\label{t:mapping cone}
    Let $I=(x_{1},\dots, x_{n})^{d}$ and $m$ be a degree $d$ monomial which is not a pure power.  Then there is a short exact sequence
    \[
    0\to \frac{S}{\widehat{I}:m}(m^{-1})\xrightarrow{\times m}\frac{S}{\widehat{I}}\to \frac{S}{I}\to 0.
    \]
    (Here the $(m^{-1})$ represents a Serre twist of multidegree $m$.)  Furthermore, $\frac{S}{\widehat{I}:m}\cong \frac{S}{(x_{1},\dots, x_{n})}$ is minimally resolved by the Koszul complex.
\end{theorem}

The minimal resolution of $\frac{S}{I}$ is very well known; see for example \cite{FranciscoMerminSchweigBorelgenerators}*{Example 6.5} or \cite{BR}.
\begin{theorem}\label{t:VeroneseResolution}
The minimal resolution of $\frac{S}{I}$ is linear.  The multigraded Betti numbers are $\beta_{0,0}(\frac{S}{I})=1$, $\beta_{i+1,i+d}(\frac{S}{I})=\sum_{j=1}^{d}\binom{j+d-2}{d-1}\binom{j-1}{i}$ for $i<n$, and $\beta_{p,q}(\frac{S}{I})=0$ for all other $p,q$.     
\end{theorem}

$\frac{S}{I}$ is (non-minimally) resolved by the mapping cone resolution arising from the Koszul complex on $(x_{1},\dots, x_{n})$ and the resolution of $\frac{S}{\widehat{I}}$.  Since we know the two resolutions on the outside, understanding the resolution of $\frac{S}{\widehat{I}}$ is equivalent to understanding the cancellation in the mapping cone.

\begin{proposition}\label{p:cancellation}
    The cancellation in the mapping cone is precisely the non-linear strand in the resolution of $\frac{S}{\widehat{I}}$.  That is:
    \begin{enumerate}
        \item The linear strand of $\frac{S}{\widehat{I}}$ is not cancelled in the mapping cone.
        \item If $(p,q)\neq (i+1,i+d)$, and $\beta_{p,q}(\frac{S}{\widehat{I}})\neq 0$, then all generators for $\beta_{p,q}(\frac{S}{\widehat{I}})$ are cancelled in the mapping cone.
    \end{enumerate}
\end{proposition}

\begin{proof}
    For (1), observe that $\Tor_{i+1}(\frac{S}{\widehat{I}:m}(m^{-1}),k)$ is concentrated in degree $i+d+1$, while the linear strand of $\Tor_{i+1}(\frac{S}{\widehat{I}},k)$ is concentrated in degree $i+d$. Since the degrees are unequal, there can be no cancellation.

    For (2), observe that $\Tor_{i+1}(\frac{S}{I},k)$ is concentrated in degree $i+d$.  Thus any generators for $\Tor_{i+1}(\frac{S}{\widehat{I}},k)$ in different degrees must cancel. 
\end{proof}

To count the non-linear strand of $\widehat{I}$, it suffices to determine the non-linear faces of $\widehat{Y}$.  The Betti number $\beta_{p,p+d}$ is counted by the nonlinear faces of dimension $p$.

\begin{proposition}\label{p:nonLinearFaces}
    The number and dimension of the non-linear faces of $\widehat{Y}$ depends on the cardinality of $\supp(m)$.  In particular,
    \begin{enumerate}
        \item If $|\supp(m)|=1$, then $\widehat{Y}$ has no non-linear faces.
        \item If $|\supp(m)|=s\neq 1$, then the non-linear faces of $\widehat{Y}$ of dimension $p$ are in bijection with the subsets of $\{1,\dots,n\}\smallsetminus \supp(m)$ with  cardinality $n-1-p$.  Each such face has degree $d+p+1$.
    \end{enumerate}
\end{proposition}

\begin{proof}
    If $|\supp(m)|=1$, then, without loss of generality, $m=x_{1}^{d}$ and $\widehat{Y}=[x_{1}^{d}]\times \{x_{2},\dots, x_{n}\}$.

    Otherwise, the faces of $\widehat{Y}$ which are not faces of $Y$ are precisely the intersections of (the geometric realization of) $Y$ with the coordinate hyperplane $\langle x_{j}=0: j\in A\rangle$ for each subset $A$ such that  $A\cap \supp(m)=\varnothing$.  Each such intersection has dimension $n-1-|A|$ and multidegree $m\prod_{j\not\in A}x_{j}$.  
\end{proof}

\begin{corollary}
    If $|\supp(m)|=s>1$, then the graded Betti numbers of $\frac{S}{\widehat{I}}$ are given by:
\begin{align*}\label{c:BettiNumbers}
    \beta_{0,0}(\frac{S}{\widehat{I}})&=1\\
    \beta_{p,d+p-1}(\frac{S}{\widehat{I}})&=\beta_{p,d+p-1}(\frac{S}{I})-\binom{n}{p}+\binom{n-s}{n-p}\\
    \beta_{p,d+p}(\frac{S}{\widehat{I}})&=\binom{n-s}{n-p-1}\\
    \beta_{p,q}(\frac{S}{\widehat{I}})&=0&\text{otherwise}.
\end{align*}
The Poincar\'{e} polynomial $P_{\frac{S}{\widehat{I}}}(t,u)=\sum \beta_{p,q}(\frac{S}{\widehat{I}})t^{p}u^{q}$ is given by 
\[
P_{\frac{S}{\widehat{I}}}(t,u)=P_{\frac{S}{I}}(t,u) - tu^{d}(1+tu)^{n} + (1+t)t^{s-1}u^{s+d}(1+tu)^{n-s}.
\]
\end{corollary}
\begin{proof}
    The Betti numbers of the non-minimal mapping cone resolution for $\frac{S}{I}$ are obtained by adding the corresponding Betti numbers for $\frac{S}{\widehat{I}}$ and $\frac{S}{\widehat{I}:m}(m^{-1})$.  The Betti numbers for the minimal resolution are then obtained by subtracting the cancellations.  By Proposition \ref{p:cancellation}, these cancellations affect $\beta_{p,d+p}(\frac{S}{\widehat{I}})$ and $\beta_{p+1,d+p}(\frac{S}{\widehat{I}:m})(m^{-1})$.  The size of the cancellation is given by Proposition \ref{p:nonLinearFaces}.  Rearranging the algebra yields formulas for $\beta(\frac{S}{\widehat{I}})$.
\end{proof}

\begin{example}
    Let $I=(a,b,c,d)^{4}$.  The Betti table for $\frac{S}{I}$ is 
    \[
    \begin{matrix}
 & 0 & 1 & 2 & 3 & 4\\
\text{total:} & 1 & 35 & 84 & 70 & 20\\
0: & 1 & . & . & . & .\\
1: & . & . & . & . & .\\
2: & . & . & . & . & .\\
3: & . & 35 & 84 & 70 & 20
\end{matrix}.
    \]
The Betti tables for $\frac{S}{\widehat{I}}$ for various choices of $m$ are given below.

\begin{tabular}{|c|c|c|}\hline
    $m$ & Betti table for $\frac{S}{\widehat{I}}$ &Difference\\\hline
    $a^{4}$ & $\begin{matrix}
 & 0 & 1 & 2 & 3 & 4\\
\text{total:} & 1 & 34 & 81 & 67 & 19\\
0: & 1 & . & . & . & .\\
1: & . & . & . & . & .\\
2: & . & . & . & . & .\\
3: & . & 34 & 81 & 67 & 19
\end{matrix}$ & $\begin{matrix}
&0 & 1 & 2 & 3 & 4\\
\text{total:}&. & -1 & -3 & -3 & -1\\
0: & . & . & . & . & .\\
1: & . & . & . & . & .\\
2: & . & . & . & . & .\\
3: & . & -1 & -3 & -3 & -1
\end{matrix}$\\\hline
$a^{3}b$ & $\begin{matrix}
 & 0 & 1 & 2 & 3 & 4\\
\text{total:} & 1 & 34 & 81 & 67 & 19\\
0: & 1 & . & . & . & .\\
1: & . & . & . & . & .\\
2: & . & . & . & . & .\\
3: & . & 34 & 80 & 65 & 18\\
4: & . & . & 1 & 2 & 1
\end{matrix}$ & $\begin{matrix}
&0 & 1 & 2 & 3 & 4\\
\text{total:}&. & -1 & -3 & -3 & -1\\
0: & . & . & . & . & .\\
1: & . & . & . & . & .\\
2: & . & . & . & . & .\\
3: & . & -1 & -4 & -5 & -2\\
4: & . & . & 1 & 2 & 1
\end{matrix}$\\\hline
$a^{2}bc$  & $\begin{matrix}
 & 0 & 1 & 2 & 3 & 4\\
\text{total:} & 1 & 34 & 80 & 65 & 18\\
0: & 1 & . & . & . & .\\
1: & . & . & . & . & .\\
2: & . & . & . & . & .\\
3: & . & 34 & 80 & 64 & 17\\
4: & . & . & . & 1 & 1
\end{matrix}$ & $\begin{matrix}
&0 & 1 & 2 & 3 & 4\\
\text{total:}&. & -1 & -4 & -5 & -2\\
0: & . & . & . & . & .\\
1: & . & . & . & . & .\\
2: & . & . & . & . & .\\
3:&. & -1 & -4 & -6 & -3\\
4:&. & . & . & 1 & 1
\end{matrix}$\\\hline
$abcd$ & $\begin{matrix}
 & 0 & 1 & 2 & 3 & 4\\
\text{total:} & 1 & 34 & 80 & 65 & 18\\
0: & 1 & . & . & . & .\\
1: & . & . & . & . & .\\
2: & . & . & . & . & .\\
3: & . & 34 & 80 & 64 & 16\\
4: & . & . & . & . & 1
\end{matrix}$ & $\begin{matrix}
&0 & 1 & 2 & 3 & 4\\
\text{total:}&. & -1 & -4 & -6 & -3\\
0: & . & . & . & . & .\\
1: & . & . & . & . & .\\
2: & . & . & . & . & .\\
3:&. & -1 & -4 & -6 & -4\\
4:&. & . & . & . & 1
\end{matrix}$\\\hline
\end{tabular}
    
\end{example}



\section{Further questions}\label{sec8}

Two natural questions arise when hoping to expand the work in this paper.

\begin{question}\label{q:DeleteMore}
    What happens to the resolution when we delete more than one monomial?
    In particular:
    \begin{enumerate}
        \item If we delete several monomials that are adjacent (via Borel moves)?
        \item If we delete monomials that are widely separated?
    \end{enumerate}
\end{question}

\begin{question}\label{q:BoxResolutions}
    Are there other nice combinatorial structures on the resolution of $I$ (or of $\widehat{I}$)?
\end{question}

A partial answer to Question \ref{q:DeleteMore}(1) is found in \cite{DE}, where Dao and Eisenbud classify the monomial ideals with ``almost-linear'' resolutions.  These ideals are characterized by the deletion of inverted simplices $\mathcal{S}_{d,f}=\{m:m \text{ divides } f\}$ from the generators of $\mathbf{m}^{d}$, for various monomials $f$ of degree at least $d$.  (It is necessary that all $\mathcal{S}_{d}$ be pairwise disjoint, and that, if $\deg(f)=d+t$, $x_{i}^{t}$ must divide $f$ for all $i$.)

  A similar intuition to our construction of $\widehat{X}$ works with at least the simplest cases of these almost-linear ideals.  It is tedious but straightforward to verify that, in three variables, if only one inverted simplex $\mathcal{S}{d,f}$ is removed, then replacing $S_{d,f}$ with a larger polygon and tiling the rest of the dilated simplex with boxes as in Figure \ref{f:7.2} yields a minimal resolution.  (Note that the grey ``hexagon'', which has the deleted monomials on its interior, is not actually a hexagon; for example, in Figure \ref{f:7.2}, it has nine vertices and edges.)

We speculate that a similar construction will work in general for these ideals.

\begin{question}\label{q:almostLinear}
Let $I$ be an almost-linear ideal, built by deleting several disjoint simplices $\mathcal{S}_{d,f}$.  For each $\mathcal{S}_{d,f}$, let $\mathcal{T}_{d,f}$ be the smallest polytope with lattice points as vertices and containing every point of $\mathcal{S}_{d,f}$ on its interior.  Under what conditions on the collection of $\mathcal{S}_{d,f}$ is $I$ minimally resolved by a polytopal complex consisting of the $\mathcal{T}_{d,f}$ and oriented boxes $\Gamma_{i}(m)$?
\end{question}

\begin{figure}[hbt]
\begin{multicols}{2}
\centering
    \begin{tikzpicture}[scale=0.3][line join = round, line cap = round]
        \fill[rose]   (0,0) -- (3,6) -- (7,6) -- (10,0) -- (0,0);
        \fill[sand]    (10,0) -- (8,4) -- (11,10) -- (16, 0) -- (10,0);
        \fill[cyan]     (3,6) -- (9,6) -- (11,10) -- (8,16) -- (3,6);
 
        \draw[-, thick] (0,0) -- (16, 0) -- (8,16) -- (0,0);

        \draw[-, thick] (8,0) -- (5,6);
        \draw[-, thick] (6,0)--(3,6) -- (7,6) -- (10,0);
        \draw[-, thick] (4,0)--(2,4) -- (8,4);
        \draw[-, thick] (2,0) -- (1,2) -- (9,2);
 
        \draw[-, thick] (8,4) -- (11,10);
        \draw[-, thick] (15,2) -- (14,0) -- (10,8);
        \draw[-, thick] (14,4) -- (12,0) -- (9,6);
        \draw[-, thick] (10,0) -- (13,6) ;
        \draw[-, thick] (9,2) -- (12,8) ; 
        
        \draw[-, thick] (7,6) -- (9,6);
        \draw[-, thick] (7,14) -- (9,14) -- (5,6);
        \draw[-, thick] (6,12) -- (10,12) -- (7,6);
        \draw[-, thick] (5,10) -- (11,10) ;
        \draw[-, thick] (4,8) -- (10,8) ;
 
        \fill[] (0,0) circle (4pt) node[below] {$a^8$};
        \fill[] (2,0) circle (4pt) node[below] { };
        \fill[] (4,0) circle (4pt) node[below] { };
        \fill[] (6,0) circle (4pt) node[below] { };
        \fill[] (8,0) circle (4pt) node[below] { };
        \fill[] (10,0) circle (4pt) node[below] { };
        \fill[] (12,0) circle (4pt) node[below] { };
        \fill[] (14, 0) circle (4pt) ;
        \fill[] (16, 0) circle (4pt) node[below] {$b^8$};

        \fill[] (1,2) circle (4pt) node[below] { };
        \fill[] (3,2) circle (4pt) node[below] { };
        \fill[] (5,2) circle (4pt) node[below] { };
        \fill[] (7,2) circle (4pt) node[below] { };
        \fill[] (9,2) circle (4pt) node[below] { };
        \fill[] (11,2) circle (4pt) node[below] { };
        \fill[] (13,2) circle (4pt) node[below] { };
        \fill[] (15,2) circle (4pt) node[below] { };

        \fill[] (2,4) circle (4pt) node[below] { };
        \fill[] (4,4) circle (4pt) node[below] { };
        \fill[] (6,4) circle (4pt) node[below] { };
        \fill[red] (8,4) circle (6pt) node[below] { };
        \fill[] (10,4) circle (4pt) node[below] { };
        \fill[] (12,4) circle (4pt) node[below] { };
        \fill[] (14,4) circle (4pt) node[below] { };
 
        \fill[] (3,6) circle (4pt) node[below] { };
        \fill[] (5,6) circle (4pt) node[below] { };
        \fill[red] (7,6) circle (6pt) node[below] { };
        \fill[red] (9,6) circle ( 6pt) node[below] { };
        \fill[] (11,6) circle (4pt) node[below] { }; 
        \fill[] (13,6) circle (4pt) node[below] { };

        \fill[] (4,8) circle (4pt) node[below] { };
        \fill[] (6,8) circle (4pt) node[below] { };
        \fill[] (8,8) circle (4pt) node[below] { };
        \fill[] (10,8) circle (4pt) node[below] { }; 
        \fill[] (12,8) circle (4pt) node[below] { }; 

        \fill[] (5,10) circle (4pt) node[below] { };
        \fill[] (7,10) circle (4pt) node[below] { };
        \fill[] (9,10) circle (4pt) node[below] { }; 
        \fill[] (11,10) circle (4pt) node[below] { }; 

        \fill[] (6,12) circle (4pt) node[below] { }; 
        \fill[] (8,12) circle (4pt) node[below] { }; 
        \fill[] (10,12) circle (4pt) node[below] { }; 

        \fill[] (7,14) circle (4pt) node[above] {};
        \fill[] (9,14) circle (4pt) node[above] {};

        \fill[] (8,16) circle (4pt) node[above] {$c^8$};
 
    \end{tikzpicture}

    \columnbreak 

    \begin{tikzpicture}[scale=0.3][line join = round, line cap = round]
        \fill[rose]   (0,0) -- (3,6) -- (7,6) -- (10,0) -- (0,0);
        \fill[sand]    (10,0) -- (8,4) -- (11,10) -- (16, 0) -- (10,0);
        \fill[cyan]     (3,6) -- (9,6) -- (11,10) -- (8,16) -- (3,6);

        \draw[-, thick] (0,0) -- (16, 0) -- (8,16) -- (0,0);

        \draw[-, thick] (8,0) -- (5,6);
        \draw[-, thick] (6,0)--(3,6) -- (7,6) -- (10,0);
        \draw[-, thick] (4,0)--(2,4) -- (8,4);
        \draw[-, thick] (2,0) -- (1,2) -- (9,2);
 
        \draw[-, thick] (8,4) -- (11,10);
        \draw[-, thick] (15,2) -- (14,0) -- (10,8);
        \draw[-, thick] (14,4) -- (12,0) -- (9,6);
        \draw[-, thick] (10,0) -- (13,6) ;
        \draw[-, thick] (9,2) -- (12,8) ; 
        
        \draw[-, thick] (7,6) -- (9,6);
        \draw[-, thick] (7,14) -- (9,14) -- (5,6);
        \draw[-, thick] (6,12) -- (10,12) -- (7,6);
        \draw[-, thick] (5,10) -- (11,10) ;
        \draw[-, thick] (4,8) -- (10,8) ;
 
        \fill[] (0,0) circle (4pt) node[below] {$a^8$};
        \fill[] (2,0) circle (4pt) node[below] { };
        \fill[] (4,0) circle (4pt) node[below] { };
        \fill[] (6,0) circle (4pt) node[below] { };
        \fill[] (8,0) circle (4pt) node[below] { };
        \fill[] (10,0) circle (4pt) node[below] { };
        \fill[] (12,0) circle (4pt) node[below] { };
        \fill[] (14, 0) circle (4pt) ;
        \fill[] (16, 0) circle (4pt) node[below] {$b^8$};

        \fill[] (1,2) circle (4pt) node[below] { };
        \fill[] (3,2) circle (4pt) node[below] { };
        \fill[] (5,2) circle (4pt) node[below] { };
        \fill[] (7,2) circle (4pt) node[below] { };
        \fill[] (9,2) circle (4pt) node[below] { };
        \fill[] (11,2) circle (4pt) node[below] { };
        \fill[] (13,2) circle (4pt) node[below] { };
        \fill[] (15,2) circle (4pt) node[below] { };

        \fill[] (2,4) circle (4pt) node[below] { };
        \fill[] (4,4) circle (4pt) node[below] { };
        \fill[] (6,4) circle (4pt) node[below] { }; 
        \fill[] (10,4) circle (4pt) node[below] { };
        \fill[] (12,4) circle (4pt) node[below] { };
        \fill[] (14,4) circle (4pt) node[below] { };
 
        \fill[] (3,6) circle (4pt) node[below] { };
        \fill[] (5,6) circle (4pt) node[below] { };
        \fill[] (11,6) circle (4pt) node[below] { }; 
        \fill[] (13,6) circle (4pt) node[below] { };

        \fill[] (4,8) circle (4pt) node[below] { };
        \fill[] (6,8) circle (4pt) node[below] { };
        \fill[] (8,8) circle (4pt) node[below] { };
        \fill[] (10,8) circle (4pt) node[below] { }; 
        \fill[] (12,8) circle (4pt) node[below] { }; 

        \fill[] (5,10) circle (4pt) node[below] { };
        \fill[] (7,10) circle (4pt) node[below] { };
        \fill[] (9,10) circle (4pt) node[below] { }; 
        \fill[] (11,10) circle (4pt) node[below] { }; 

        \fill[] (6,12) circle (4pt) node[below] { }; 
        \fill[] (8,12) circle (4pt) node[below] { }; 
        \fill[] (10,12) circle (4pt) node[below] { }; 

        \fill[] (7,14) circle (4pt) node[above] {};
        \fill[] (9,14) circle (4pt) node[above] {};

        \fill[] (8,16) circle (4pt) node[above] {$c^8$};

        \fill[light-gray] (7,2) -- (9,2) -- (11,6) -- (10,8) -- (9,8) -- (6,8) -- (5,6)--(7,2); 
        
        \draw[-, thick] (7,2) -- (9,2) -- (11,6) -- (10,8) -- (9,8) -- (6,8) -- (5,6)--(7,2) ;

    \end{tikzpicture}
    
\end{multicols}
    
\caption{A polytopal complex supporting the minimal resolution of $(a,b,c)^{8}\smallsetminus \mathcal{S}_{8,a^{3}b^{3}c^{3}}$.}
    \label{f:7.2}
\end{figure}


Question \ref{q:BoxResolutions} is well understood in three variables, where every monomial ideal is known to have a polytopal resolution (\cite{Miller}). The situation is considerably wilder in four or more variables; it's not obvious that polytopal resolutions are the right target.

We refine Question \ref{q:BoxResolutions} by making the following definitions:
\begin{definition}
    Let $I$ be an equigenerated monomial ideal.  We say that a minimal resolution of $I$ is:
    \begin{itemize}
    \item A \emph{box resolution} if it is supported on a polytopal complex in which every cell is an admissible box, in the sense of Definition \ref{d:boxlabels}.
    \item A \emph{cyclic box resolution} if it is supported on a polytopal complex in which every cell is a box of the form $\Gamma_{i}(m)$, for some monomial $m$ and some cyclic ordering $Q_{i}$.
    \item A \emph{box-plus resolution} if it is supported on a polytopal complex in which every cell is either some $\Gamma_{i}(m)$ or a copy of $\widehat{Y}$.
    \end{itemize}
\end{definition}

\begin{example}  The complex of boxes is a cyclic box resolution, for any Borel ideal.

All the resolutions of $\dfrac{S}{I}$ constructed earlier in the paper are cyclic box resolutions.  The resolutions of $\dfrac{S}{\widehat{I}}$ in Figures \ref{f:YHat}, \ref{f:3pinched}, \ref{f:4pinched}, \ref{f:bcd-bcd}, and \ref{f:abcd3-bcd} are all box-plus resolutions, and will become cyclic box resolutions if we replace their copy of $\widehat{Y}$ with $Y$.

Another example of a box-plus resolution (in three variables) appears in Figure \ref{fig:box-plus}.
    \begin{figure}[hbt]
        \centering

        \begin{multicols}{2}

        \begin{tikzpicture}[scale=0.65]
        \tkzDefPoint(-4,0){1}
        \tkzDefPoint(-3,0){2}
        \tkzDefPoint(-2,0){3}
        \tkzDefPoint(0,0){4}
        \tkzDefPoint(1,0){5}
        \tkzDefPoint(3,0){6}
        \tkzDefPoint(4,0){7}

        \tkzDefPoint(-3.5,1){8}
        \tkzDefPoint(-1.5,1){9}
        \tkzDefPoint(-0.5,1){10}
        \tkzDefPoint(1.5,1){11}
        \tkzDefPoint(2.5,1){12}

        \tkzDefPoint(-3,2){13}
        \tkzDefPoint(-2,2){14}
        \tkzDefPoint(0,2){15}
        \tkzDefPoint(1,2){16}
        \tkzDefPoint(3,2){17} 

        \tkzDefPoint(-1.5,3){18}
        \tkzDefPoint(-0.5,3){19}
        \tkzDefPoint(1.5,3){20}
        \tkzDefPoint(2.5,3){21}

        \tkzDefPoint(-2,4){22}
        \tkzDefPoint(0,4){23}
        \tkzDefPoint(1,4){24} 

        \tkzDefPoint(-1.5,5){25}
        \tkzDefPoint(-0.5,5){26}
        \tkzDefPoint(1.5,5){27}

        \tkzDefPoint(0,6){28}
        \tkzDefPoint(1,6){29} 
        
        \tkzDefPoint(-0.5,7){30}

        \tkzDefPoint(0,8){31}

        \fill[light-gray] (1) -- (7) -- (31) -- (1);
        \fill[cyan] (1) -- (2) -- (8)--(1); 

        \draw [thick] (1) -- (2) -- (8)--(1); 
        \draw [thick] (1) -- (7) -- (31) -- (1);
        \draw [thick] (13) -- (14) -- (9) -- (3);
        \draw [thick] (9) -- (10) -- (4);
        \draw [thick] (10) -- (15) -- (16)--(11) -- (5);
        \draw [thick] (11) -- (12) -- (6);
        \draw [thick] (17) -- (12) ;
        \draw [thick] (22) -- (18) --(14) ;
        \draw [thick] (18) -- (19) --(15) ;
        \draw [thick] (16) -- (20) --(21) ;
        \draw [thick] (25) -- (26) --(23) -- (19) ;
        \draw [thick] (23) -- (24) --(20) ;
        \draw [thick] (27) -- (24) ;
        \draw [thick] (30) -- (28) -- (26);
        \draw [thick] (28) -- (29);

        \tkzDrawPoints[size=3](1,2,3,4,5,6,7,8,9,10,11,12,13,14,15,16,17,18,19,20,21,22,23,24,25,26,27,28,29,30,31)  

        \node[above] at (0,8)   {$c^8$};
        \node[below] at (-4,0)   {$a^8$};
        \node[below] at (4,0)   {$b^8$};
        
        \end{tikzpicture} 

        \columnbreak

        \begin{tikzpicture}[scale=0.65]
        \tkzDefPoint(-4,0){1}
        \tkzDefPoint(-3,0){2}
        \tkzDefPoint(-2,0){3}
        \tkzDefPoint(-1,0){32}
        \tkzDefPoint(0,0){4}
        \tkzDefPoint(1,0){5}
        \tkzDefPoint(2,0){33}
        \tkzDefPoint(3,0){6}
        \tkzDefPoint(4,0){7}

        \tkzDefPoint(-3.5,1){8}
        \tkzDefPoint(-2.5,1){34}
        \tkzDefPoint(-1.5,1){9}
        \tkzDefPoint(-0.5,1){10}
        \tkzDefPoint(0.5,1){35}
        \tkzDefPoint(1.5,1){11}
        \tkzDefPoint(2.5,1){12} 
        \tkzDefPoint(3.5,1){36}

        \tkzDefPoint(-3,2){13}
        \tkzDefPoint(-2,2){14}
        \tkzDefPoint(0,2){15}
        \tkzDefPoint(1,2){16}
        \tkzDefPoint(2,2){37}
        \tkzDefPoint(3,2){17} 

        \tkzDefPoint(-2.5,3){38}
        \tkzDefPoint(-1.5,3){18}
        \tkzDefPoint(-0.5,3){19}
        \tkzDefPoint(1.5,3){20}
        \tkzDefPoint(2.5,3){21}

        \tkzDefPoint(-2,4){22}
        \tkzDefPoint(0,4){23}
        \tkzDefPoint(1,4){24} 
        \tkzDefPoint(2,4){42} 

        \tkzDefPoint(-1.5,5){25}
        \tkzDefPoint(-0.5,5){26}
        \tkzDefPoint(0.5,5){39}
        \tkzDefPoint(1.5,5){27}

        \tkzDefPoint(-1,6){40}
        \tkzDefPoint(0,6){28}
        \tkzDefPoint(1,6){29} 
        
        \tkzDefPoint(-0.5,7){30}
        \tkzDefPoint(0.5,7){41}

        \tkzDefPoint(0,8){31}

        \fill[cyan] (1) -- (7) -- (31)--(1); 
        \fill[light-gray] (9) -- (10) -- (15) -- (16)--(20)--(24)--(23) -- (26) -- (25) -- (22) -- (18) -- (14) -- (9);
        \fill[rose] (34)--(9)--(14)--(34);
        \fill[rose] (35)--(12)--(37)--(11)--(16)--(35);

        \draw [thick] (9) -- (10) -- (15) -- (16)--(20)--(24)--(23) -- (26) -- (25) -- (22) -- (18) -- (14) -- (9); 
        \draw [thick] (19) -- (23);
        \draw [thick] (19) -- (18);
        \draw [thick] (19) -- (15);
        \draw [thick] (1) -- (7) -- (31)--(1);
        \draw [thick] (8) -- (2) -- (34)--(13);
        \draw [thick] (34) --  (9) -- (3);
        \draw [thick] (9) --  (10) -- (32);
        \draw [thick] (4) --  (10) -- (35) -- (5);
        \draw [thick] (35) --  (11) -- (33) ;
        \draw [thick] (11) --  (12) -- (6) ;
        \draw [thick] (12) -- (36) ;
        \draw [thick] (38) -- (14) -- (34);
        \draw [thick] (16) -- (35);
        \draw [thick] (20) -- (37)--(11)--(16);
        \draw [thick] (12) -- (37)--(17);
        \draw [thick] (20) -- (21);
        \draw [thick] (39)--(24) -- (42);
        \draw [thick] (26) -- (27);
        \draw [thick] (30)--(28)--(26) -- (40);
        \draw [thick] (28)--(29) -- (39);
        \draw [thick] (28)--(41);

        \tkzDrawPoints[size=3](1,2,3,4,5,6,7,8,9,10,11,12,13,14,15,16,17,18,19,20,21,22,23,24,25,26,27,28,29,30,31,32,33,34,35,36,37,38,39,40,41,42)

        \node[above] at (0,8)   {$c^8$};
        \node[below] at (-4,0)   {$a^8$};
        \node[below] at (4,0)   {$b^8$};
        
        \end{tikzpicture} 
            
        \end{multicols} 
        \caption{The complex on the left supports a box-plus resolution.  The complex on the right does not, because the red triangles are admissible boxes but not equal to $\Gamma_{i}(m)$ for any $i$ and $m$.}
        \label{fig:box-plus}
    \end{figure}
    
    \end{example}

    \begin{example}\label{e:not-box-plus}
    Let $\widehat{I}$ be the ideal of $k[a,b,c]$ generated by all degree-five monomials except $a^{3}bc$, $ab^{3}c$, and $abc^{3}$.  Then $\widehat{I}$ does not admit a box-plus resolution:  We would need the hexagons around the three omitted generators, as in figure \ref{f:notBoxPlus}, creating the isolated triangular region in the center, which is too small to be tiled by any $\Gamma_{i}$.  (In fact, because this triangle has negative orientation, it is not a box and cannot be tiled by boxes.)  However, the complex in this picture does support a non-minimal resolution of $\dfrac{S}{{I}}$.  Deleting any one of the three sides of this triangle yields a polytopal minimal resolution that is not box-plus.
\end{example}

\begin{figure}[hbt]\label{f:notBoxPlus}
        \centering

        \begin{multicols}{2}

        \begin{tikzpicture}[scale=1]
        \tkzDefPoint(-2.5,0){a^5}
        \tkzDefPoint(-1.5,0){a^4b}
        \tkzDefPoint(-0.5,0){a^3b^2}
        \tkzDefPoint(0.5,0){a^2b^3}
        \tkzDefPoint(1.5,0){ab^4}
        \tkzDefPoint(2.5,0){b^5}

        \tkzDefPoint(-2,1){a^4c} 
        \tkzDefPoint(0,1){a^2b^2c} 
        \tkzDefPoint(2,1){b^4c}

        \tkzDefPoint(-1.5,2){a^3c^2} 
        \tkzDefPoint(-0.5,2){a^2bc^2} 
        \tkzDefPoint(0.5,2){ab^2c^2}
        \tkzDefPoint(1.5,2){b^3c^2} 

        \tkzDefPoint(-1,3){a^2c^3}
        \tkzDefPoint(1,3){b^2c^3}

        \tkzDefPoint(-0.5,4){ac^4}
        \tkzDefPoint(0.5,4){bc^4}

        \tkzDefPoint(0,5){c^5}

        \fill[cyan] (a^5) -- (b^5) -- (c^5)--(a^5);
        \fill[light-gray] (a^4c) -- (a^4b) -- (a^3b^2) -- (a^2b^2c) -- (a^2bc^2)--(a^3c^2)--(a^4c);
        \fill[light-gray] (a^2b^2c)--(a^2b^3)--(ab^4)--(b^4c)--(b^3c^2)--(ab^2c^2)--(a^2b^2c);

        \fill[light-gray] (a^2bc^2)--(ab^2c^2)--(b^2c^3)--(bc^4)--(ac^4)--(a^2c^3)--(a^2bc^2);

        \fill[rose] (a^2bc^2) -- (a^2b^2c) -- (ab^2c^2);

        \draw [thick] (a^5) -- (b^5) -- (c^5) -- (a^5);
        \draw [thick] (ac^4) -- (bc^4);
        \draw [thick] (a^2c^3) -- (a^2b^3);
        \draw [thick] (a^3c^2) -- (b^3c^2);
        \draw [thick] (a^3b^2) -- (b^2c^3);
        \draw [thick] (a^4c) -- (a^4b);
        \draw [thick] (b^4c) -- (ab^4);

        \tkzDrawPoints[size=3](a^5,a^4b,a^3b^2,a^2b^3,ab^4,b^5,a^4c,a^2b^2c,b^4c,a^3c^2,a^2bc^2,ab^2c^2,b^3c^2,a^2c^3,b^2c^3,ac^4,bc^4,c^5)

        \tkzLabelPoints[above](c^5) 
        \tkzLabelPoints[below](a^5,a^4b,a^3b^2,a^2b^3,ab^4,b^5)
        \tkzLabelPoints[left](a^4c,a^3c^2,a^2c^3,ac^4)
        \tkzLabelPoints[right](b^4c,b^3c^2,b^2c^3,bc^4,a^2b^2c)
        \end{tikzpicture}

        \columnbreak

        \begin{tikzpicture}[scale=1]
        \tkzDefPoint(-2.5,0){a^5}
        \tkzDefPoint(-1.5,0){a^4b}
        \tkzDefPoint(-0.5,0){a^3b^2}
        \tkzDefPoint(0.5,0){a^2b^3}
        \tkzDefPoint(1.5,0){ab^4}
        \tkzDefPoint(2.5,0){b^5}

        \tkzDefPoint(-2,1){a^4c} 
        \tkzDefPoint(0,1){a^2b^2c} 
        \tkzDefPoint(2,1){b^4c}

        \tkzDefPoint(-1.5,2){a^3c^2} 
        \tkzDefPoint(-0.5,2){a^2bc^2} 
        \tkzDefPoint(0.5,2){ab^2c^2}
        \tkzDefPoint(1.5,2){b^3c^2} 

        \tkzDefPoint(-1,3){a^2c^3}
        \tkzDefPoint(1,3){b^2c^3}

        \tkzDefPoint(-0.5,4){ac^4}
        \tkzDefPoint(0.5,4){bc^4}

        \tkzDefPoint(0,5){c^5}

        \fill[cyan] (a^5) -- (b^5) -- (c^5)--(a^5);
        \fill[light-gray] (a^4c) -- (a^4b) -- (a^3b^2) -- (a^2b^2c) -- (a^2bc^2)--(a^3c^2)--(a^4c);
        \fill[light-gray] (a^2b^2c)--(a^2b^3)--(ab^4)--(b^4c)--(b^3c^2)--(ab^2c^2)--(a^2b^2c);

        \fill[light-gray] (a^2bc^2)--(ab^2c^2)--(b^2c^3)--(bc^4)--(ac^4)--(a^2c^3)--(a^2bc^2);

        \fill[green] (ac^4) -- (bc^4) -- (b^2c^3) -- (a^2b^2c) -- (a^2c^3);

        \draw [thick] (a^5) -- (b^5) -- (c^5) -- (a^5);
        \draw [thick] (ac^4) -- (bc^4);
        \draw [thick] (a^2c^3) -- (a^2b^3);
        \draw [thick] (a^3c^2) -- (a^2bc^2);
        \draw [thick] (ab^2c^2) -- (b^3c^2);
        \draw [thick] (a^3b^2) -- (b^2c^3);
        \draw [thick] (a^4c) -- (a^4b);
        \draw [thick] (b^4c) -- (ab^4);

        \tkzDrawPoints[size=3](a^5,a^4b,a^3b^2,a^2b^3,ab^4,b^5,a^4c,a^2b^2c,b^4c,a^3c^2,a^2bc^2,ab^2c^2,b^3c^2,a^2c^3,b^2c^3,ac^4,bc^4,c^5)

        \tkzLabelPoints[above](c^5) 
        \tkzLabelPoints[below](a^5,a^4b,a^3b^2,a^2b^3,ab^4,b^5)
        \tkzLabelPoints[left](a^4c,a^3c^2,a^2c^3,ac^4)
        \tkzLabelPoints[right](b^4c,b^3c^2,b^2c^3,bc^4,a^2b^2c)
        \end{tikzpicture}
            
        \end{multicols}

\caption{Let $\widehat{I}$ be the ideal generated by all degree five monomials except $a^{3}bc$, $ab^{3}c$,and $abc^{3}$.  The complex on the left supports a non-minimal resolution of $\widehat{I}$.  The complex on the right supports a minimal resolution, but is not box-plus.}
\end{figure}

It is natural to ask which monomial ideals support which type of resolution.  In particular, the difference between a box resolution and a cyclic box resolution seems important.  For example, if $n=4$ and the generating degree is large, then almost every cell of the Eliahou-Kervaire resolution has the form $\{f\}\times\{a,d\}\times \{b,d\}\times \{c,d\}$ for some $f$.  These cells are admissible boxes, but are not equal to any $\Gamma_{i}(m)$.

On the other hand, we can make some substantial statements about cyclic box resolutions and box-plus resolutions by rephrasing these objects in terms of staircase diagrams, as we do below.

Staircase diagrams have a long history in the study of monomial ideals; see for example \cite[Chapters 3,4,6]{miller2004combinatorial} or \cite{Bayer1996duality}.  Normally these are studied as a Stanley decomposition adjacent to the Newton polytope, (so that corners correspond to generators or syzygies,) but we will use them somewhat differently to parametrize resolutions.

\begin{remark}
    We note that the staircase diagrams discussed below feel similar to the tropical pictures found in \cite{DochtermannJoswigSanyal}.  While we are far from experts in tropical geometry, we hope that the subject might have something interesting to say about these questions.
\end{remark}

\begin{definition}\label{d:staircase}
    A \emph{staircase diagram} (see Figure \ref{f:staircase-example} for an example of a staircase diagram) in $\mathbb{R}^{n}$ is a nonempty infinite $n$-dimensional cubical complex $\mathcal{S}\subset \mathbb{R}^{n}$ satisfying:
    \begin{enumerate}
        \item Every vertex is a lattice point.
        \item Every vertex is the  outermost corner of a $d$-dimensional cube.  More precisely, if $\mathbf{v}=(a_{1},\dots, a_{n})$ is a vertex of $\mathcal{S}$, then the hypercube $[a_{1}-1,a_{1}]\times [a_{2}-1,a_{2}]\times \dots \times [a_{n}-1,a_{n}]$ is a facet of $\mathcal{S}$. 
        \end{enumerate}

    The \emph{visible surface} of $\mathcal{S}$ is the $(n-1)$-dimensional cubical complex consisting of the faces of $\mathcal{S}$ visible from infinitely far into the first orthant in the direction $(1,1,1,\dots, 1)$.
\end{definition}

\begin{figure}[h!]
    \centering

    \begin{tikzpicture}[scale=0.5]
        \tkzDefPoint(0,0){a^4}
        \tkzDefPoint(2,-1){a^3b}
        \tkzDefPoint(4,-2){a^2b^2}
        \tkzDefPoint(6,-3){ab^3}
        \tkzDefPoint(8,-4){b^4}
        
        \tkzDefPoint(2,1){a^3c}
        \tkzDefPoint(4,0){a^2bc}
        \tkzDefPoint(6,-1){ab^2c}
        \tkzDefPoint(8,-2){b^3c}
        
        \tkzDefPoint(4,2){a^2c^2}
        \tkzDefPoint(6,1){abc^2}
        \tkzDefPoint(8,0){b^2c^2}
        
        \tkzDefPoint(6,3){ac^3}
        \tkzDefPoint(8,2){bc^3}
        
        \tkzDefPoint(8,4){c^4}

        \tkzDefPoint(10,-3){1}
        \tkzDefPoint(12,-2){2}
        \tkzDefPoint(14,-1){3}
        \tkzDefPoint(16,0){4}
        \tkzDefPoint(0,2){5}
        \tkzDefPoint(10,-1){6}
        \tkzDefPoint(12,0){7}
        \tkzDefPoint(10,1){x}
        
        \tkzDefPoint(14,1){8}
        \tkzDefPoint(16,2){9}
        \tkzDefPoint(0,4){10}
        \tkzDefPoint(2,3){11}
        \tkzDefPoint(12,2){12}
        \tkzDefPoint(14,3){13}
        \tkzDefPoint(16,4){14}

        \tkzDefPoint(0,6){15}
        \tkzDefPoint(2,5){16}
        \tkzDefPoint(4,4){17}
        \tkzDefPoint(10,3){18}
        \tkzDefPoint(12,4){19}
        \tkzDefPoint(14,5){20}
        \tkzDefPoint(16,6){21}

        \tkzDefPoint(0,8){22}
        \tkzDefPoint(2,7){23}
        \tkzDefPoint(4,6){24}
        \tkzDefPoint(6,5){25}
        \tkzDefPoint(10,5){26}
        \tkzDefPoint(12,6){27}
        \tkzDefPoint(14,7){28}
        \tkzDefPoint(16,8){29}

        \tkzDefPoint(2,9){30}
        \tkzDefPoint(4,8){31}
        \tkzDefPoint(6,7){32}
        \tkzDefPoint(8,6){33}
        \tkzDefPoint(10,7){34}
        \tkzDefPoint(12,8){35}
        \tkzDefPoint(14,9){36}

        \tkzDefPoint(4,10){37}
        \tkzDefPoint(6,9){38}
        \tkzDefPoint(8,8){39}
        \tkzDefPoint(10,9){40}
        \tkzDefPoint(12,10){41}

        \tkzDefPoint(6,11){42}
        \tkzDefPoint(8,10){43}
        \tkzDefPoint(10,11){44}

        \tkzDefPoint(8,12){45}

        \fill[cyan] (a^4) -- (b^4) -- (4) -- (29) -- (45) -- (22) -- (a^4);
        \fill[light-light-gray] (a^3c) -- (a^2c^2) -- (b^2c^2) -- (12) -- (8) -- (b^3c) -- (a^3c);
        \fill[light-light-gray] (10) -- (24) -- (c^4) -- (26) -- (19) -- (20) -- (14) -- (13) -- (19) -- (bc^3) -- (17) -- (11) -- (10);

        \fill[light-light-gray] (31) -- (43) -- (28) -- (27) -- (34) -- (33) -- (31);

        \draw [thick] (a^4) -- (b^4) -- (4) -- (29) -- (45) -- (22) -- (a^4);
        \draw [thick] (5) -- (b^3c) -- (9);
        \draw [thick] (10) -- (11);
        \draw [thick] (a^2c^2) -- (6);
        \draw [thick] (16) -- (bc^3);
        \draw [thick] (x) -- (7);
        \draw [thick] (24) -- (18);
        \draw [thick] (12) -- (8);
        \draw [thick] (31) -- (33);

        \draw [thick] (26) -- (13);
        \draw [thick] (38) -- (27);
        \draw [thick] (20) -- (14);
        \draw [thick] (15) -- (43) -- (21);

        \draw [thick] (ab^2c) -- (12);
        \draw [thick] (13) -- (14);
        \draw [thick] (a^2bc) -- (abc^2);
        \draw [thick] (bc^3) -- (20);

        \draw [thick] (a^3c) -- (a^2c^2);
        \draw [thick] (ac^3) -- (26);
        \draw [thick] (27) -- (28);
        \draw [thick] (11) -- (25);

        \draw [thick] (33) -- (35);
        \draw [thick] (10) -- (24);
        \draw [thick] (32) -- (40);

        \draw [thick] (30) -- (16);
        \draw [thick] (11) -- (a^3b);
        \draw [thick] (37) -- (24);
        \draw [thick] (17) -- (a^2c^2);
        \draw [thick] (a^2bc) -- (a^2b^2);
        
        \draw [thick] (42) -- (38);
        \draw [thick] (32) -- (25);
        \draw [thick] (ac^3) -- (abc^2);

        \draw [thick] (ab^2c) -- (ab^3);

        \draw [thick] (45) -- (43);
        \draw [thick] (33) -- (c^4);

        \draw [thick] (bc^3) -- (b^2c^2);
        \draw [thick] (b^3c) -- (b^4);

        \draw [thick] (44) -- (40);

        \draw [thick] (34) -- (26);

        \draw [thick] (18) -- (x);
        \draw [thick] (6) -- (1);

        \draw [thick] (41) -- (35);
        \draw [thick] (27) -- (12);
        \draw [thick] (7) -- (2);

        \draw [thick] (36) -- (20);

        \draw [thick] (13) -- (3); 

        \tkzDrawPoints[size=3](a^4,a^3b,a^2b^2,ab^3,b^4,a^3c,a^2bc,ab^2c,b^3c,a^2c^2,abc^2,b^2c^2,ac^3,bc^3,c^4,1,2,3,4,5,6,7,8,9,10,11,12,13,14,15,16,17,18,19,20,21,22,23,24,25,26,27,28,29,30,31,32,33,34,35,36,37,38,39,40,41,42,43,44,45,x)        
        
        \node[left] at (0,0)   {(4,0,0)};
        \node[above] at (8,12) {(0,0,4)}; 
        \node[right] at (16,0) {(0,4,0)};

        \node[left] at (0,8)   {(4,0,4)};
        \node[below] at (8,-4) {(4,4,0)};        
        \node[right] at (16,8) {(0,4,4)};
        \end{tikzpicture}
    
    \caption{An example of a staircase diagram.}
    \label{f:staircase-example}
\end{figure}

\begin{construction}\label{c:staircaselabel}

        Fix a degree $d$ monomial $m$ and a staircase diagram $\mathcal{S}$.  To the lattice point $(a_{1},a_{2},\dots, a_{n})$, we associate the monomial
        \begin{align*}
             m\cdot(x_{1}^{a_{1}-a_{2}} x_{2}^{a_{2}-a_{3}} \dots x_{n}^{a_{n}-a_{1}})= m\left(\frac{x_{1}}{x_{n}}\right)^{a_{1}}\left(\frac{x_{2}}{x_{1}}\right)^{a_{2}}\cdots \left(\frac{x_{n}}{x_{n-1}}\right)^{a_{n}} 
        \end{align*} 
        in $k[x_{1},\dots,x_{n},x_{1}^{-1},\dots,x_{n}^{-1}]$. (Note that this means the origin is associated to $m$.)  Label every vertex of the visible surface of $\mathcal{S}$ by the associated monomial.  We refer to the resulting labeled staircase diagram as $\mathcal{S}_{m}$.

\end{construction}

We illustrate Construction \ref{c:staircaselabel} via the following example.

\begin{example} 
Consider the staircase diagram given in Figure \ref{f:staircase-example}, and set $m=c^{4}\in k[a,b,c]$.
Then $(2,2,2)$ is labeled with $m=m\cdot a^{0}b^{0}c^{0}$. Meanwhile, $(3,2,2)$ is labeled with $c^4 \cdot (a^{1}b^0c^{-1})=ac^3$.  Similarly, we can obtain labels of all lattice points (see Figure \ref{f:staircase} for labels of some lattice points of the staircase diagram).

\begin{tiny}
    
\end{tiny}

\begin{figure}[hbt]
    \centering

    \begin{tikzpicture}[scale=0.5]
        \tkzDefPoint(0,0){a^4}
        \tkzDefPoint(2,-1){a^3b}
        \tkzDefPoint(4,-2){a^2b^2}
        \tkzDefPoint(6,-3){ab^3}
        \tkzDefPoint(8,-4){b^4}
        
        \tkzDefPoint(2,1){a^3c}
        \tkzDefPoint(4,0){a^2bc}
        \tkzDefPoint(6,-1){ab^2c}
        \tkzDefPoint(8,-2){b^3c}
        
        \tkzDefPoint(4,2){a^2c^2}
        \tkzDefPoint(6,1){abc^2}
        \tkzDefPoint(8,0){b^2c^2}
        
        \tkzDefPoint(6,3){ac^3}
        \tkzDefPoint(8,2){bc^3}
        
        \tkzDefPoint(8,4){c^4}

        \tkzDefPoint(10,-3){1}
        \tkzDefPoint(12,-2){2}
        \tkzDefPoint(14,-1){3}
        \tkzDefPoint(16,0){4}
        \tkzDefPoint(0,2){5}
        \tkzDefPoint(10,-1){6}
        \tkzDefPoint(12,0){7}
        \tkzDefPoint(10,1){x}
        
        \tkzDefPoint(14,1){8}
        \tkzDefPoint(16,2){9}
        \tkzDefPoint(0,4){10}
        \tkzDefPoint(2,3){11}
        \tkzDefPoint(12,2){12}
        \tkzDefPoint(14,3){13}
        \tkzDefPoint(16,4){14}

        \tkzDefPoint(0,6){15}
        \tkzDefPoint(2,5){16}
        \tkzDefPoint(4,4){17}
        \tkzDefPoint(10,3){18}
        \tkzDefPoint(12,4){19}
        \tkzDefPoint(14,5){20}
        \tkzDefPoint(16,6){21}

        \tkzDefPoint(0,8){22}
        \tkzDefPoint(2,7){23}
        \tkzDefPoint(4,6){24}
        \tkzDefPoint(6,5){25}
        \tkzDefPoint(10,5){26}
        \tkzDefPoint(12,6){27}
        \tkzDefPoint(14,7){28}
        \tkzDefPoint(16,8){29}

        \tkzDefPoint(2,9){30}
        \tkzDefPoint(4,8){31}
        \tkzDefPoint(6,7){32}
        \tkzDefPoint(8,6){33}
        \tkzDefPoint(10,7){34}
        \tkzDefPoint(12,8){35}
        \tkzDefPoint(14,9){36}

        \tkzDefPoint(4,10){37}
        \tkzDefPoint(6,9){38}
        \tkzDefPoint(8,8){39}
        \tkzDefPoint(10,9){40}
        \tkzDefPoint(12,10){41}

        \tkzDefPoint(6,11){42}
        \tkzDefPoint(8,10){43}
        \tkzDefPoint(10,11){44}

        \tkzDefPoint(8,12){45}

        \fill[cyan] (a^4) -- (b^4) -- (4) -- (29) -- (45) -- (22) -- (a^4);
        \fill[light-light-gray] (a^3c) -- (a^2c^2) -- (b^2c^2) -- (12) -- (8) -- (b^3c) -- (a^3c);
        \fill[light-light-gray] (10) -- (24) -- (c^4) -- (26) -- (19) -- (20) -- (14) -- (13) -- (19) -- (bc^3) -- (17) -- (11) -- (10);

        \fill[light-light-gray] (31) -- (43) -- (28) -- (27) -- (34) -- (33) -- (31);

        \draw [thick] (a^4) -- (b^4) -- (4) -- (29) -- (45) -- (22) -- (a^4);
        \draw [thick] (5) -- (b^3c) -- (9);
        \draw [thick] (10) -- (11);
        \draw [thick] (a^2c^2) -- (6);
        \draw [thick] (16) -- (bc^3);
        \draw [thick] (x) -- (7);
        \draw [thick] (24) -- (18);
        \draw [thick] (12) -- (8);
        \draw [thick] (31) -- (33);

        \draw [thick] (26) -- (13);
        \draw [thick] (38) -- (27);
        \draw [thick] (20) -- (14);
        \draw [thick] (15) -- (43) -- (21);

        \draw [thick] (ab^2c) -- (12);
        \draw [thick] (13) -- (14);
        \draw [thick] (a^2bc) -- (abc^2);
        \draw [thick] (bc^3) -- (20);

        \draw [thick] (a^3c) -- (a^2c^2);
        \draw [thick] (ac^3) -- (26);
        \draw [thick] (27) -- (28);
        \draw [thick] (11) -- (25);

        \draw [thick] (33) -- (35);
        \draw [thick] (10) -- (24);
        \draw [thick] (32) -- (40);

        \draw [thick] (30) -- (16);
        \draw [thick] (11) -- (a^3b);
        \draw [thick] (37) -- (24);
        \draw [thick] (17) -- (a^2c^2);
        \draw [thick] (a^2bc) -- (a^2b^2);
        
        \draw [thick] (42) -- (38);
        \draw [thick] (32) -- (25);
        \draw [thick] (ac^3) -- (abc^2);

        \draw [thick] (ab^2c) -- (ab^3);

        \draw [thick] (45) -- (43);
        \draw [thick] (33) -- (c^4);

        \draw [thick] (bc^3) -- (b^2c^2);
        \draw [thick] (b^3c) -- (b^4);

        \draw [thick] (44) -- (40);

        \draw [thick] (34) -- (26);

        \draw [thick] (18) -- (x);
        \draw [thick] (6) -- (1);

        \draw [thick] (41) -- (35);
        \draw [thick] (27) -- (12);
        \draw [thick] (7) -- (2);

        \draw [thick] (36) -- (20);

        \draw [thick] (13) -- (3); 

        \tkzDrawPoints[size=3](a^4,a^3b,a^2b^2,ab^3,b^4,a^3c,a^2bc,ab^2c,b^3c,a^2c^2,abc^2,b^2c^2,ac^3,bc^3,c^4,1,2,3,4,5,6,7,8,9,10,11,12,13,14,15,16,17,18,19,20,21,22,23,24,25,26,27,28,29,30,31,32,33,34,35,36,37,38,39,40,41,42,43,44,45,x)

        \tkzLabelPoints[right](a^3c)  
        \tkzLabelPoints[above](a^2bc,ab^2c,b^3c,bc^3,ac^3) 
        \tkzLabelPoints[below](a^4,a^3b,a^2b^2,ab^3,b^4,b^2c^2,abc^2,a^2c^2,c^4) 

        \node[below] at (12.1,-2.1)   {$a^{-2}b^4c^{2}$}; 
        \node[above] at (12,0.25)   {$a^{-2}b^3c^{3}$};
        \node[above] at (2,3.2)   {$a^{3}b^{-1}c^{2}$};

        \end{tikzpicture}
    
    \caption{An example of $\mathcal{S}_{m}$.}
    \label{f:staircase}
\end{figure}

\end{example}

\begin{remark}\label{r:labeledLine}
    Observe that for any lattice point $v$, the labels on $v$ and $v+(1,1,\dots, 1)$ are equal.  In fact, the labeling in Construction \ref{c:staircaselabel} provides a bijection among the set of lattice points on the visible surface of $\mathcal{S}$, the set of lattice points in $\mathbb{R}^{n}$ (modulo translation by $(1,1,\dots, 1)$),    
    and the set of degree $d$ monomials in the ring $k[x_{1},\dots,x_{n},x_{1}^{-1},\dots,x_{n}^{-1}].$  (Indeed, this bijection is given by orthogonal projection along the vector $(1,1,\dots, 1)$ onto the degree-$d$ hyperplane $\sum x_{i}=d$.)
\end{remark}

\begin{remark}\label{r:staircaseRecursion}
        The labeling of the visible surface of $\mathcal{S}_{m}$ can be achieved by the following inductive procedure:

        First, choose an arbitrary vertex $v$ on the visible surface of $\mathcal{S}$, and label $v$ as described in Construction \ref{c:staircaselabel}.  Then label the remaining vertices of the visible surface inductively as follows:
        
        If $w$ is already labeled by a monomial $m$, and $u$ is adjacent to $w$, then $u-w=\pm e_{j}$ for some $j$.  We label $u$ with $m\dfrac{x_{j}}{x_{j-1}}$ if the coefficient on $e_{j}$ is positive, and with $m\dfrac{x_{j-1}}{x_{j}}$ if the coefficient is negative.
        
\end{remark}

\begin{construction}\label{c:staircaseresolution}
    Fix a labeled staircase diagram $\mathcal{S}_{m}$, and let $T_{m}$ be its orthogonal projection onto the degree $d$ hyperplane $\mathcal{H}:=\{\mathbf{v}=(a_{1},\dots, a_{n}): a_{1}+\dots + a_{n}=d\}$.  Observe that $T_{m}$ is an (infinite) polytopal complex that tiles the hyperplane.
    
    Consider the convex hull $\mathcal{F}_{m}$ in $\mathcal{H}$ of all vertices whose labels are in $k[x_{1},\dots, x_{n}]$ (i.e., the vertices with no negative exponents), and view $\mathcal{F}_{m}$ as a simplex in the obvious way.  Make $\mathcal{F}_{m}$ into a polytopal complex by setting the $i$-dimensional faces equal to the $i$-dimensional intersections between faces of $\mathcal{F}_{m}$ and faces of $T_{m}$.
    \end{construction}

\begin{example}\label{e:staircaseresolution}
    In Figure \ref{f:staircaseresolution}, the simplex $\mathcal{H}$ is bounded by the red triangle.  Observe that the red line cuts across some diagonals of faces in $\mathcal{S}$, so $\mathcal{H}$ contains some triangular faces (e.g., $\{ab^{2}c, b^{3}c, b^{2}c^{2}\}$) in addition to the many square faces arising from the projection.

\begin{figure}[hbt]
    \centering

    \begin{tikzpicture}[scale=0.5]
        \tkzDefPoint(0,0){a^4}
        \tkzDefPoint(2,-1){a^3b}
        \tkzDefPoint(4,-2){a^2b^2}
        \tkzDefPoint(6,-3){ab^3}
        \tkzDefPoint(8,-4){b^4}
        
        \tkzDefPoint(2,1){a^3c}
        \tkzDefPoint(4,0){a^2bc}
        \tkzDefPoint(6,-1){ab^2c}
        \tkzDefPoint(8,-2){b^3c}
        
        \tkzDefPoint(4,2){a^2c^2}
        \tkzDefPoint(6,1){abc^2}
        \tkzDefPoint(8,0){b^2c^2}
        
        \tkzDefPoint(6,3){ac^3}
        \tkzDefPoint(8,2){bc^3}
        
        \tkzDefPoint(8,4){c^4}

        \tkzDefPoint(10,-3){1}
        \tkzDefPoint(12,-2){2}
        \tkzDefPoint(14,-1){3}
        \tkzDefPoint(16,0){4}
        \tkzDefPoint(0,2){5}
        \tkzDefPoint(10,-1){6}
        \tkzDefPoint(12,0){7}
        \tkzDefPoint(10,1){x}
        
        \tkzDefPoint(14,1){8}
        \tkzDefPoint(16,2){9}
        \tkzDefPoint(0,4){10}
        \tkzDefPoint(2,3){11}
        \tkzDefPoint(12,2){12}
        \tkzDefPoint(14,3){13}
        \tkzDefPoint(16,4){14}

        \tkzDefPoint(0,6){15}
        \tkzDefPoint(2,5){16}
        \tkzDefPoint(4,4){17}
        \tkzDefPoint(1
        0,3){18}
        \tkzDefPoint(12,4){19}
        \tkzDefPoint(14,5){20}
        \tkzDefPoint(16,6){21}

        \tkzDefPoint(0,8){22}
        \tkzDefPoint(2,7){23}
        \tkzDefPoint(4,6){24}
        \tkzDefPoint(6,5){25}
        \tkzDefPoint(10,5){26}
        \tkzDefPoint(12,6){27}
        \tkzDefPoint(14,7){28}
        \tkzDefPoint(16,8){29}

        \tkzDefPoint(2,9){30}
        \tkzDefPoint(4,8){31}
        \tkzDefPoint(6,7){32}
        \tkzDefPoint(8,6){33}
        \tkzDefPoint(10,7){34}
        \tkzDefPoint(12,8){35}
        \tkzDefPoint(14,9){36}

        \tkzDefPoint(4,10){37}
        \tkzDefPoint(6,9){38}
        \tkzDefPoint(8,8){39}
        \tkzDefPoint(10,9){40}
        \tkzDefPoint(12,10){41}

        \tkzDefPoint(6,11){42}
        \tkzDefPoint(8,10){43}
        \tkzDefPoint(10,11){44}

        \tkzDefPoint(8,12){45}

        \fill[cyan] (a^4) -- (b^4) -- (4) -- (29) -- (45) -- (22) -- (a^4);
        \fill[light-light-gray] (a^3c) -- (a^2c^2) -- (b^2c^2) -- (12) -- (8) -- (b^3c) -- (a^3c);
        \fill[light-light-gray] (10) -- (24) -- (c^4) -- (26) -- (19) -- (20) -- (14) -- (13) -- (19) -- (bc^3) -- (17) -- (11) -- (10);

        \fill[light-light-gray] (31) -- (43) -- (28) -- (27) -- (34) -- (33) -- (31);

        \draw [thick] (a^4) -- (b^4) -- (4) -- (29) -- (45) -- (22) -- (a^4);
        \draw [thick] (5) -- (b^3c) -- (9);
        \draw [thick] (10) -- (11);
        \draw [thick] (a^2c^2) -- (6);
        \draw [thick] (16) -- (bc^3);
        \draw [thick] (x) -- (7);
        \draw [thick] (24) -- (18);
        \draw [thick] (12) -- (8);
        \draw [thick] (31) -- (33);

        \draw [thick] (26) -- (13);
        \draw [thick] (38) -- (27);
        \draw [thick] (20) -- (14);
        \draw [thick] (15) -- (43) -- (21);

        \draw [thick] (ab^2c) -- (12);
        \draw [thick] (13) -- (14);
        \draw [thick] (a^2bc) -- (abc^2);
        \draw [thick] (bc^3) -- (20);

        \draw [thick] (a^3c) -- (a^2c^2);
        \draw [thick] (ac^3) -- (26);
        \draw [thick] (27) -- (28);
        \draw [thick] (11) -- (25);

        \draw [thick] (33) -- (35);
        \draw [thick] (10) -- (24);
        \draw [thick] (32) -- (40);

        \draw [thick] (30) -- (16);
        \draw [thick] (11) -- (a^3b);
        \draw [thick] (37) -- (24);
        \draw [thick] (17) -- (a^2c^2);
        \draw [thick] (a^2bc) -- (a^2b^2);
        
        \draw [thick] (42) -- (38);
        \draw [thick] (32) -- (25);
        \draw [thick] (ac^3) -- (abc^2);

        \draw [thick] (ab^2c) -- (ab^3);

        \draw [thick] (45) -- (43);
        \draw [thick] (33) -- (c^4);

        \draw [thick] (bc^3) -- (b^2c^2);
        \draw [thick] (b^3c) -- (b^4);

        \draw [thick] (44) -- (40);

        \draw [thick] (34) -- (26);

        \draw [thick] (18) -- (x);
        \draw [thick] (6) -- (1);

        \draw [thick] (41) -- (35);
        \draw [thick] (27) -- (12);
        \draw [thick] (7) -- (2);

        \draw [thick] (36) -- (20);

        \draw [thick] (13) -- (3);

        \draw [thick,red] (a^4) -- (b^4) -- (c^4) -- (a^4);

        \tkzDrawPoints[size=3.5](a^4,a^3b,a^2b^2,ab^3,b^4,a^3c,a^2bc,ab^2c,b^3c,a^2c^2,abc^2,b^2c^2,ac^3,bc^3,c^4,1,2,3,4,5,6,7,8,9,10,11,12,13,14,15,16,17,18,19,20,21,22,23,24,25,26,27,28,29,30,31,32,33,34,35,36,37,38,39,40,41,42,43,44,45,x) 

        \tkzLabelPoints[below](a^4,a^3b,a^2b^2,ab^3,b^4)  
        
        \node[left] at (2,2)   {$a^3c$}; 
        \node[above] at (4,0.1)   {$a^2bc$};
        \node[above] at (6,-0.9)   {$ab^2c$}; 
        \node[right] at (8,-2.1)   {$b^3c$};  

        \node[left] at (4,2.15)   {$a^2c^2$}; 
        \node[right] at (6,1.2)   {$abc^2$}; 
        \node[right] at (8.3,0)   {$b^2c^2$}; 

        \node[above] at (6,3)   {$ac^3$};
        \node[right] at (8,1.9)   {$bc^3$}; 

        \node[right] at (8.3,4.1)   {$c^4$};   
        \end{tikzpicture}
    
    \caption{An example of $T_m$.}
    \label{f:staircaseresolution}
\end{figure}
\end{example}

\begin{proposition}\label{p:staircasesMakeBoxes}
    Fix a staircase, a vertex, and a monomial.  Then Construction \ref{c:staircaseresolution} produces a cyclic box resolution.
\end{proposition}
\begin{proof}
    Fix a labeled staircase $\mathcal{S}_{m}$.  The cells arising from Construction \ref{c:staircaseresolution} are clearly boxes, but we need to verify that the facets have the form $\Gamma_{i}(\mu)$ for some $\mu$.  We also need to show that Construction \ref{c:staircaseresolution} in fact supports a resolution.
    
    We first show that every facet has the form $\Gamma_{i}(\mu)$.  Let a facet $F$ be given; then $F$ corresponds to an $n-1$-dimensional face of $\mathcal{S}$, which by abuse we also refer to as $F$.  Let $\mathbf{v}$ be the outermost vertex of $F$ (that is, the vertex maximizing the dot product $\mathbf{v}\bullet (1,1,\dots, 1)$), and observe that, without loss of generality $F$ is the intersection of $\mathcal{F}_{m}$ with the  convex hull of the $2^{n-1}$ points $v-\sum_{j\in\sigma}e_{j}$, the sum taken over all subsets $\sigma\subset \{2,\dots, n\}$.

    Let $\mathbf{v}$ be labeled by $\mu$, and observe that the vertex $\mathbf{v}-\sum_{j\in\sigma}e_{j}$ is labeled by $\mu\cdot \prod_{j\in\sigma}\frac{x_{j-1}}{x_{j}}$.  Thus the convex hull of the $\mathbf{v}-\sum_{j\in\sigma}e_{j}$ has the form $\dfrac{\mu}{x_{2}\dots x_{n}}\times [x_{1},x_{2}]\times [x_{2},x_{3}]\times \dots \times [x_{n-1},x_{n}]$.  If $\supp(\mu)$ contains $\{2,\dots, n\}$, this convex hull is $\Gamma_{1}(\mu)$.  If not, we obtain $\Gamma_{1}(\mu)$ upon intersecting with $\mathcal{F}_{m}$.

    To verify that the construction supports a resolution, fix a multidegree $\boldsymbol{\alpha}=x_{1}^{\alpha_{1}}\dots x_{n}^{\alpha_{n}}$; we will show that $(\mathcal{F}_{m})_{\leq \boldsymbol\alpha}$ is contractible.

    If any $\alpha_{i}<0$, $(\mathcal{F}_{m})_{\leq \boldsymbol\alpha}$ is empty and hence contractible.

    We may assume without loss of generality that $0\leq \alpha_{i}\leq d$.

    Let $\mathcal{G}_{\boldsymbol{\alpha}}$ be the polytope described by the half-spaces $\{a_{j}\geq 0\}_{j}$ and $\{a_{j}\leq \alpha_{j}\}_{j}$, and observe that $\mathcal{G}_{\boldsymbol{\alpha}}$ is contractible.  We will show that $\mathcal{F}_{\leq\boldsymbol\alpha}$ is homotopic to $\mathcal{G}_{\boldsymbol\alpha}$.

    Observe that every facet $\Gamma$ of $\mathcal{F}_{m}$ is equal to $\Gamma_{i}(\mu)$ for some monomial $\mu$.  Writing $\mu=\prod x_{i}^{t_{i}}$, we see that $\Gamma$ is the set of all points $(a_{1},\dots, a_{n})$ satisfying the equations

    \[\begin{pmatrix}
        t_{i}&\leq& a_{i}&\leq &t_{i}+1\\
        \max(0,t_{i-1}-1) &\leq& a_{i-1}&\leq& t_{i-1}\\
        \max(0,t_{j}-1)&\leq& a_{j}&\leq& t_{j}+1 
    \end{pmatrix}
\]
    We have that $\Gamma\subset \mathcal{F}_{\leq\boldsymbol{\alpha}}$ if and only $\Gamma\subset \mathcal{G}_{\boldsymbol{\alpha}}$.  If $\Gamma\cap \mathcal{G}_{\boldsymbol{\alpha}}$ is a nontrivial proper subset of $\Gamma$, set $A=\{j: \alpha_{j}=t_{j}\}$, the set of indices such that the hyperplane $(a_{j}=\alpha_{j})$ slices through $\Gamma$.  Movement of every point on the interior of $\Gamma\cap \mathcal{G}_{\boldsymbol{\alpha}}$ along the vector $\sum_{j\in A}e_{j}$ defines a deformation retract from $\Gamma\cap \mathcal{G}_{\boldsymbol{\alpha}}$ to $\Gamma\cap \mathcal{F}_{\leq \boldsymbol{\alpha}}$.  The union of all these deformations is a deformation retract from $\mathcal{G}_{\boldsymbol{\alpha}}$ to $\mathcal{F}_{\leq \boldsymbol{\alpha}}$.  In particular, $\mathcal{F}_{\leq \boldsymbol{\alpha}}$ is homotopic to $\mathcal{G}_{\boldsymbol\alpha}$ and thus is contractible. 
\end{proof}

\begin{proposition}\label{p:allBoxAreStaircase}
    Every cyclic box resolution arises from a staircase diagram via Construction \ref{c:staircaseresolution}.
\end{proposition}
\begin{proof}
    Let $X$ be a cyclic box resolution.  We will build a corresponding staircase diagram, following the intuition of Remark \ref{r:staircaseRecursion}.

    Choose an arbitrary vertex $m$ of $X$ and an arbitrary point $P$ in $\mathbb{Z}^{n}$, and label $P$ with $m$.  Then assign the remaining vertices of $X$ to points in $\mathbb{Z}^{n}$ recursively as follows:  If $f$ has already labeled a point $Q$ and $g$ is connected to $f$ by an edge of $X$, then for some monomial $m$ and index $i$, $f$ and $g$ are both vertices of $\Gamma_{i}(m)$.  
    It follows that there exist indices $j$ and $j'$ such that either $g=f\frac{x_{j}}{x_{j'}}$ or $g=f\frac{x_{j'}}{x_{j}}$ and, cyclically, $i\leq j<j'\leq i-1$.  If $g=f\frac{x_{j}}{x_{j'}}$, assign the label $g$ to $Q+e_{j+1}+e_{j+2}+\dots + e_{j'}$.  If $g=f\frac{x_{j'}}{x_{j}}$, assign the label $g$ to $Q-e_{j+1}-e_{j+2}-\dots - e_{j'}$.

    Our staircase diagram is then the cubical complex built by taking the shadow of all labeled vertices (that is, all points $A=(a_{1},\dots, a_{n})\in \mathbb{R}^{n}$ such that, for some labeled vertex $Q=(q_{1},\dots, q_{n})$, we have $a_{i}\leq q_{i}$ for all $i$).
\end{proof}

\begin{remark}
    The map from staircases to cyclic box resolutions is surjective but not injective, as illustrated by Figure \ref{f:twoStaircasesOneResolution}. 

    \begin{figure}[hbt]
    \centering

    \begin{tikzpicture}[scale=0.5]
        \tkzDefPoint(0,0){a^4}
        \tkzDefPoint(2,-1){a^3b}
        \tkzDefPoint(4,-2){a^2b^2}
        \tkzDefPoint(6,-3){ab^3}
        \tkzDefPoint(8,-4){b^4}
        
        \tkzDefPoint(2,1){a^3c}
        \tkzDefPoint(4,0){a^2bc}
        \tkzDefPoint(6,-1){ab^2c}
        \tkzDefPoint(8,-2){b^3c}
        
        \tkzDefPoint(4,2){a^2c^2}
        \tkzDefPoint(6,1){abc^2}
        \tkzDefPoint(8,0){b^2c^2}
        
        \tkzDefPoint(6,3){ac^3}
        \tkzDefPoint(8,2){bc^3}
        
        \tkzDefPoint(8,4){c^4}

        \tkzDefPoint(10,-3){1}
        \tkzDefPoint(12,-2){2}
        \tkzDefPoint(14,-1){3}
        \tkzDefPoint(16,0){4}
        \tkzDefPoint(0,2){5}
        \tkzDefPoint(10,-1){6}
        \tkzDefPoint(12,0){7}
        \tkzDefPoint(10,1){x}
        
        \tkzDefPoint(14,1){8}
        \tkzDefPoint(16,2){9}
        \tkzDefPoint(0,4){10}
        \tkzDefPoint(2,3){11}
        \tkzDefPoint(12,2){12}
        \tkzDefPoint(14,3){13}
        \tkzDefPoint(16,4){14}

        \tkzDefPoint(0,6){15}
        \tkzDefPoint(2,5){16}
        \tkzDefPoint(4,4){17}
        \tkzDefPoint(10,3){18}
        \tkzDefPoint(12,4){19}
        \tkzDefPoint(14,5){20}
        \tkzDefPoint(16,6){21}

        \tkzDefPoint(0,8){22}
        \tkzDefPoint(2,7){23}
        \tkzDefPoint(4,6){24}
        \tkzDefPoint(6,5){25}
        \tkzDefPoint(10,5){26}
        \tkzDefPoint(12,6){27}
        \tkzDefPoint(14,7){28}
        \tkzDefPoint(16,8){29}

        \tkzDefPoint(2,9){30}
        \tkzDefPoint(4,8){31}
        \tkzDefPoint(6,7){32}
        \tkzDefPoint(8,6){33}
        \tkzDefPoint(10,7){34}
        \tkzDefPoint(12,8){35}
        \tkzDefPoint(14,9){36}

        \tkzDefPoint(4,10){37}
        \tkzDefPoint(6,9){38}
        \tkzDefPoint(8,8){39}
        \tkzDefPoint(10,9){40}
        \tkzDefPoint(12,10){41}

        \tkzDefPoint(6,11){42}
        \tkzDefPoint(8,10){43}
        \tkzDefPoint(10,11){44}

        \tkzDefPoint(8,12){45}

        \fill[cyan] (22) -- (45) -- (29) -- (4) -- (b^4) -- (a^4) -- (22);

        \fill[light-light-gray] (17) -- (16) -- (10) -- (11)--(17);
        \fill[light-light-gray] (19) -- (20) -- (14) -- (13)--(19);
        \fill[light-light-gray] (23) -- (43) -- (28) -- (c^4)--(23);
        \fill[light-light-gray] (8) -- (12) -- (b^2c^2) -- (a^2c^2)--(a^3c) -- (b^3c) -- (8);

        \draw[thick] (22) -- (45) -- (29) -- (4) -- (b^4) -- (a^4) -- (22);

        \draw[thick] (17) -- (16) -- (10) -- (11)--(17);
        \draw[thick] (19) -- (20) -- (14) -- (13)--(19);
        \draw[thick] (23) -- (43) -- (28) -- (c^4)--(23);
        \draw[thick] (8) -- (12) -- (b^2c^2) -- (a^2c^2)--(a^3c) -- (b^3c) -- (8);

        \draw[thick] (38) -- (27);
        \draw[thick] (31) -- (26);
        \draw[thick] (24) -- (40);
        \draw[thick] (25) -- (35);
        \draw[thick] (a^2bc) -- (abc^2);
        \draw[thick] (ab^2c) -- (b^2c^2) -- (6);
        \draw[thick] (x) -- (7);

        \draw[thick] (30) -- (16);
        \draw[thick] (11) -- (a^3b);
        \draw[thick] (31) -- (37);
        \draw[thick] (24) -- (a^2c^2);
        \draw[thick] (a^2bc) -- (a^2b^2);
        \draw[thick] (42) -- (38);
        \draw[thick] (25) -- (abc^2);
        \draw[thick] (ab^2c) -- (ab^3);
        \draw[thick] (45) -- (43);
        \draw[thick] (c^4) -- (b^2c^2);
        \draw[thick] (b^3c) -- (b^4);
        \draw[thick] (44) -- (40);
        \draw[thick] (26) -- (x);
        \draw[thick] (6) -- (1);
        \draw[thick] (41) -- (35);
        \draw[thick] (27) -- (12);
        \draw[thick] (7) -- (2);
        \draw[thick] (36) -- (20);
        \draw[thick] (13) -- (3);

        \draw[thick] (15) -- (23);
        \draw[thick] (5) -- (a^3c);
        \draw[thick] (17) -- (bc^3)--(19);
        \draw[thick] (28) -- (21);
        \draw[thick] (8)--(9);

        \draw[thick,red] (a^4) -- (b^4)--(c^4)--(a^4);

        \tkzDrawPoints[size=3.5](a^4,a^3b,a^2b^2,ab^3,b^4,a^3c,a^2bc,ab^2c,b^3c,a^2c^2,abc^2,b^2c^2,ac^3,bc^3,c^4,1,2,3,4,5,6,7,8,9,10,11,12,13,14,15,16,17,18,19,20,21,22,23,24,25,26,27,28,29,30,31,32,33,34,35,36,37,38,39,40,41,42,43,44,45,x) 

        \tkzLabelPoints[below](a^4,a^3b,a^2b^2,ab^3,b^4)

        \node[left] at (2,2)   {$a^3c$}; 
        \node[above] at (4,0.1)   {$a^2bc$};
        \node[above] at (6,-0.9)   {$ab^2c$}; 
        \node[right] at (8,-2.1)   {$b^3c$};  

        \node[left] at (4,2.15)   {$a^2c^2$}; 
        \node[right] at (6,1.2)   {$abc^2$}; 
        \node[right] at (8.3,0)   {$b^2c^2$}; 

        \node[above] at (6,3)   {$ac^3$};
        \node[right] at (8,1.9)   {$bc^3$}; 

        \node[right] at (8.3,4.1)   {$c^4$};   
        \end{tikzpicture} 
    
    \caption{This staircase diagram is different from the one given in Figure \ref{f:staircase-example}, because the face containing $c^{4}$ is perpendicular to the $x_{1}$-axis instead of the $x_{3}$-axis. Nevertheless, the two diagrams support the same resolution for $I=(a,b,c)^4$.}
    \label{f:twoStaircasesOneResolution}
\end{figure}
\end{remark}

\begin{definition}
    Fix a labeled staircase diagram $\mathcal{S}$ and a monomial $m$.  We say that $\mathcal{S}$ has an \emph{exterior corner} at $m$ if there exists a vertex $v$ on the visible surface of $\mathcal{S}$ such that:
    \begin{itemize}
        \item $v$ is labeled with $m$.
        \item For all $i$, $v-e_{i}$ is on the visible surface of $\mathcal{S}$.
    \end{itemize}
\end{definition}

\begin{proposition}\label{p:boxPlusIfExterior}
    Fix a monomial ideal $I$, equigenerated in degree $d$.  Then $I$ has a minimal box-plus resolution if and only if there exists a labeled staircase diagram with exterior corners at the degree-$d$ monomials not contained in $I$. 
\end{proposition}
    \begin{proof}
        First, suppose that $\widehat{I}$ admits a box-plus resolution, supported on the polytopal complex $\widehat{X}$.  Then the facets of $\widehat{X}$ are either boxes of the form $\Gamma_{i}(m)$ or copies of $\widehat{Y}_{m}$, centered at various omitted monomials $m$.  Replacing each $Y_{m}$ with the collection of boxes $\Gamma_{i}(m)$ (for all $i$) yields a box-plus resolution whose faces include all $\Gamma_{i}(m)$.  Applying Proposition \ref{p:allBoxAreStaircase} gives us a staircase diagram in which all $m$ are exterior corners, as desired.

        In the other direction, suppose $\mathcal{S}$ is a staircase diagram with exterior corners labeled by various monomials $m$.  Applying Proposition \ref{p:staircasesMakeBoxes} yields a cyclic box resolution in which $\Gamma_{i}(m)$ is a facet for all $m$ and all $i$.  For each $m$, we may replace the collection of $\Gamma_{i}(m)$ with a copy of $\widehat{Y}(m)$; the result is a box-plus resolution.
        \end{proof}

Proposition \ref{p:boxPlusIfExterior} provides a new perspective on Question \ref{q:BoxResolutions}, allowing us to make several statements about which 
``dense'' ideals admit box-plus resolutions.  We begin with the deletion of exactly two monomials.

\begin{lemma}\label{l:notComparable}
    Let $\mathbf{v}$ and $\mathbf{v}'$ be two distinct lattice points in $\mathbb{R}^{n}$.  Then there exists a staircase diagram with exterior corners at $\mathbf{v}$ and $\mathbf{v}'$ if and only if $\mathbf{v}$ and $\mathbf{v}'$ are incomparable. 
\end{lemma}
\begin{proof}
    Take $\mathcal{S}$ equal to the collection of all points $\mathbf{w}$ such that $\mathbf{w}\preceq \mathbf{v}$ or $\mathbf{w}\preceq \mathbf{v}'$.  (Here, $\mathbf{w}\preceq \mathbf{v}$ means that every entry of $\mathbf{w}$ is less than or equal to the corresponding entry of $\mathbf{v}$.)
\end{proof}

\begin{corollary}\label{c:notComparable}
    Let $m=\prod x_{i}^{a_{i}}$ and $m'=\prod x_{i}^{b_{i}}$ be two monomials of degree $d$, and let $\widehat{I}$ be the ideal generated by all degree-$d$ monomials except $m$ and $m'$.  The following are equivalent:
    \begin{enumerate}
        \item $\widehat{I}$ has a minimal box-plus resolution.
        \item For some (not necessarily distinct) $i$ and $j$, the cyclic sum $\sum_{t=i}^{t=j}(a_{t}-b_{t})$ is greater than or equal to $2$.
        \item For some (not necessarily distinct) $i$ and $j$, the cyclic sum $\sum_{t=i}^{t=j}(a_{t}-b_{t})$ is less than or equal to $-2$.
        \item The vector $(a_{1}-b_{1},a_{2}-b_{2},\dots,a_{n}-b_{n})$ either contains an entry not in $\{0,\pm 1\}$ or contains two nonzero entries of the same sign separated only by zeroes.  (That is, if we were to remove all the zeroes from this vector, there should be either a $2$ or two consecutive entries with the same sign.) 
        \item  $m$ is not contained in $\Gamma_{i}(m')$ for any $i$.
        \item $m'$ is not contained in $\Gamma_{i}(m)$ for any $i$.
    \end{enumerate}
\end{corollary}

\begin{proof}
    First observe that (2) and (3) are equivalent since $\sum_{t=1}^{n}(a_{t}-b_{t})=\deg(m)-\deg(m')=0$.  
    Observe also that (2) and (3) are equivalent to (4):  Given (2), choose $i$ and $j$ with $j-i$ minimal.  If $j=i$, we have $a_{i}-b_{i}\geq 2$.  Otherwise, we have $a_{i}-b_{i}=a_{j}-b_{j}=1$ and $a_{t}-b_{t}=0$ whenever $i<t<j$.  Conversely, given (4), if $a_{i}-b_{i}\geq 2$, we may take $j=i$; otherwise, if $a_{i}-b_{i}=a_{j}-b_{j}=1$ and $a_{t}-b_{t}=0$ for all $i<t<j$ we have (2) immediately.
    
    Next, observe that (5) and (6) are equivalent since $m'\in\Gamma_{i}(m)$ means that $m'=m\prod_{t\in \sigma}\frac{x_{t-1}}{x_{t}}$ for some nontrivial proper $\sigma\subset\{1,\dots,n\}$, which is equivalent to $m=m'\prod_{t\not\in\sigma}\frac{x_{t-1}}{x_{t}}$, i.e., $m\in \Gamma_{j}(m')$ for some $j$.  
    
    To see that (1) implies (5), it suffices to observe that, if (5) fails, the two copies of $\widehat{J}$ centered around $m$ and $m'$ must contain $m'$ and $m$ as vertices, respectively.

    The remaining implications are more technical.

    (5) implies (2),(3),(4):  Suppose (2) fails, so for every $i$ and $j$ we have $\sum_{t=i}^{t=j}(a_{t}-b_{t})\leq 1$.  Without loss of generality, assume $a_{1}-b_{1}=1$.  For all $j$, set $c_{i}=\sum_{t=1}^{t=j}(a_{t}-b_{t})$.  Observe that $c_{j}\in \{0,1\}$ for all $j$:  If $c_{j}\geq 2$, then (2) would hold, and if $c_{j}<0$, then $\sum_{t=j+1}^{n}(a_{t}-b_{t})>0$, so the cyclic sum $\sum_{t=j+1}^{1}(a_{t}-b_{t})$ is greater than $1$, contradicting the failure of (2).  Now observe that $m=m'\prod_{c_{j}=1}\frac{x_{j}}{x_{j+1}}$.  It follows that $m\in \Gamma_{1}(m')$ (since $c_{n}=\deg(m)-\deg(m')=0$). 

    (2) implies (1):  Assume (2) holds, and label the lattice points of $\mathbb{R}^{n}$ according to construction \ref{c:staircaselabel} so that the origin (and all points of the form $\mathbf{w}_{s}=s(1,1,\dots, 1)$) are labeled by $m$.  It follows that the points 
    \begin{align*}
        \mathbf{v}_{s}=\left(\sum_{t=1}^{n-1}(b_{t}-a_{t}),\sum_{t=2}^{n-1}(b_{t}-a_{t}),\dots, \sum_{t=n-1}^{n-1}(b_{t}-a_{t}),0\right) + s(1,1,\dots, 1)
    \end{align*}
     are all labeled by $m'$.

    Since (2) holds, without loss of generality we have $\sum_{t=n}^{t=j}(a_{t}-b_{t})\geq 2$ for some $j$.  Thus the entry in the $(j+1)^{\text{th}}$ spot of $\mathbf{v}_{0}$ is $\sum_{t=j+1}^{n-1}(a_{t}-b_{t})\leq -2$.  It follows that $\mathbf{v}_{0}$ and $\mathbf{w}_{-1}$ are incomparable (since $\mathbf{v}_{0}$ is smaller in the $(j-1)^{\text{th}}$ position and larger in the $n^{\text{th}}$ position).  In particular, there exists a staircase diagram with exterior corners at $\mathbf{v}_{0}$ and $\mathbf{w}_{-1}$ by Lemma \ref{l:notComparable}; this diagram supports a minimal box-plus resolution of $\widehat{I}$ by Proposition \ref{p:boxPlusIfExterior}.
\end{proof}

\begin{remark}\label{r:badTriangle}
    Lemma \ref{c:notComparable} holds for two monomials but not three.  Consider the monomial ideal $I$ of Example \ref{e:not-box-plus}, which is generated by all degree five monomials except $a^{3}bc$, $ab^{3}c$, and $abc^{3}$. The omitted generators satisfy the conditions of Lemma \ref{c:notComparable} pairwise, but the ideal does not have a box-plus resolution.  The problem is that, if we attempt to build the necessary staircase diagram, we end up with a triangle connecting the monomials $a^{2}b^{2}c$, $a^{2}bc^{2}$, and $ab^{2}c^{2}$.  Since the triangle is an odd cycle, attempting to build a staircase diagram for $I$ by assigning lattice points as in Remark \ref{r:staircaseRecursion} will yield an Escher-style infinite staircase around this triangle, forcing all choices of $\mathbf{v}+s(1,1,1)$ for the corresponding lattice points.
\end{remark}

\begin{question}
    Are forced odd cycles like in Remark \ref{r:badTriangle} the only obstruction to box-plus resolutions?  Or are there more interesting ways for things to go wrong, especially with more than three variables?
\end{question}

The best positive answer we have found to Question \ref{q:DeleteMore} is that a box-plus resolution is always possible when the deleted monomials are not adjacent and have distinct exponents on one variable.

\begin{proposition}\label{p:betterseparatedSequence}
Suppose $m_{1}=\prod x_{i}^{a_{1,i}}, \dots, m_{s}=\prod x_{i}^{a_{s,i}}$ are monomials of degree $d$ with distinct exponents on $x_{n}$, ordered so that $a_{1,n}<a_{2,n}<\dots<a_{s,n}$.  Suppose further that all pairs $(m_{i},m_{j})$ satisfy the conditions of Corollary \ref{c:notComparable}.  Set $\widehat{I}$ equal to the ideal generated by all degree $d$ monomials except $m_{1},\dots, m_{s}$.  Then $\widehat{I}$ has a minimal box-plus resolution.    
\end{proposition}
\begin{proof}
     We will construct a labeled staircase diagram that has exterior corners labeled with each of the $m_{i}$.

    Label the origin with $x_{1}^{d}$. Given any monomial $\mu$, recall from Remark \ref{r:labeledLine} that the collection of vertices labeled by $\mu$ forms a line with direction vector $(1,1,\dots,1)$.  In particular, given $\mu$ and an integer $t$, there is a unique vertex  
    $\mathbf{v}_{t,\mu}=(c_{t,1},c_{t,2},\dots, c_{t,n-1},t)$ with last entry $t$ and label $\mu$.

    For each monomial $\mu=x_{1}^{b_{1}}\dots x_{n}^{b_{n}}$, define a statistic $\phi(\mu)$ to be the maximal length of a chain 
    \[
    \mu\geq_{Q_{n}}m_{i_{1}}>_{Q_{n}}m_{i_{2}}>_{Q_{n}}\dots >_{Q_{n}}m_{i_{\phi(\mu)}},
    \]
    where $\geq_{Q_{n}}$ is the Borel partial order with respect to the order $x_{n}<x_{1}<x_{2}<\dots <x_{n-1}$; that is, $m_{i}\geq _{Q_{n}}m_{j}$ if the cyclic inequalities $\sum_{p=n}^{r}a_{i,p}\geq \sum_{p=n}^{r}a_{j,p}$ hold for all $i$.

    Let $\mathcal{S}$ be the staircase diagram consisting of all $\mathbf{v}_{\phi(\mu),\mu}$ for monomials $\mu$ of degree $d$ (which includes both the $m_{i}$ and the generators of $\widehat{I}$, and all lattice points that are below any of the $\mathbf{v}_{\phi(\mu),\mu}$.  We claim that $\mathcal{S}$ has exterior corners at $\mathbf{v}_{\phi(m_{i}),m_{i}}$.

    \begin{figure}[hbt]
    \centering

    \begin{tikzpicture}[scale=0.5]
        \tkzDefPoint(0,0){a^6}
        \tkzDefPoint(2,-1){a^5b}
        \tkzDefPoint(4,-2){a^4b^2}
        \tkzDefPoint(6,-3){a^3b^3}
        \tkzDefPoint(8,-4){a^2b^4}
        \tkzDefPoint(10,-5){ab^5}
        \tkzDefPoint(12,-6){b^6}

        \tkzDefPoint(2,1){a^5c}
        \tkzDefPoint(4,0){a^4bc}
        \tkzDefPoint(6,-1){a^3b^2c}
        \tkzDefPoint(8,-2){a^2b^3c}
        \tkzDefPoint(10,-3){ab^4c}
        \tkzDefPoint(12,-4){b^5c}
        
        \tkzDefPoint(4,2){a^4c^2}
        \tkzDefPoint(6,1){a^3bc^2}
        \tkzDefPoint(8,0){a^2b^2c^2}
        \tkzDefPoint(10,-1){ab^3c^2}
        \tkzDefPoint(12,-2){b^4c^2}
        
        \tkzDefPoint(6,3){a^3c^3}
        \tkzDefPoint(8,2){a^2bc^3}
        \tkzDefPoint(10,1){ab^2c^3}
        \tkzDefPoint(12,0){b^3c^3}
        
        \tkzDefPoint(8,4){a^2c^4}
        \tkzDefPoint(10,3){abc^4}
        \tkzDefPoint(12,2){b^2c^4}

        \tkzDefPoint(10,5){ac^5}
        \tkzDefPoint(12,4){bc^5}

        \tkzDefPoint(12,6){c^6}

        \tkzDefPoint(14,-5){1}
        \tkzDefPoint(16,-4){2}
        \tkzDefPoint(18,-3){3}
        
        \tkzDefPoint(14,-3){4}
        \tkzDefPoint(16,-2){5}
        \tkzDefPoint(18,-1){6}

        \tkzDefPoint(14,-1){7}
        \tkzDefPoint(16,0){8}
        \tkzDefPoint(18,1){9}
        
        \tkzDefPoint(14,1){10}
        \tkzDefPoint(16,2){11}
        \tkzDefPoint(18,3){12}

        \tkzDefPoint(14,3){13}
        \tkzDefPoint(16,4){14}
        \tkzDefPoint(18,5){15}

        \tkzDefPoint(14,5){16}
        \tkzDefPoint(16,6){17}

        \tkzDefPoint(14,7){18}

        \tkzDefPoint(12,8){19}
        \tkzDefPoint(10,9){20}
        \tkzDefPoint(8,10){21}
        \tkzDefPoint(6,11){22}

        \tkzDefPoint(10,7){23}
        \tkzDefPoint(8,8){24}
        \tkzDefPoint(6,9){25}
        \tkzDefPoint(4,10){26}

        \tkzDefPoint(8,6){27}
        \tkzDefPoint(6,7){28}
        \tkzDefPoint(4,8){29}
        \tkzDefPoint(2,9){30}

        \tkzDefPoint(6,5){31}
        \tkzDefPoint(4,6){32}
        \tkzDefPoint(2,7){33}
        \tkzDefPoint(0,8){34}

        \tkzDefPoint(4,4){35}
        \tkzDefPoint(2,5){36}
        \tkzDefPoint(0,6){37}

        \tkzDefPoint(2,3){38}
        \tkzDefPoint(0,4){39}

        \tkzDefPoint(0,2){40}

        \fill[cyan] (a^6) -- (b^6) -- (3)--(15)--(22) -- (34) -- (a^6);

        \fill[light-light-gray] (40) -- (38) -- (a^3bc^2) -- (a^4bc); 
        \fill[light-light-gray] (a^4b^2) -- (a^3b^2c) -- (b^5c) -- (5) -- (3) -- (b^6);
        \fill[light-light-gray] (36) -- (32) -- (abc^4) -- (bc^5) -- (13) -- (ab^2c^3) -- (36);
        \fill[light-light-gray] (29) -- (25) -- (c^6) -- (ac^5) -- (29);
        \fill[light-light-gray] (29) -- (25) -- (c^6) -- (ac^5) -- (29);
        \fill[light-light-gray] (b^4c^2) -- (8) -- (10) -- (ab^3c^2) -- (b^4c^2);

        \draw [thick] (a^6) -- (b^6) -- (3)--(15)--(22) -- (34) -- (a^6);
        \draw [thick] (40) -- (38) -- (a^3bc^2) -- (a^4bc)--(40);
        \draw [thick] (a^4b^2) -- (a^3b^2c) -- (b^5c) -- (5) -- (3);
        \draw [thick] (36) -- (32) -- (abc^4) -- (bc^5) -- (13) -- (ab^2c^3) -- (36);
        \draw [thick] (29) -- (25) -- (c^6) -- (ac^5) -- (29);
        \draw [thick] (b^4c^2) -- (8) -- (10) -- (ab^3c^2) -- (b^4c^2);

        \draw [thick] (a^5b)--(a^5c);
        \draw [thick] (30)--(38);
        \draw [thick] (26)--(32);
        \draw [thick] (35)--(a^4c^2);
        \draw [thick] (a^4b^2)--(a^4bc);
        \draw [thick] (22)--(25);
        \draw [thick] (28)--(31);
        \draw [thick] (a^3c^3)--(a^3b^2c);
        \draw [thick] (21)--(24);
        \draw [thick] (27)--(a^2c^4);
        \draw [thick] (a^2b^3c)--(a^2bc^3);
        \draw [thick] (20)--(23);
        \draw [thick] (ac^5)--(abc^4);
        \draw [thick] (ab^2c^3)--(ab^4c);
        \draw [thick] (19)--(bc^5);
        \draw [thick] (b^2c^4)--(b^3c^3);
        \draw [thick] (b^4c^2)--(b^5c);
        \draw [thick] (18)--(10);
        \draw [thick] (7)--(4);
        \draw [thick] (17)--(5);

        \draw [thick] (a^3bc^2)--(ab^3c^2);
        \draw [thick] (c^6)--(12);
        \draw [thick] (13)--(9);
        \draw [thick] (8)--(6);

        \draw [thick] (b^3c^3)--(7);
        \draw [thick] (a^3b^3)--(a^2b^3c);
        \draw [thick] (a^2b^4)--(ab^4c);
        \draw [thick] (ab^5)--(b^5c) -- (1);
        \draw [thick] (4)--(2);
        
        \draw [thick] (37)--(29);
        \draw [thick] (39)--(36);
        \draw [thick] (28)--(24);
        \draw [thick] (35)--(31);
        \draw [thick] (27)--(23);
        \draw [thick] (a^2bc^3)--(abc^4) -- (b^2c^4);

        \draw [thick] (a^6) -- (b^6) -- (c^6)--(a^6);

        \draw [thick,red] (a^6) -- (b^6) -- (c^6)--(a^6);

        \tkzDrawPoints[size=3.5](a^6,b^6,c^6,a^5b,a^4b^2,a^3b^3,a^2b^4,ab^5,a^5c,a^4bc,a^3b^2c,a^2b^3c,ab^4c,b^5c,a^4c^2,a^3bc^2,a^2b^2c^2,ab^3c^2,a^3c^3,a^2bc^3,ab^2c^3,b^3c^3,a^2c^4,abc^4,b^2c^4,ac^5,bc^5) 
        
        \tkzDrawPoints[red,size=3.5](a^4bc,ab^2c^3,ac^5,b^4c^2) 
        
        \tkzDrawPoints[size=3.5](1,2,3,4,5,6,7,8,9,10,11,12,13,14,15,16,17,18,19,20,21,22,23,24,25,26,27,28,29,30,31,32,33,34,35,36,37,38,39,40)

        \tkzLabelPoints[below](a^6,b^6)  
        \tkzLabelPoints[right](c^6,b^4c^2)
        \tkzLabelPoints[above](ac^5,ab^2c^3,a^4bc)

        \end{tikzpicture} 
    
    \caption{This complex supports a minimal resolution for $I=(a,b,c)^6$ with exterior corners at $a^4bc$, $b^{4}c^{2}$, 
    $ab^2c^3$, and $ac^5$. The staircase is constructed by assigning each monomial $\mu$ to the vertex $\mathbf{v}_{\phi(\mu),\mu}$, where $\phi(\mu)$ is defined as in the proof of Proposition \ref{p:betterseparatedSequence}.} 
        \label{f:construction-box-plus}
\end{figure}

It suffices to show that, for any nontrivial nonnegative $\mathbf{u}=\sum c_{j}e_{j}$, we have $\mathbf{v}_{\phi(m_{i}),m_{i}}+\mathbf{u}\not\in\mathcal{S}$, for which it is enough to show that $\mathbf{v}_{\phi(m_{i}),m_{i}}+\mathbf{u}\neq \mathbf{v}_{\phi(\mu),\mu}$ for any $\mu$.  Suppose to the contrary that, for some $\mu=\mu_{\mathbf{u}}$, we have $\mathbf{v}_{\phi(m_{i}),m_{i}}+\mathbf{u}= \mathbf{v}_{\phi(\mu_{\mathbf{u}}),\mu_{\mathbf{u}}}$.

    First observe that, since $\mathbf{v}_{\phi(m_{i}),m_{i}}+\mathbf{u}=\mathbf{v}_{\phi(\mu_{\mathbf{u}}),\mu_{\mathbf{u}}}$, we have $\phi(\mu_{\mathbf{u}})=\phi(m_{i})+c_{n}$ immediately.  If $c_{n}=0$, then $\phi(\mu_{\mathbf{u}})=\phi(m_{i})$, so there exists a chain
    \[
    \mu_{\mathbf{u}}\geq m_{j_{1}}>_{Q_{n}}m_{j_{2}}>_{Q_{n}}\dots >_{Q_{n}} m_{j_{\phi(\mu_{\mathbf{u}})}}.
    \]
    On the other hand, $\mu_{\mathbf{u}}=m_{i}(\frac{x_{1}}{x_{n}})^{c_{1}}(\frac{x_{2}}{x_{1}})^{c_{2}}\dots(\frac{x_{n-1}}{x_{n-2}})^{c_{n-1}}$, so $m_{i}>_{Q_{n}}\mu$ and the chain
    \[
    m_{i}\geq_{Q_{n}} m_{i}>_{Q_{n}}m_{j_{1}}>_{Q_{n}}\dots >_{Q_{n}}m_{j_{\phi(\mu)}}
    \]
    has length $\phi(m_{i})+1$, a contradiction.

    If $c_{n}=1$, there exists a chain
    \[
    \mu_{\mathbf{u}}\geq_{Q_{n}} m_{j_{1}}>_{Q_{n}}m_{j_{2}}>_{Q_{n}}\dots >_{Q_{n}} m_{j_{\phi(\mu_{\mathbf{u}})}}.
    \]
    By assumption, the $x_{n}$-exponent on $m_{j_{2}}$ is strictly less than the $x_{n}$-exponent on $m_{j_{1}}$. Since $\mu_{\mathbf{u}}\geq_{Q_{n}}m_{j_{1}}>_{Q_{n}}m_{j_{2}}$, we conclude that $\mu_{\mathbf{u}-e_{n}}\geq_{Q_{n}}m_{j_{2}}$.  Since $c_{n}=1$, $m_{i}\geq_{Q_{n}}\mu_{\mathbf{u}-e_{n}}$, and we have the chain
    \[
    m_{i}\geq_{Q_{n}} m_{i}\geq_{Q_{n}}m_{j_{2}}>_{Q_{n}}\dots >_{Q_{n}} m_{j_{\phi(\mu_{\mathbf{u}})}}.
    \]
    Since  $\phi(m_{i})\neq \phi(\mu_{\mathbf{u}})$,the second inequality cannot be strict, i.e., $m_{j_{2}}=m_{i}$ and $\mathbf{u}=e_{n}$.  But then $m_{j_{1}}=m_{i}\frac{x_{n}}{x_{1}}$, violating the conditions of Corollary \ref{c:notComparable}.

    Finally, if $c_{n}>1$,  a similar argument shows that $m_{j_{c_{n}}}=m_{i}\frac{x_{n}}{x_{1}}$. 
    \end{proof}


\begin{bibdiv}
\begin{biblist}[\resetbiblist{BLSWZ}]

\bib{Bayer1996duality}{article}{
author={Bayer, Dave},
title={Monomial Ideals and Duality},
date={1996},
status={preprint},
eprint={https://www.math.columbia.edu/~bayer/papers/Duality_B96/Duality_B96.pdf},
label={Ba},
}


\bib{BPS}{article}{
   author={Bayer, Dave},
   author={Peeva, Irena},
   author={Sturmfels, Bernd},
   title={Monomial resolutions},
   journal={Math. Res. Lett.},
   volume={5},
   date={1998},
   number={1-2},
   pages={31--46},
   issn={1073-2780},
   review={\MR{1618363}},
   doi={10.4310/MRL.1998.v5.n1.a3},
}

\bib{BS}{article}{
   author={Bayer, Dave},
   author={Sturmfels, Bernd},
   title={Cellular resolutions of monomial modules},
   journal={J. Reine Angew. Math.},
   volume={502},
   date={1998},
   pages={123--140},
   issn={0075-4102},
   review={\MR{1647559}},
   doi={10.1515/crll.1998.083},
}


\bib{JenniferBiermann}{article}{
   author={Biermann, Jennifer},
   title={Cellular structure on the minimal resolution of the edge ideal of
   the complement of the $n$-cycle},
   journal={Comm. Algebra},
   volume={42},
   date={2014},
   number={8},
   pages={3665--3681},
   issn={0092-7872},
   review={\MR{3196068}},
   doi={10.1080/00927872.2013.768629},
   label={Bi},
}

\bib{BuchsbaumEisenbud}{article}{
   author={Buchsbaum, David A.},
   author={Eisenbud, David},
   title={Generic free resolutions and a family of generically perfect
   ideals},
   journal={Advances in Math.},
   volume={18},
   date={1975},
   number={3},
   pages={245--301},
   issn={0001-8708},
   review={\MR{0396528}},
   doi={10.1016/0001-8708(75)90046-8},
}

\bib{BR}{article}{
   author={Buchsbaum, David A.},
   author={Rim, Dock S.},
   title={A generalized Koszul complex. II. Depth and multiplicity},
   journal={Trans. Amer. Math. Soc.},
   volume={111},
   date={1964},
   pages={197--224},
   issn={0002-9947},
   review={\MR{0159860}},
   doi={10.2307/1994241},
}


\bib{CharalambousEvans}{article}{
   author={Charalambous, Hara},
   author={Evans, E. Graham, Jr.},
   title={Resolutions obtained by iterated mapping cones},
   journal={J. Algebra},
   volume={176},
   date={1995},
   number={3},
   pages={750--754},
   issn={0021-8693},
   review={\MR{1351361}},
   doi={10.1006/jabr.1995.1270},
}


\bib{Clark}{article}{
   author={Clark, Timothy B. P.},
   title={A minimal poset resolution of stable ideals},
   conference={
      title={Progress in commutative algebra 1},
   },
   book={
      publisher={de Gruyter, Berlin},
   },
   isbn={978-3-11-025034-3},
   date={2012},
   pages={143--166},
   review={\MR{2932584}},
   label={Cl},
}

\bib{ClarkMapes}{article}{
   author={Clark, Timothy B. P.},
   author={Mapes, Sonja},
   title={Rigid monomial ideals},
   journal={J. Commut. Algebra},
   volume={6},
   date={2014},
   number={1},
   pages={33--52},
   issn={1939-0807},
   review={\MR{3215560}},
   doi={10.1216/JCA-2014-6-1-33},
}

\bib{ClarkTchernev}{article}{
   author={Clark, Timothy B. P.},
   author={Tchernev, Alexandre B.},
   title={Minimal free resolutions of monomial ideals and of toric rings are
   supported on posets},
   journal={Trans. Amer. Math. Soc.},
   volume={371},
   date={2019},
   number={6},
   pages={3995--4027},
   issn={0002-9947},
   review={\MR{3917215}},
   doi={10.1090/tran/7614},
}

\bib{DE}{article}{
   author={Dao, Hailong},
   author={Eisenbud, David},
   title={Linearity of free resolutions of monomial ideals},
   journal={Res. Math. Sci.},
   volume={9},
   date={2022},
   number={2},
   pages={Paper No. 35, 15},
   issn={2522-0144},
   review={\MR{4431293}},
   doi={10.1007/s40687-022-00330-6},
}

\bib{DochtermannJoswigSanyal}{article}{
   author={Dochtermann, Anton},
   author={Joswig, Michael},
   author={Sanyal, Raman},
   title={Tropical types and associated cellular resolutions},
   journal={J. Algebra},
   volume={356},
   date={2012},
   pages={304--324},
   issn={0021-8693},
   review={\MR{2891135}},
   doi={10.1016/j.jalgebra.2011.12.028},
}


\bib{EagonMillerOrdog}{article}{
    author={Eagon, John},
    author={Millor, Ezra},
    author={Ordog, Erika},
    title={Minimal Resolutions of Monomial Ideals},
    eprint={arXiv:1906.08837},
    doi={10.48550/arXiv.1906.08837},
    status={preprint},
    date={2019}
}


\bib{EK}{article}{
   author={Eliahou, Shalom},
   author={Kervaire, Michel},
   title={Minimal resolutions of some monomial ideals},
   journal={J. Algebra},
   volume={129},
   date={1990},
   number={1},
   pages={1--25},
   issn={0021-8693},
   review={\MR{1037391}},
   doi={10.1016/0021-8693(90)90237-I},
}

\bib{FranciscoMerminSchweigBorelgenerators}{article}{
   author={Francisco, Christopher A.},
   author={Mermin, Jeffrey},
   author={Schweig, Jay},
   title={Borel generators},
   journal={J. Algebra},
   volume={332},
   date={2011},
   pages={522--542},
   issn={0021-8693},
   review={\MR{2774702}},
   doi={10.1016/j.jalgebra.2010.09.042},
}


\bib{FMS2}{article}{
   author={Francisco, Christopher A.},
   author={Mermin, Jeffrey},
   author={Schweig, Jay},
   title={Generalizing the Borel property},
   journal={J. Lond. Math. Soc. (2)},
   volume={87},
   date={2013},
   number={3},
   pages={724--740},
   issn={0024-6107},
   review={\MR{3073673}},
   doi={10.1112/jlms/jds071},
}


\bib{GasharovPeevaWelker}{article}{
   author={Gasharov, Vesselin},
   author={Peeva, Irena},
   author={Welker, Volkmar},
   title={The lcm-lattice in monomial resolutions},
   journal={Math. Res. Lett.},
   volume={6},
   date={1999},
   number={5-6},
   pages={521--532},
   issn={1073-2780},
   review={\MR{1739211}},
   doi={10.4310/MRL.1999.v6.n5.a5},
}


\bib{Mapes}{article}{
   author={Mapes, Sonja},
   title={Finite atomic lattices and resolutions of monomial ideals},
   journal={J. Algebra},
   volume={379},
   date={2013},
   pages={259--276},
   issn={0021-8693},
   review={\MR{3019256}},
   doi={10.1016/j.jalgebra.2013.01.005},
   label={Ma}
}

\bib{mermin2010cellular}{article}{
   author={Mermin, Jeffrey},
   title={The Eliahou-Kervaire resolution is cellular},
   journal={J. Commut. Algebra},
   volume={2},
   date={2010},
   number={1},
   pages={55--78},
   issn={1939-0807},
   review={\MR{2607101}},
   doi={10.1216/JCA-2010-2-1-55},
   label={Me1},
}

\bib{mermin2011simplicial}{article}{
   author={Mermin, Jeff},
   title={Three simplicial resolutions},
   conference={
      title={Progress in commutative algebra 1},
   },
   book={
      publisher={de Gruyter, Berlin},
   },
   isbn={978-3-11-025034-3},
   date={2012},
   pages={127--141},
   review={\MR{2932583}},
   label={Me2}
}


\bib{Miller}{article}{
   author={Miller, Ezra},
   title={Planar graphs as minimal resolutions of trivariate monomial
   ideals},
   journal={Doc. Math.},
   volume={7},
   date={2002},
   pages={43--90},
   issn={1431-0635},
   review={\MR{1911210}},
   label={Mi},
}

\bib{miller2004combinatorial}{book}{
   author={Miller, Ezra},
   author={Sturmfels, Bernd},
   title={Combinatorial commutative algebra},
   series={Graduate Texts in Mathematics},
   volume={227},
   publisher={Springer-Verlag, New York},
   date={2005},
   pages={xiv+417},
   isbn={0-387-22356-8},
   review={\MR{2110098}},
}



\bib{NagelReiner}{article}{
   author={Nagel, Uwe},
   author={Reiner, Victor},
   title={Betti numbers of monomial ideals and shifted skew shapes},
   journal={Electron. J. Combin.},
   volume={16},
   date={2009},
   number={2},
   pages={Research Paper 3, 59},
   review={\MR{2515766}},
   doi={10.37236/69},
}

\bib{Novik}{article}{
   author={Novik, Isabella},
   title={Lyubeznik's resolution and rooted complexes},
   journal={J. Algebraic Combin.},
   volume={16},
   date={2002},
   number={1},
   pages={97--101},
   issn={0925-9899},
   review={\MR{1941987}},
   doi={10.1023/A:1020838732281},
   label={No},
}

\bib{Peeva}{book}{
   author={Peeva, Irena},
   title={Graded syzygies},
   series={Algebra and Applications},
   volume={14},
   publisher={Springer-Verlag London, Ltd., London},
   date={2011},
   pages={xii+302},
   isbn={978-0-85729-176-9},
   review={\MR{2560561}},
   doi={10.1007/978-0-85729-177-6},
   label={Pe},
}



\bib{PS}{article}{
   author={Peeva, Irena},
   author={Stillman, Mike},
   title={The minimal free resolution of a Borel ideal},
   journal={Expo. Math.},
   volume={26},
   date={2008},
   number={3},
   pages={237--247},
   issn={0723-0869},
   review={\MR{2437094}},
   doi={10.1016/j.exmath.2007.10.003},
}


\bib{Sinefakopoulos2007}{article}{
   author={Sinefakopoulos, Achilleas},
   title={On Borel fixed ideals generated in one degree},
   journal={J. Algebra},
   volume={319},
   date={2008},
   number={7},
   pages={2739--2760},
   issn={0021-8693},
   review={\MR{2397405}},
   doi={10.1016/j.jalgebra.2008.01.017},
   label={Si},
}

\bib{Tchernev}{article}{
    author={Tchernev, Alexandre},
    title={Dynamical Systems on chain complexes and canonical minimal resolutions},
    eprint={arXiv:1909.08577},
    doi={10.48550/arXiv.1906.08837},
    status={preprint},
    date={2019},
    label={Tc},
}

\bib{TchernevVarisco}{article}{
   author={Tchernev, Alexandre},
   author={Varisco, Marco},
   title={Modules over categories and Betti posets of monomial ideals},
   journal={Proc. Amer. Math. Soc.},
   volume={143},
   date={2015},
   number={12},
   pages={5113--5128},
   issn={0002-9939},
   review={\MR{3411130}},
   doi={10.1090/proc/12643},
}



\end{biblist}
\end{bibdiv}
\end{document}